\documentclass[11pt]{article}%
\usepackage[margin=1in,a4paper]{geometry}
\usepackage{commath}
\usepackage{amsfonts}
\usepackage{amsthm}
\usepackage{amsmath}
\usepackage{amssymb}
\usepackage{enumerate}
\usepackage[longnamesfirst]{natbib}
\usepackage[bookmarks,colorlinks = true,linkcolor = blue,urlcolor  = blue,citecolor =
blue,anchorcolor = blue]{hyperref}
\usepackage{booktabs}
\usepackage{xr}
\usepackage{graphicx}
\usepackage{multirow}
\usepackage{float}
\usepackage{color}%
\providecommand{\U}[1]{\protect\rule{.1in}{.1in}}
\newtheorem{lemma}{Lemma}[section]
\newtheorem{example}{Example}[section]

\newtheorem{assumption}{Assumption}[section]

\newtheorem{definition}{Definition}[section]
\newtheorem{thm}{Theorem}[section]
\newtheorem{corollary}{Corollary}[section]
\numberwithin{equation}{section}
\DeclareMathOperator*{\ve}{vec}
\DeclareMathOperator*{\vech}{vech}
\DeclareMathOperator*{\tr}{tr}

\begin{document}

\title{Estimation of a Multiplicative Correlation Structure in the Large Dimensional
Case}
\author{Christian M. Hafner\thanks{Institut de statistique, biostatistique et sciences
actuarielles, and CORE, Universit\'e catholique de Louvain, Louvain-la-Neuve,
Belgium. Email: \texttt{christian.hafner@uclouvain.be}.}\\Universit\'{e}{\small \ }catholique de Louvain
\and Oliver B. Linton\thanks{Faculty of Economics, Austin Robinson Building,
Sidgwick Avenue, Cambridge, CB3 9DD. Email: \texttt{obl20@cam.ac.uk}.}\\University of Cambridge
\and Haihan Tang\thanks{Corresponding author. Fanhai International School of
Finance and School of Economics, Fudan University, 220 Handan Road, Yangpu District, Shanghai, 200433, China. Email:
\texttt{hhtang@fudan.edu.cn}.}\\Fudan University}
\date{\today}
\maketitle

\begin{abstract}
\noindent We propose a Kronecker product model for correlation or covariance matrices in
the large dimensional case. The number of parameters of the model increases
logarithmically with the dimension of the matrix. We propose a minimum
distance (MD) estimator based on a log-linear property of the model, as well
as a one-step estimator, which is a one-step approximation to the
quasi-maximum likelihood estimator (QMLE). We establish rates of convergence
and central limit theorems (CLT) for our estimators in the large dimensional
case. A specification test and tools for Kronecker product model selection and
inference are provided. In a Monte Carlo study where a Kronecker product model
is correctly specified, our estimators exhibit superior performance. In an
empirical application to portfolio choice for S\&P500 daily returns, we
demonstrate that certain Kronecker product models are good approximations to
the general covariance matrix.

\begin{description}
\item[Keywords:] Correlation matrix, Kronecker product, matrix logarithm; multiway array data; portfolio choice.

\item[JEL classification] C55; C58; G11.

\end{description}

\end{abstract}

%\textit{Some} \textit{key words: }Correlation matrix; Kronecker product;
%Matrix logarithm; Multiway array data; Portfolio choice; Sparsity

%\textit{JEL subject classification}: C55, C58, G11
%\textit{AMS 2000 subject classification}: 62F12%
%TCIMACRO{\TeXButton{page number 0}{\setcounter{page}{0}}}%
%BeginExpansion
\setcounter{page}{0}%
%EndExpansion%
%TCIMACRO{\TeXButton{no page number}{\thispagestyle{empty}}}%
%BeginExpansion
\thispagestyle{empty}%
%EndExpansion%
%TCIMACRO{\TeXButton{TeX}{\newpage}}%
%BeginExpansion
\newpage
%EndExpansion

\section{Introduction}

Covariance and correlation matrices are of great importance in many fields. In
finance, they are a key element in portfolio choice and risk management. In
psychology, scholars have long assumed that some observed variables are
related to the key unobserved traits through a factor model, and then use the
covariance matrix of the observed variables to deduce properties of the latent
traits. \cite{anderson1984} is a classic statistical reference that studies
the estimation of covariance matrices and hypotheses testing about them in the
low dimensional case (i.e., the dimension of the covariance matrix, $n$, is
small compared with the sample size $T$).

More recent work has considered the case where $n$ is large along with $T$.
This is because many datasets now used are large. For instance, as finance
theory suggests that one should choose a well-diversified portfolio that
perforce includes a large number of assets with non-zero weights, investors
now consider many securities when forming a portfolio. The listed company
Knight Capital Group claims to make markets in thousands of securities
worldwide, and is constantly updating its inventories/portfolio weights to
optimize its positions. If $n/T$ is not negligible when compared to zero but
still less than one, we call this the \textit{large dimensional} case in this
article. (We reserve the phrase "the high dimensional case" for $n>T$.) The
correct theoretical framework to study the large dimensional case is to use
the joint asymptotics (i.e., both $n$ and $T$ diverge to infinity
simultaneously albeit subject to some restriction on their relative growth
rate), not the usual asymptotics (i.e., $n$ fixed, $T$ tends to infinity
alone). Standard statistical methods such as principal component analysis
(PCA) and canonical-correlation analysis (CCA), do not directly generalize to
the large dimensional case; applications to, say, portfolio choice, face
considerable difficulties (see \cite{wangfan2016}).

There are many new methodological approaches for the large dimensional case,
for example \cite{ledoitwolf2003}, \cite{bickellevina2008}, \cite{onatski2009}%
, \cite{fanfanlv2008}, \cite{ledoitwolf2012} \cite{fanliaomincheva2013}, and
\cite{ledoitwolf2015}. \cite{yaozhengbai2015} gave an excellent account of the
recent developments in the theory and practice of estimating large dimensional
covariance matrices. Generally speaking, the approach is either to impose some
sparsity on the covariance matrix, meaning that many elements of the
covariance matrix are assumed to be zero or small, thereby reducing the number
of parameters to be estimated, or to use some device, such as shrinkage or a
factor model, to reduce dimension. Most of this literature assumes i.i.d. data.

We consider a parametric model for the covariance or correlation matrix - the
Kronecker product model. For a real symmetric, positive definite $n\times n$
matrix $\Delta$, a \textit{Kronecker product model} is a family of $n\times n$
matrices $\{\Delta^{\ast}\}$, each of which has the following structure:
\begin{equation}
\Delta^{\ast}=\Delta_{1}^{\ast}\otimes\Delta_{2}^{\ast}\otimes\cdots
\otimes\Delta_{v}^{\ast}, \label{eqn model}%
\end{equation}
where $\Delta_{j}^{\ast}$ is an $n_{j}\times n_{j}$ dimensional real
symmetric, positive definite \textit{sub-matrix} such that $n=n_{1}%
\times\cdots\times n_{v}$. We require that $n_{j}\in\mathbb{Z}$ and $n_{j}%
\geq2$ for all $j$; the $\{n_{j}\}_{j=1}^{v}$ need not be distinct. We suppose
that $\Delta$ is the covariance or correlation matrix of an observable series
with sample size $T$ and $\{\Delta^{\ast}\}$ is a model for $\Delta$.

We study the Kronecker product model in the large dimensional case. Since $n$
tends to infinity in the joint asymptotics, there are two main cases: (1)
$n_{j}\rightarrow\infty$ for $j=1,\ldots,v$ and $v$ is fixed; (2)
$\{n_{j}\}_{j=1}^{v}$ are all fixed and $v\rightarrow\infty.$ We shall study
case (2) in detail because of its dimensionality reduction property. In this
case, the number of parameters of a Kronecker product model grows
\textit{logarithmically} with $n$. In particular, we show that a Kronecker
product model {induces a type of sparsity} on the covariance or correlation
matrix: The logarithm of a Kronecker product model has many zero elements, so
that sparsity is explicitly imposed on the logarithm of the covariance or
correlation matrix - we call this \textit{log sparsity}.
%What is new of our Kronecker product model compared to our predecessors. It induces log sparsity.

The Kronecker product model has a number of intrinsic advantages for
applications. The eigenvalues of a Kronecker product are products of the
eigenvalues of its sub-matrices, which in the simplest case are obtainable in
closed form.
%Compared with strict factor models whose eigenvalues have a spikedness property (\cite{johnstoneonatskiy2018}), the Kronecker product model has more flexibility in the large dimensional case.
In the large dimensional case the eigenvalue distribution can be quite
general, and there is no spikedness property as in strict factor models
(\cite{johnstoneonatskiy2018}). The inverse covariance matrix, its
determinant, and other key quantities are easily obtained from the
corresponding quantities of the sub-matrices, which facilitates computation
and analysis.

We primarily focus on correlation matrices rather than covariance matrices.
This is partly because the asymptotic theory for the correlation matrix model
nests that for the covariance matrix model, and partly because this will allow
us to adopt a more flexible approach to approximating a general covariance
matrix: We can allow the diagonal elements of the covariance matrix to be
unrestricted (and they can be estimated by other well-understood methods). In
practice, fitting a correlation matrix with a Kronecker product model tends to
perform better than doing so for its corresponding covariance matrix. To avoid
confusion, we would like to remark that if a Kronecker product model is
correctly specified for a correlation matrix, its corresponding covariance
matrix need not have a Kronecker product structure, and vice versa. In other
words, log sparsity on a correlation matrix does not necessarily imply that
its corresponding covariance matrix has log sparsity, and vice versa.

We show that the logarithm of a Kronecker product model is linear in its
unknown parameters. We use this as a basis to propose a minimum distance (MD)
estimator that is in closed form. We establish a crude upper bound rate of
convergence for the MD estimator under the joint asymptotics, but we
anticipate that this bound could be improved with better technology and we
leave this for future research. There is a large literature on the optimal
rate of convergence for estimation of high-dimensional covariance matrices and
inverse (i.e., \textit{precision}) matrices (see \cite{caizhangzhou2010} and
\cite{caizhou2012}). \cite{cairenzhou2014} gave a nice review on those recent
results. However their optimal rates are not applicable to our setting because
here sparsity is not imposed on the covariance or correlation matrix, but on
its logarithm. In addition, we allow for weakly dependent data, whereas the
above cited papers all assume i.i.d. structures.

Next, we discuss a quasi-maximum likelihood estimator (QMLE) and a one-step
estimator, which is an approximate QMLE. Under the joint asymptotics, we
provide feasible central limit theorems (CLT) for the MD and one-step
estimators, the latter of which is shown to achieve the parametric efficiency
bound (Cramer-Rao lower bound) in the fixed $n$ case. When choosing the
weighting matrix optimally, we also show that the optimally-weighted MD and
one-step estimators have the same asymptotic distribution. These CLTs are of
independent interest and contribute to the literature on the large dimensional
CLTs (see \cite{huber1973}, \cite{yohaimaronna1979}, \cite{portnoy1985},
\cite{mammen1989}, \cite{welsh1989}, \cite{baiwu1994},
\cite{saikkonenlutkepohl1996} and \cite{heshao2000}). Last, we give a
specification test which allows us to test whether a Kronecker product model
is correctly specified.

We discuss in Section \ref{sec model} what kind of data gives rise to a
Kronecker product model. However, a given covariance or correlation matrix
might not exactly correspond to a Kronecker product; in which case a Kronecker
product model is misspecified, so $\Delta\notin\{\Delta^{\ast}\}$. The
previous literature on Kronecker product models did not touch this, but we
shall demonstrate in this article that a Kronecker product model is a very
good approximating device to general covariance or correlation matrices, by
trading off variance with bias. We show that for a given Kronecker product
model there always exists a member in it that is closest to the covariance or
correlation matrix in some sense to be made precise shortly.

%The extreme eigenvalues of the sample correlation matrix are known to behave poorly when the dimension of the matrix increases, but in our case because of the tight structure we impose we obtain consistency and a CLT under general conditions. We also apply our methods to the question of estimating the variance of the minimum variance portfolio formed using the Kronecker product correlation matrix.

%We establish the rate of convergence and asymptotic normality of the estimated parameters when both $n$ and $T$ diverge under restrictions on the relative rate of growth of these quantities. In particular, we show that $\Vert\hat{\theta}%_{T}-\theta^{0}\Vert_{2}=O_{p}((n\kappa(W)/T)^{1/2})$, which improves on the crude rate implied by the unrestricted correlation matrix estimator, $O_{p}((n^{2}/T)^{1/2}).$ Our QMLE procedure works much better numerically than the sample correlation matrix, consistent with the faster rate of convergence we expect.

We provide some simulation evidence that the Kronecker product model works
very well when it is correctly specified. In the empirical study, we apply the
Kronecker product model to S\&P500 daily stock returns and compare it with
\cite{ledoitwolf2004}'s linear shrinkage estimator as well as with
\cite{ledoitwolf2017}'s direct nonlinear shrinkage estimator. We find that the
minimum variance portfolio implied by a Kronecker product model is almost as
good as that constructed from \cite{ledoitwolf2004}'s linear shrinkage
estimator. In future work we aim to improve the practical performance of our
method by combining it with other approaches such as factor models and by
improving the estimation methodology. 

\subsection{Literature Review}

%\textbf{Literature Review}.
The Kronecker product model has been previously considered in the psychometric
literature (see \cite{campbelloconnell1967}, \cite{swain1975},
\cite{cudeck1988}, \cite{verheeswansbeek1990} etc). In a
multitrait-multimethod (MTMM) context, "multi-mode" data give rise to a
Kronecker product model naturally (we will further discuss this in Section
\ref{sec model}). \cite{verheeswansbeek1990} outlined several estimation
methods of the model based on the least squares and maximum likelihood
principles, and provided large sample variances under assumptions of
Gaussianity and fixed $n$. There is a growing Bayesian and Frequentist
literature on multiway array or tensor datasets, where a Kronecker product
model is commonly employed. See for example \cite{akdemirgupta2011},
\cite{allen2012}, \cite{brownemaccallum2002}, \cite{cohenusevichcomon2016},
\cite{constantinoukokoszkareimherr2015}, \cite{dobra2014},
\cite{fosdickhoff2014}, \cite{gerardhoff2015}, \cite{hoff2011},
\cite{hoff2015}, \cite{hoff2016}, \cite{krijnen2004}, \cite{leivaroy2014},
\cite{lengtang2012}, \cite{lizhang2016}, \cite{manceuradutilleul2013},
\cite{ningliu2013}, \cite{ohlsonahmadavonrosen2013}, \cite{singullahmad2012},
\cite{volovskyhoff2014}, \cite{volovskyhoff2015}, and \cite{yinli2012}. In
this literature, they also work with fixed $n$.

%However, in both these (apparently separate) literatures, (i) $n$ is fixed, (ii) the number $v$ of sub-matrices of a Kronecker product is fixed and typically small, and (iii) each $n_{j}$ is also fixed but perhaps of moderate size.
%Our alternative model for covariance matrix estimation, and its past.

In the spatial literature, there are a number of studies that consider a
Kronecker product structure for the correlation matrix of a random field, see
for example \cite{lohlam2000}.

This article is the first one studying Kronecker product models in the large
dimensional case. Our work is also among the first exploiting log sparsity;
the other is \cite{batteyfan2017}, although there are a few differences.
First, their log sparsity is an assumption from the onset, in a similar spirit
as \cite{bickellevina2008}, whereas our log sparsity is induced by a Kronecker
product model. Second, they work with covariance matrices while we shall focus
on correlation matrices. Even if we look at covariance matrices for the
purpose of comparison, the Kronecker product model imposes different sparsity
restrictions - as compared to those imposed by \cite{batteyfan2017} - on the
elements of the logarithm of the covariance matrix. Third and perhaps most
important, we look at different estimators.
%difference between ours and Battey and Fan

%For a financial application one might consider different often employed sorting characteristics such as industry, size, and value, by which each stock is labelled. For example, we may have 10 industries, 3 sizes and $3$ different value buckets, which yields 90 buckets. If one has precisely one firm in each industry $\varepsilon_{1}$, of each size $\varepsilon_{2}$ and of each value category $\varepsilon_{3}$ then the multi-array model is directly applicable.

\subsection{Roadmap}

The rest of the article is structured as follows. In Section \ref{sec model}
we lay out the Kronecker product model in detail. Section
\ref{sec MD estimation} introduces the MD estimator, gives its asymptotic
properties, and includes a specification test, while Section \ref{sec QMLE}
discusses the QMLE and one-step estimator, and provides the asymptotic
properties of the one-step estimator. Section \ref{sec modelselection}
examines the issue of model selection. Section \ref{sec simu} provides
numerical evidence for the model as well as an empirical application. Section
\ref{sec conclusion} concludes. Major proofs are to be found in Appendix; the
remaining proofs are put in Supplementary Material (SM in what follows).

\subsection{Notation}

Let $A$ be an $m\times n$ matrix. Let $\ve A$ denote the vector obtained by
stacking the columns of $A$ one underneath the other. The \textit{commutation
matrix} $K_{m,n}$ is an $mn\times mn$ \textit{orthogonal} matrix which
translates $\ve A$ to $\ve(A^{\intercal})$, i.e., $\ve(A^{\intercal}%
)=K_{m,n}\ve(A)$. If $A$ is a symmetric $n\times n$ matrix, its $n(n-1)/2$
supradiagonal elements are redundant in the sense that they can be deduced
from symmetry. If we eliminate these redundant elements from $\ve A$, we
obtain a new $n(n+1)/2\times1$ vector, denoted $\vech A$. They are related by
the full-column-rank, $n^{2}\times n(n+1)/2$ \textit{duplication matrix}
$D_{n}$: $\ve A=D_{n}\vech A$. Conversely, $\vech A=D_{n}^{+}\ve A$, where
$D_{n}^{+}$ is $n(n+1)/2\times n^{2}$ and the Moore-Penrose generalized
inverse of $D_{n}$. In particular, $D_{n}^{+}=(D_{n}^{\intercal}D_{n}%
)^{-1}D_{n}^{\intercal}$ because $D_{n}$ is full-column-rank. We use
$\mathrm{vecl}(A)$ to denote the vectorization operator of the lower
off-diagonal elements of $A$ (so this operator excludes the diagonal elements
unlike the related $\vech(\cdot)$ operator).

For $x\in\mathbb{R}^{n}$, let $\Vert x\Vert_{2}:=\sqrt{\sum_{i=1}^{n}x_{i}%
^{2}}$ and $\|x\|_{\infty}:=\max_{1\leq i\leq n}|x_{i}|$ denote the Euclidean
norm and the element-wise maximum norm, respectively.
%Notation $\text{diag}(x)$ gives an $n\times n$ diagonal matrix with the diagonal being the elements of $x$.
Let $\text{maxeval}(\cdot)$ and $\text{mineval}(\cdot) $ denote the maximum
and minimum eigenvalues of some real symmetric matrix, respectively. For any
real $m\times n$ matrix $A=(a_{i,j})_{1\leq i\leq m, 1\leq j\leq n}$, let
$\Vert A\Vert_{F}:=[\text{tr}(A^{\intercal}A)]^{1/2}\equiv[\text{tr}%
(AA^{\intercal})]^{ 1/2}\equiv\Vert\ve A\Vert_{2}$ and $\Vert A\Vert_{\ell
_{2}}:=\max_{\Vert x\Vert_{2}=1}\Vert Ax\Vert_{2}\equiv\sqrt{\text{maxeval}%
(A^{\intercal}A)}$
%, and $\|A\|_{\ell_{\infty}}:=\max_{1\leq i\leq m}\sum_{j=1}^{n}|a_{i,j}|$
denote the Frobenius norm and spectral norm ($\ell_{2}$ operator norm)
%and maximum row sum matrix norm ($\ell_{\infty}$ operator norm)
of $A$, respectively. Note that $\|\cdot\|_{\infty}$ can also be applied to
matrix $A$, i.e., $\|A\|_{\infty}=\max_{1\leq i\leq m,1\leq j\leq n}|a_{i,j}%
|$; however $\|\cdot\|_{\infty}$ is not a matrix norm so it does not have the
submultiplicative property of a matrix norm.

Consider two sequences of $n\times n$ real random matrices $X_{T}$ and $Y_{T}%
$. Notation $X_{T}=O_{p}(\|Y_{T}\|)$, where $\|\cdot\|$ is some matrix norm,
means that for every real $\varepsilon>0$, there exist $M_{\varepsilon}>0$,
$N_{\varepsilon}>0$ and $T_{\varepsilon}>0$ such that for all
$n>N_{\varepsilon}$ and $T>T_{\varepsilon}$, $\mathbb{P}(\|X_{T}%
\|/\|Y_{T}\|>M_{\varepsilon})<\varepsilon$. Notation $X_{T}=o_{p}(\|Y_{T}\|)$,
where $\|\cdot\|$ is some matrix norm, means that $\|X_{T}\|/\|Y_{T}%
\|\xrightarrow{p}0$ as $n,T\to\infty$ simultaneously. Landau notation in this
article, unless otherwise stated, should be interpreted in the sense that
$n,T\to\infty$ simultaneously.

Let $a \vee b$ and $a \wedge b$ denote $\max(a,b)$ and $\min(a,b)$,
respectively.
%For two real sequences $a_{n,T}$ and $b_{n,T}$, $a_{n,T}\lesssim b_{n,T}$ means that $a_{n,T}\leq Cb_{n,T}$ for some positive real number $C$ for all $n,T\geq1$.
%$a_{T}\sim b_{T}$ means that $a_{T}$ and $b_{T}$ are asymptotically equivalent, i.e., $a_{T}/b_{T}\to1$ as $T\to\infty$.
For $x\in\mathbb{R}$, let $\lfloor x\rfloor$ denote the greatest integer
\textit{strictly less} than $x$ and $\lceil x \rceil$ denote the smallest
integer greater than or equal to $x $. Notation $\sigma(\cdot)$ defines sigma algebra.

For matrix calculus, what we adopt is called the \textit{numerator layout} or
\textit{Jacobian formulation}; that is, the derivative of a scalar with
respect to a column vector is a row vector.
%As the result, our chain rule is never backward.

\section{The Kronecker Product Model}

\label{sec model}

%\subsection{The Model and Identification}

In this section we provide more details on the model. Consider an
$n$-dimensional weakly stationary time series vector $y_{t},$ where
$\mu:=\mathbb{E}y_{t}$ and covariance matrix $\Sigma:=\mathbb{E}[(y_{t}%
-\mu)(y_{t}-\mu)^{\intercal}]$. Let $D$ be the diagonal matrix containing the
diagonal entries of $\Sigma$.\footnote{Matrix $D$ should not be confused with
the duplication matrix $D_{n}$ defined in Notation.} The correlation matrix of
$y_{t},$ denoted $\Theta,$ is $\Theta:=D^{-1/2}\Sigma D^{-1/2}.$ A Kronecker
product model for the covariance or correlation matrix is given by
(\ref{eqn model}). The factorization $n=n_{1}\times\cdots\times n_{v}$ could
be the prime factorization, which exists for any integer $n,$ or it could be
an aggregation of that. For example, if $n=2^{8},$ one factorization is
$2\times2\times\cdots\times2,$ called the minimal factorization, at the other
extreme $2^{8}\times1$ is the maximal factorization (we do not consider the maximal factorization in this article). One also has
$4\times4\times4\times4$ and \ $2\times16\times2\times4$ etc. Highly composite
numbers such as $2^{8}$ offer many possible factorizations, but if $n$ is not
composite or not composite enough, one can add a vector of pseudo variables to
the system until the final dimension is composite enough.\footnote{It is
recommended to add a vector of independent variables $z_{t}\sim N\left(
0,I_{k}\right)  $ such that $(y_{t}^{\intercal},z_{t}^{\intercal})^{\intercal
}$ is an $n^{\prime}\times1$ random vector with $n^{\prime}\times n^{\prime}$
correlation matrix
\[
\Theta=\left[
\begin{array}
[c]{cc}%
\Theta_{y} & 0\\
0 & I_{k}%
\end{array}
\right]  .
\]
}

Let $\Delta$ denote $\Sigma$ or $\Theta$ according to the modelling purpose.
If $\Delta\in\{\Delta^{\ast}\}$, we say that the Kronecker product model
$\{\Delta^{\ast}\}$ is \textit{correctly specified}. Otherwise the Kronecker
product model $\{\Delta^{\ast}\}$ is \textit{misspecified}. We first make
clearer when a Kronecker product model is correctly specified (see
\cite{verheeswansbeek1990} and \cite{cudeck1988} for more discussion). A
Kronecker product arises when data have some \textit{multiplicative array}
structure. For example, suppose that $u_{j,k}$ are error terms in a panel
regression model with $j=1,\ldots,{J}$ and $k=1,\ldots,{K}$. The interactive
effects model of \cite{bai2009} is that $u_{j,k}=\gamma_{j}f_{k},$ which
implies that $u=\gamma\otimes f,$ where $u$ is the $JK\times1$ vector
containing all the elements of $u_{j,k},$ $\gamma=(\gamma_{1},\ldots
,\gamma_{{J}})^{\intercal}$, and $f=(f_{1},\ldots,f_{{K}})^{\intercal}.$
Suppose that $\gamma,f$ are random, where $\gamma$ is independent of $f$, and
both vectors have mean zero. Then,%
\[
\mathbb{E}[uu^{\intercal}]=\mathbb{E}[\gamma\gamma^{\intercal}]\otimes
\mathbb{E}[ff^{\intercal}].
\]
In this case the covariance matrix of $u$ is a Kronecker product of two
sub-matrices. If one dimension were time and the other were firm, then this
implies that the covariance matrix of $u$ is the product of a covariance matrix
representing cross-sectional dependence and a covariance matrix representing
the time series dependence.

We can think of our more general model (\ref{eqn model}) arising from
\textit{multi-index data with $v$ multiplicative factors}. Multiway arrays are
one such example as each observation has $v$ different indices (see
\cite{hoff2015}). Suppose that $u_{i_{1},i_{2},\ldots,i_{v}}=\varepsilon
_{1,i_{1}}\varepsilon_{2,i_{2}}\cdots\varepsilon_{v,i_{v}},$ $i_{j}%
=1,\ldots,n_{j}$ for $j=1,\ldots,v,$ or in vector form
\[
u=(u_{1,1,\ldots,1},\ldots,u_{n_{1},n_{2},\ldots,n_{v}})^{\intercal
}=\varepsilon_{1}\otimes\varepsilon_{2}\otimes\cdots\otimes\varepsilon
_{v},\label{in}%
\]
where the factor $\varepsilon_{j}=(\varepsilon_{j,1},\ldots,\varepsilon
_{j,n_{j}})^{\intercal}$ is a mean zero random vector of length $n_{j}$ with
covariance matrix $\Sigma_{j}$ for $j=1,\ldots,v$, and in addition the factors
$\varepsilon_{1},\ldots,\varepsilon_{v}$ are mutually independent. Then
\[
\Sigma=\mathbb{E}[uu^{\intercal}]=\Sigma_{1}\otimes\Sigma_{2}\otimes
\cdots\otimes\Sigma_{v}.
\]
We hence see that the covariance matrix is a Kronecker product of $v$
sub-matrices. Such multiplicative effects may be a valid description of a data
generating process.\footnote{For example, in portfolio choice, one might
consider, say, 250 equity portfolios constructed by intersections of 5 size
groups (quintiles), 5 book-to-market equity ratio groups (quintiles) and 10
industry groups, in the spirit of \cite{famafrench1993}. For example, one
equity portfolio might consist of stocks which are in the smallest size
quintile, largest book-to-market equity ratio quintile, and construction
industry simultaneously. Then a Kronecker product model is applicable either
directly to the covariance matrix of returns of these 250 equity portfolios or
to the covariance matrix of the residuals after purging other common risk
factors such as momentum.
%Motivation of Kronecker product
}

In earlier versions of this article we emphasized the Kronecker product model
for the covariance matrix. We now focus primarily on the correlation matrix
for the reasons mentioned in the introduction and leave the diagonal variance
matrix $D$ unrestricted. For the present discussion we assume that $D$ (as
well as $\mu$) is known. A Kronecker product model for $\Theta$ is given by
(\ref{eqn model}) with $\Delta^{\ast}$ and $\{\Delta_{j}^{\ast}\}_{j=1}^{v}$
replaced by $\Theta^{\ast}$ and $\{\Theta_{j}^{\ast}\}_{j=1}^{v}$,
respectively. Since $\Theta$ is a correlation matrix, this implies that the
diagonal entries of $\Theta_{j}^{\ast}$ must be the same, although this
diagonal entry could differ as $j$ varies (so long as the diagonal entries of
the implied $\Theta^{\ast}$ are one). Without loss of generality, we may
impose a normalization constraint that all the diagonal entries of sub-matrices $\{\Theta_{j}^{\ast}\}_{j=1}^{v}$ are equal to one.

The Kronecker product model substantially reduces the number of parameters to
estimate. In an unrestricted correlation matrix, there are $n(n-1)/2$
parameters, while a Kronecker product model has only $\sum_{j=1}^{v}%
n_{j}(n_{j}-1)/2$ parameters. As an extreme illustration, when $n=256$, the
unrestricted correlation matrix has 32,640 parameters while a Kronecker
product model of factorization $256=2^{8}$ has only 8 parameters! Since we do
not restrict the diagonal matrix we have an additional $n$ variance
parameters,\footnote{These parameters can be estimated in a first step by
standard methods.} so overall the correlation matrix version of the model has
more parameters and more flexibility than the covariance matrix version of the model. The Kronecker product model induces sparsity. Specifically,
although $\Theta^{\ast}$ is not sparse, the matrix $\log\Theta^{\ast}$ is
sparse, where $\log$ denotes the (principal) matrix logarithm defined through
the eigendecomposition of a real symmetric, positive definite matrix (see
\cite{higham2008} p20 for a definition). This is due to a property of
Kronecker products (see Lemma \ref{prop log kronecker} in SM \ref{sec B1} for
derivation), that
\begin{equation}
\log\Theta^{\ast}=\log\Theta_{1}^{\ast}\otimes I_{n_{2}}\otimes\cdots\otimes
I_{n_{v}}+I_{n_{1}}\otimes\log\Theta_{2}^{\ast}\otimes I_{n_{3}}\otimes
\cdots\otimes I_{n_{v}}+\cdots+I_{n_{1}}\otimes I_{n_{2}}\otimes\cdots
\otimes\log\Theta_{v}^{\ast}, \label{loglin}%
\end{equation}
whence we see that $\log\Theta^{\ast}$ has many zero elements, generated by
identity sub-matrices.\footnote{A final property of the Kronecker product
model is that it is invariant under the Lie group of transformations
$\mathcal{G}$ generated by $A_{1}\otimes A_{2}\otimes\cdots\otimes A_{v},$
where $A_{j}$ are $n_{j}\times n_{j}$ nonsingular matrices (see
\cite{browneshapiro1991}). This structure can be used to characterise the
tangent space $\mathcal{T}$ of $\mathcal{G}$ and to define a relevant
equivariance concept for restricting the class of estimators for optimality
considerations.} That is, we can write $\mathrm{vech}(\log\Theta^{\ast
})=E\theta^{\ast}$ for some matrix $E$ of zeros and ones and vector
$\theta^{\ast}$ containing the unrestricted elements of $\log\Theta_{1}^{\ast
},\ldots,\log\Theta_{v}^{\ast}.$

%\subsection{Parameters in $\Theta_{j}^{*}$}

We next discuss some further identification/parameterization issues. First of all, sub-matrices $\{\Theta_{j}^{*}\}_{j=1}^{v}$ are uniquely identified by the
following argument based on the architecture of $\Theta^{*}$.  Suppose that $\Theta^{*}=\tilde{\Theta}^{*}_{1}\otimes
\cdots\otimes\tilde{\Theta}^{*}_{v}$ for other sub-matrices $\{\tilde{\Theta
}_{j}^{*}\}_{j=1}^{v}$, with diagonal elements being one, whose dimensions agree with those of $\{\Theta_{j}%
^{*}\}_{j=1}^{v}$. Let $\rho_{j,k\ell}^{*}$ and $\tilde{\rho}^{*}_{j,k\ell}$
denote a typical off-diagonal element of $\Theta_{j}^{*}$ and $\tilde{\Theta
}_{j}^{*}$, respectively ($k,\ell=1,\ldots,n_{j}$, $k\neq\ell$). Note that
$\rho^{*}_{j,k\ell}$ appears, on its own, in some elements of $\Theta^{*}$, so does $\tilde{\rho}^{*}_{j,k\ell}$ in the same positions. We must have
$\rho_{j,k\ell}^{*}=\tilde{\rho}_{j;k\ell}^{*}$ for all $k,\ell=1,\ldots
,n_{j}$, $k\neq\ell$ and $j=1,\ldots,v$. Therefore, sub-matrices $\{\Theta
_{j}^{*}\}_{j=1}^{v}$ are identified from $\Theta^{*}$ once $\{n_{j}%
\}_{j=1}^{v}$ are specified, or equivalently $\{\rho^{*}_{j,k\ell}:
k,\ell=1,\ldots, n_{j}, k\neq\ell\}_{j=1}^{v}$ are identified from $\Theta
^{*}$ once $\{n_{j}\}_{j=1}^{v}$ are specified. We call $\{\rho^{*}_{j,k\ell}:
k,\ell=1,\ldots, n_{j}, k\neq\ell\}_{j=1}^{v}$ the \textit{original
parameters} of some member $\Theta^{*}$ in the Kronecker product model. If $\{\Theta
_{j}^{*}\}_{j=1}^{v}$ are positive definite correlation matrices, then so is
$\Theta^{*}.$
%In a finite dimensional setting, this guarantees identification of maximum likelihood or nonlinear minimum distance estimation strategies. Because of the large dimensional setting we will turn to the matrix logarithm of the correlation matrix because of the linearity in (\ref{loglin}).

Second, the matrix logarithm of a correlation matrix has a complicated
structure, with its diagonal elements taking any non-positive values and its off-diagonal elements taking any values (\cite{archakovhansen2018} Lemma 2). As an illustration, suppose that%
\[
\Theta_{1}^{\ast}=\left(
\begin{array}
[c]{ccc}%
1 & 0.8 & 0.5\\
0.8 & 1 & 0.2\\
0.5 & 0.2 & 1
\end{array}
\right)  ,
\]
then%
\[
\log\Theta_{1}^{\ast}=\left(
\begin{array}
[c]{ccc}%
-0.75 & 1.18 & 0.64\\
1.18 & -0.55 & -0.07\\
0.64 & -0.07 & -0.17
\end{array}
\right)  .
\]
There are $\sum_{j=1}^{v}n_{j}(n_{j}+1)/2$ parameters in $\{\log\Theta
_{j}^{\ast}\}_{j=1}^{v}$; we call these \textit{log parameters} of some member
$\Theta^{\ast}$ in the Kronecker product model. On the other hand,
$\Theta^{\ast}$ has only $\sum_{j=1}^{v}n_{j}(n_{j}-1)/2$ original parameters.
For each $\Theta_{j}^{*}$, its $n_{j}(n_{j}-1)/2$ original parameters
completely pin down its $n_{j}(n_{j}+1)/2$ log parameters. In other words,
there exists a function $f:\mathbb{R}^{n_{j}(n_{j}-1)/2}\rightarrow
\mathbb{R}^{n_{j}(n_{j}+1)/2}$ which maps the original parameters to the log
parameters. However, when $n_{j}>4$, $f$ does \textit{not} have a closed form
because when $n_{j}>4$ the continuous functions which map elements of a matrix
to its eigenvalues have no closed form. When $n_{j}=2$, we can solve $f$ by
hand (see Example \ref{ex closed form when nj=2}).
%When $n_{j}=3$, one could use software to perform symbolic computation, but the expressions will be extremely complicated.

\bigskip

\begin{example}
\label{ex closed form when nj=2} Suppose
\[
\Theta^{*}_{1}=\left(
\begin{array}
[c]{cc}%
1 & \rho_{1}^{*}\\
\rho_{1}^{*} & 1
\end{array}
\right)  .
\]
The eigenvalues of $\Theta_{1}^{*}$ are $1+\rho_{1}^{*}$ and $1-\rho_{1}^{*}$,
respectively. The corresponding eigenvectors are $(1,1)^{\intercal}/\sqrt{2}$
and $(1,-1)^{\intercal}/\sqrt{2}$, respectively. Therefore
\begin{align*}
\log\Theta_{1}^{*}  &  =\left(
\begin{array}
[c]{cc}%
1 & 1\\
1 & -1
\end{array}
\right)  \left(
\begin{array}
[c]{cc}%
\log(1+\rho_{1}^{*}) & 0\\
0 & \log(1-\rho_{1}^{*})
\end{array}
\right)  \left(
\begin{array}
[c]{cc}%
1 & 1\\
1 & -1
\end{array}
\right)  \frac{1}{2}\\
&  =\left(
\begin{array}
[c]{cc}%
\frac{1}{2}\log(1-[\rho_{1}^{*}]^{2}) & \frac{1}{2}\log
\del[2]{\frac{1+\rho_1^*}{1-\rho_1^*}}\\
\frac{1}{ 2}\log\del[2]{\frac{1+\rho_1^*}{1-\rho_1^*}} & \frac{1}{2}%
\log(1-[\rho_{1}^{*}]^{2})
\end{array}
\right)  .
\end{align*}
Thus
\[
f(\rho)=\del[3]{\frac{1}{2}\log(1-\rho^{2}),\frac{1}{2}\log
\del[2]{\frac{1+\rho}{1-\rho}}, \frac{1}{2}\log(1-\rho^{2}) }^{\intercal}.
\]
%whence we see that $\rho_{1}^{0}$ generates two distinct entries for $\Omega_{1}^{0}$. The off-diagonal element $\frac{1}{2}\log \del[1]{\frac{1+\rho_1^0}{1-\rho_1^0}}$ is the Fisher's $z$-transformation of $\rho_{1}^{0}$, which has a fine statistical pedigree. We also see that $\Omega_{1}^{0}$ is not only symmetric about the diagonal, but also symmetric about the cross-diagonal (from the upper right to the lower left). We can use entries of $\Omega_{1}^{0}$ to recover $\rho_{1}^{0}$ in some over-identified sense.

\end{example}

\bigskip

Third, there are several ways to achieve identification of $\{\log\Theta
_{j}^{\ast}\}_{j=1}^{v}$ given $\log\Theta^{\ast}$ (i.e., identification of
log parameters of $\Theta^{*}$). We outline two methods. The \textit{fill and
shrink} method estimates the log parameters without imposing the restrictions implied by that $\{\Theta_{j}^{\ast}\}_{j=1}^{v}$ being correlation matrices, and then imposes those restrictions afterwards. In this case at the estimation stage, we must impose $v-1$ identification restrictions on the log parameters because in (\ref{loglin}) the diagonal elements of $\log\Theta^{*}$
are sums of diagonal elements from $\{\log\Theta_{j}^{\ast}\}_{j=1}^{v}$ (see
Example \ref{ex2x2x2} in SM \ref{sec B1}). There are several ways to impose
these $v-1$ identification restrictions. For example, one can set
$\tr(\log\Theta_{j}^{\ast})$ to be some fixed value for $j=1,\ldots,v-1$, or
one can set $v-1$ diagonal elements of $\{\log\Theta_{j}^{*}\}_{j=1}^{v}$ to
be zero (see Examples \ref{ex2x2} and \ref{ex2x2x2} in SM \ref{sec B1} for
illustrations). Then in Theorem \ref{prop kronecker formula} in Appendix
\ref{sec A.1} we show that there exists an $n(n+1)/2\times s$ full column
rank, deterministic matrix $E$ such that
\[
\vech(\log\Theta^{\ast})=E\theta^{\ast},
\]
where $\theta^{\ast}\in\mathbb{R}^{s}$ are the \textit{unrestricted} log
parameters of $\Theta^{\ast}$, where $s:=\sum_{j=1}^{v}n_{j}(n_{j}%
+1)/2-(v-1)=O(\log n)$. So far we have not imposed those restrictions that
$\{\Theta_{j}^{\ast}\}_{j=1}^{v}$ are correlation matrices. Nevertheless,
$D^{1/2}\exp(\log\Theta^{\ast})D^{1/2}$ will always be a covariance matrix, and one
could re-compute the correlation matrix from $D^{1/2}\exp(\log\Theta^{\ast})D^{1/2}$ by re-normalization.
Alternatively, one could use minimum distance estimation to shrink $\exp
(\log\Theta_{j}^{\ast})$ to a correlation matrix for $j=1,\ldots,v$ (see SM
\ref{sec two stage MD} for a discussion).

On the other hand, the \textit{shrink and fill} method identifies a
\textit{subset} of unrestricted log parameters and then fills in the remainder
afterwards. A recent paper of \cite{archakovhansen2018} proposed a neat way to achieve this. Let $\tilde{\theta}_{j}^{\ast}:=\mathrm{vecl}(\log\Theta_{j}^{\ast
})$ and we can identify $\{\tilde{\theta}_{j}^{*}\}_{j=1}^{v}$ from $\mathrm{vecl}(\log\Theta^{\ast})$.
%and let $\tilde{\theta}^{\ast}$ be the vector containing all these $\{\tilde{\theta}_{j}^{*}\}_{j=1}^{v}$. Then define the matrix $\tilde{E}$ that satisfies $\mathrm{vecl}(\log\Theta^{\ast})=\tilde{E}\tilde{\theta}%^{\ast},$ which identifies $\tilde{\theta}^{\ast}$. 
Then we can use $\tilde{\theta}_{j}^{\ast}$ to uniquely determine the diagonal elements of $\log\Theta_{j}^{\ast}$ by some function $\phi:\mathbb{R}^{n_{j}(n_{j}%
-1)/2}\rightarrow\mathbb{R}^{n_{j}},$ which can be obtained numerically (in
the case $n_{j}=2$ there exists a closed form, see Example
\ref{ex closed form when nj=2}). \cite{archakovhansen2018} gave a concrete
algorithm to do this and established its validity.

%We mentioned in the introduction that the correct theoretical framework to study the large dimensional case is to use the joint asymptotics. In particular, $n$ tends to infinity. There are actually four cases regarding $(\{n_{j}\}_{j=1}^{v}, v)$ (see Table \ref{table v and nj}). Case (a) is not interesting as it just passes the problem of estimating a Kronecker product model to that of its sub-matrices. The focus of this article is case (b) where $s=O(\log n)$. Cases (c) and (d) are more general, but more difficult to analyse than case (b); we leave them for future research.

%\begin{table}[h]
%\centering
%\begin{tabular}
%[c]{ccc}%
%\toprule & \textit{fixed} & \textit{tend to infinity}\\
%\midrule (a) & $v$ & $\{n_{j}\}_{j=1}^{v}$\\
%(b) & $\{n_{j}\}_{j=1}^{v}$ & $v$\\
%(c) & $\{n_{j}\}_{j\in J}$ & $\{n_{j}\}_{j\in J^{c}}$ \& $v$\\
%(d) & $-$ & $\{n_{j}\}_{j=1}^{v}$ \& $v$\\
%\bottomrule &  &
%\end{tabular}
%\caption{{\protect\small $J$ is some proper subset of $\{1,\ldots,v\}$.}}%
%\label{table v and nj}%
%\end{table}

We shall use the fill and shrink method in what follows; in particular we set
the first diagonal entry of $\log\Theta_{j}^{*}$ to zero for $j=1,\ldots,v-1$.
To summarise, in order to estimate a correlation matrix $\Theta$ using a
Kronecker product model $\Theta^{*}$, there are two approaches. First, one can
estimate the original parameters using the principle of maximum likelihood
(see Section \ref{sec QMLE original parameter}) or nonlinear minimum distance.
Then form an estimate of $\Theta^{*}$. Second, one can estimate the
unrestricted log parameters $\theta^{\ast}$ using the principle of minimum
distance (see Section \ref{sec MD estimation}) or maximum likelihood (see
Section \ref{sec QMLE original parameter}). Form an estimate of $\exp
(\log\Theta^{*})$ and then recover the estimated correlation matrix using
either re-normalization or shrinkage. To study the theoretical properties of a
Kronecker product model in large dimension, the second approach is more
appealing as log parameters are additive from (\ref{loglin}) while original
parameters are multiplicative in nature; additive objects are easier to
analyse theoretically than multiplicative objects.

%via matrix exponential.\footnote{When one adopts the second approach, the final estimated correlation matrix need not be a correlation matrix. In general, the ultimate goal is to estimate the covariance matrix, not the correlation matrix, even if one decides to fit the correlation matrix, instead of the covariance matrix, with a Kronecker product model. In this sense, this is not an issue of the second approach. If the goal is to estimate the correlation matrix, then one could use the first approach instead. If one insists to use the principle of minimum distance to estimate the correlation matrix, in order to make the final estimate a correlation matrix, we discuss a \textit{two-stage minimum distance} approach in SM \ref{sec two stage MD}.}

\section{Linear Minimum Distance Estimator}

\label{sec MD estimation}

In this section, we define a class of estimators of the (unrestricted) log parameters
$\theta^{\ast}$ of some member in the Kronecker product model (\ref{eqn model}%
), which are linear in the log sample correlation matrix, and give its
asymptotic properties.

\subsection{Estimation}

We observe a sample $\{y_{t}\}_{t=1}^{T}.$ Define the sample covariance matrix
and sample correlation matrix
\[
\hat{\Sigma}_{T}:=\frac{1}{T}\sum_{t=1}^{T}(y_{t}-\bar{y})(y_{t}-\bar
{y})^{\intercal},\qquad\hat{\Theta}_{T}:=\hat{D}_{T}^{-1/2}\hat{\Sigma}%
_{T}\hat{D}_{T}^{-1/2},
\]
where $\bar{y}:=(1/T)\sum_{t=1}^{T}y_{t}$ and $\hat{D}_{T}$ is a diagonal
matrix whose diagonal elements are diagonal elements of $\hat{\Sigma}_{T}$. We
show in Appendix \ref{sec A2} that in the Kronecker product model
$\{\Theta^{\ast}\}$ there exists a unique member, denoted by $\Theta^{0}$,
which is closest to the correlation matrix $\Theta$ in the following sense:
\begin{equation}
\theta^{0}=\theta^{0}(W):=\arg\min_{\theta^{\ast}\in\mathbb{R}^{s}}%
[\vech(\log\Theta)-E\theta^{\ast}]^{\intercal}W[\vech(\log\Theta
)-E\theta^{\ast}], \label{ali popu MD objective fun}%
\end{equation}
where $W$ is a $n(n+1)/2\times n(n+1)/2$ symmetric, positive definite
weighting matrix which is free to choose. Clearly, $\theta^{0}$ has the closed
form solution $\theta^{0}=(E^{\intercal}WE)^{-1}E^{\intercal}W\vech(\log
\Theta).$ The population objective function in
(\ref{ali popu MD objective fun}) allows us to define a minimum distance (MD)
estimator:
\begin{equation}
\hat{\theta}_{T}=\hat{\theta}_{T}(W):=\arg\min_{b\in\mathbb{R}^{s}}%
[\vech(\log\hat{\Theta}_{T})-Eb]^{\intercal}W[\vech(\log\hat{\Theta}_{T})-Eb],
\label{eqn MD objective function}%
\end{equation}
whence we can solve
\begin{equation}
\hat{\theta}_{T}=(E^{\intercal}WE)^{-1}E^{\intercal}W\vech(\log\hat{\Theta
}_{T}). \label{eqn thetaTW closed form solution}%
\end{equation}
Note that $\theta^{0}$ is the quantity which one should expect $\hat{\theta
}_{T}$ to converge to in some probabilistic sense regardless of whether the
Kronecker product model $\{\Theta^{\ast}\}$ is correctly specified or not.
When $\{\Theta^{\ast}\}$ is correctly specified, we have $\theta
^{0}=(E^{\intercal}WE)^{-1}E^{\intercal}W\vech(\log\Theta)=(E^{\intercal
}WE)^{-1}E^{\intercal}WE\theta=\theta.$ In this case, $\hat{\theta}_{T}$ is
indeed estimating the elements of the correlation matrix $\Theta$.

In practice the MD estimator is easy to compute. The matrix
$E^{\intercal}WE$ is of dimensions $s\times s$ and is highly structured (at
least in the diagonal $W$ case$).$ One only needs a user-defined function in
some software to generate the matrix $E$ before one can use formula
(\ref{eqn thetaTW closed form solution}) to compute the MD
estimator.\footnote{We have written a user-defined function in Matlab which
can return $E$ within a few seconds for fairly large $n$, say, $n=625$. It is
available upon request.}

\subsection{Rate of Convergence}

We now introduce some assumptions for our theoretical analysis. These
conditions are sufficient but far from necessary.

\begin{assumption}
\label{assu subgaussian vector}

\item
\begin{enumerate}
[(i)]

\item For all $t$, for every $a\in\mathbb{R}^{n}$ with $\|a\|_{2}=1$, there
exist absolute constants $K_{1}>1, K_{2}>0, r_{1}>0$ such
that\footnote{"Absolute constants" mean constants that are independent of both
$n$ and $T$.}
\[
\mathbb{E}\sbr[2]{\exp\del [1]{K_2 |a^{\intercal}y_t|^{r_1}}}\leq K_{1}.
\]

%\item $\{y_{t}\}_{t=1}^{T}$ are subgaussian random vectors. That is, for all $t$, for every $a\in\mathbb{R}^{n}$ with $\|a\|_{2}=1$, and every $\epsilon>0$
%\[\mathbb{P}(|a^{\intercal}y_{t}|\geq\epsilon)\leq Ke^{-C\epsilon^{2}},\]
%for positive absolute constants $K$ and $C$.

\item The time series $\{y_{t}\}_{t=1}^{T}$ are normally distributed.
\end{enumerate}
\end{assumption}

\begin{assumption}
\label{assu mixing} There exist absolute constants $K_{3}>0$ and $r_{2}>0$
such that for all $h\in\mathbb{N}$
\[
\alpha(h)\leq\exp\del [1]{-K_3h^{r_2}},
\]
where $\alpha(h)$ is the $\alpha$-mixing (i.e., strong mixing) coefficients of
$y_{t}$ which are defined by $\alpha(0)=1/2$ and for $h\in\mathbb{N}$
\[
\alpha(h):=2\sup_{t} \sup_{\substack{A\in\sigma(\cdots, y_{t-1},y_{t})
\\B\in\sigma(y_{t+h},y_{t+h+1},\cdots)}%
}\envert[1]{\mathbb{P}(A\cap B)-\mathbb{P}(A)\mathbb{P}(B)},
\]
where $\sigma(\cdot)$ defines sigma algebra.
\end{assumption}

\begin{assumption}
\label{assu n indexed by T}

\item
\begin{enumerate}
[(i)]

\item Suppose $n,T\to\infty$ simultaneously, and $n/T\to0$.
%c\in [0,1]$.

\item Suppose $n,T\to\infty$ simultaneously, and
%\[\frac{n^2 \varpi  \kappa^2(W)\log^2(1+T)\log ^2(1+n^2)\log n}{T}=o(1).\]
%

\[
\frac{n^{4}\varpi^{4}\kappa^{6}(W)(\log^{5} n)\log^{2}(1+T)}{T}=o(1)
\]
%\[\frac{n^{2}\kappa^{3}(W)\varpi^{2} \log^{2} n}{T}\del[2]{T^{2/\gamma}\log^2 n  \vee n^2\kappa^3(W)\varpi^2\log^2n\cdot \log^5 n^4}=o(1),\quad\text{for some }\gamma>2,\]
where $\kappa(W)$ is the \textit{condition number} of $W$ for matrix inversion
with respect to the spectral norm, i.e., $\kappa(W):= \|W^{-1}\|_{\ell_{2}%
}\|W\|_{\ell_{2}}$ and $\varpi$ is defined in Assumption
\ref{assu about D and Dhat}(ii).

\item Suppose $n,T\to\infty$ simultaneously, and

\begin{enumerate}
[(a)]

\item
\[
\frac{n^{4} \varpi^{4}\kappa(W) ( \log^{5} n) \log^{2}(1+T)}{T}=o(1),
\]

\item
\[
\frac{\varpi^{2} \log n}{n}=o(1),
\]

\end{enumerate}

where $\kappa(W)$ is the \textit{condition number} of $W$ for matrix inversion
with respect to the spectral norm, i.e., $\kappa(W):= \|W^{-1}\|_{\ell_{2}%
}\|W\|_{\ell_{2}}$, and $\varpi$ is defined in Assumption
\ref{assu about D and Dhat}(ii).
\end{enumerate}
\end{assumption}

\begin{assumption}
\label{assu about D and Dhat}

\item
\begin{enumerate}
[(i)]

\item The minimum eigenvalue of $\Sigma$ is bounded away from zero by an
absolute constant.

\item Suppose
\[
\text{mineval}\del [3]{\frac{1}{n}E^{\intercal}E}\geq\frac{1}{\varpi}>0.
\]
(At most $\varpi=o( n)$.)
\end{enumerate}
\end{assumption}

\bigskip

Assumption \ref{assu subgaussian vector}(i) is standard in high-dimensional
theoretical work (e.g., \cite{fanliaomincheva2011}, \cite{changqiuyaozou2018}
etc). In essence it assumes that a random vector has some exponential-type
tail probability (c.f. Lemma \ref{lemmaexponentialtail} in Appendix
\ref{secArateofconvergence}), which allows us to invoke some concentration
inequality such as a version of the Bernstein's inequality (e.g., Theorem
\ref{thmbernsteininequality} in Appendix \ref{sec oldappendixB}). The
parameter $r_{1}$ restricts the size of the tail of $y_{t}$ - the smaller
$r_{1}$, the heavier the tail. When $r_{1}=2$, $y_{t}$ is said to be
\textit{subgaussian}, when $r_{1}=1$, $y_{t}$ is said to be
\textit{subexponential}, and when $0<r_{1}<1$, $y_{t}$ is said to be
\textit{semiexponential}.

Needless to say, Assumption \ref{assu subgaussian vector}(i) is stronger than
a finite polynomial moment assumption as it assumes the existence of some
exponential moment. In a setting of independent observations,
\cite{vershynin2012} replaced Assumption \ref{assu subgaussian vector}(i) with
a finite polynomial moment condition and established a rate of convergence for
covariance matrices, which is slightly worse than what we have in Theorem
\ref{thm main rate of convergence}(i) for correlation matrices. For dependent
data, relaxation of the subgaussian assumption is currently an active research
area in probability theory and statistics. One of the recent work is
\cite{wuwu2016} in which they relaxed subgaussianity to a finite polynomial
moment condition in high-dimensional linear models with help of Nagaev-type
inequalities. Thus Assumption \ref{assu subgaussian vector}(i) is likely to be
relaxed when new probabilistic tools become available.

%is standard in high-dimensional theoretical work. In essence it assumes that a random vector has exponential tail probabilities, which allows us to invoke some concentration inequality such as the Bernstein's inequality in Appendix \ref{sec oldappendixB}. Note that Assumption \ref{assu subgaussian vector}(i) could be replaced by a finite moment assumption and this will only result a rate slightly worse than $\sqrt{n/T}$ in Theorem \ref{thm main rate of convergence}(i) (c.f. \cite{vershynin2012}).

Assumption \ref{assu subgaussian vector}(ii), which will only be used in
Section \ref{sec QMLE} for one-step estimation, implies Assumption
\ref{assu subgaussian vector}(i) with $0<r_{1}\leq2$. Assumption
\ref{assu subgaussian vector}(ii) is not needed for the minimum distance
estimation (Theorem \ref{thm asymptotic normality} or
\ref{thm asymptotic normality MD when D is unknown}) though.

Assumption \ref{assu mixing} assumes that $\{y_{t}\}_{t=1}^{T}$ is alpha
mixing (i.e., strong mixing) because $\alpha(h)\to0$ as $h\to\infty$. In fact,
we require it to decrease at an exponential rate. The bigger $r_{2}$ gets, the
faster the decay rate and the less dependence $\{y_{t}\}_{t=1}^{T}$ exhibits.
This assumption covers a wide range of time series. It is well known that both
classical ARMA and GARCH processes are strong mixing with mixing coefficients
which decrease to zero at an exponential rate (see Section 2.6.1 of
\cite{fanyao2003} and the references therein).

Assumption \ref{assu n indexed by T}(i) is for the derivation of a rate of
convergence of $\hat{\Theta}_{T}-\Theta$ in terms of spectral norm. To
establish the \textit{same} rate of convergence of $\hat{\Sigma}_{T}-\Sigma$
in terms of spectral norm, one only needs $n/T\to c\in[0,1]$. However for
correlation matrices, we need $n/T\to0$. This is because a correlation matrix
involves inverses of standard deviations (see Lemma
\ref{lemma saikkonen lemma} in Appendix \ref{sec oldappendixB}).

Assumptions \ref{assu n indexed by T}(ii) and (iii) are \textit{sufficient}
conditions for the asymptotic normality of the minimum distance estimators
(Theorems \ref{thm asymptotic normality} and
\ref{thm asymptotic normality MD when D is unknown}) and one-step estimator
(Theorem \ref{thm one step estimator asymptotic normality}), respectively.
%If Assumption \ref{assu subgaussian vector}(i) holds, $\gamma$ in Assumption \ref{assu n indexed by T}(ii) and (iii) could be made arbitrarily large, which makes Assumptions \ref{assu n indexed by T}(ii) and (iii) much less restrictive.
Assumption \ref{assu n indexed by T}(ii) or (iii) necessarily requires
$n^{4}/T\to0$. At first glance, it looks restrictive, but we would like to
emphasize that this is only a sufficient condition. We will have more to say
on this assumption in the discussions following Theorem
\ref{thm asymptotic normality}.

Assumption \ref{assu about D and Dhat}(i) is also standard. This ensures that
$\Theta$ is positive definite with the minimum eigenvalue bounded away from 0
by an absolute positive constant (see Lemma \ref{prop mini eigenvalue}(i) in
Appendix \ref{sec A4}) and its logarithm is well-defined. Assumption
\ref{assu about D and Dhat}(ii) postulates a lower bound for the minimum
eigenvalue of $E^{\intercal}E/n$; that is
\[
\frac{1}{\sqrt{\text{mineval}\del[1]{\frac{1}{n}E^{\intercal}E}}}%
=O(\sqrt{\varpi}).
\]
We divide $E^{\intercal}E$ by $n$ because all the non-zero elements of
$E^{\intercal}E$ are a multiple of $n$ (see Lemma \ref{prop decode E} in
Appendix \ref{sec A.1}). In words, Assumption \ref{assu about D and Dhat}(ii)
says that the minimum eigenvalue of $E^{\intercal}E/n$ is allowed to slowly
drift to zero.

\bigskip

The following theorem establishes an upper bound on the rate of convergence
for the minimum distance estimator $\hat{\theta}_{T}$. To arrive at this, we
restrict $r_{1}$ and $r_{2}$ such that $1/r_{1}+1/r_{2}>1$. However, this is
not a necessary condition.
%In words, we cannot have a times series which has both a very thin tail and a very fast exponentially decaying strong mixing coefficients.

\begin{thm}
\label{thm main rate of convergence}

\item
\begin{enumerate}
[(i)]

\item Suppose Assumptions \ref{assu subgaussian vector}(i), \ref{assu mixing},
\ref{assu n indexed by T}(i) and \ref{assu about D and Dhat}(i) hold with
$1/r_{1}+1/r_{2}>1$. Then
\[
\| \hat{\Theta}_{T}- \Theta\|_{\ell_{2}}=O_{p}\del [3]{\sqrt{\frac{n}{T}}},
\]
where $\|\cdot\|_{\ell_{2}}$ is the spectral norm.

\item Suppose that $\| \hat{\Theta}_{T}- \Theta\|_{\ell_{2}}< A$ with
probability approaching 1 for some absolute constant $A>1$, then we have
\[
\|\log\hat{\Theta}_{T}-\log\Theta\|_{\ell_{2}}=O_{p}(\| \hat{\Theta}_{T}-
\Theta\|_{\ell_{2}}).
\]

\item Suppose Assumptions \ref{assu subgaussian vector}(i), \ref{assu mixing},
\ref{assu n indexed by T}(i) and \ref{assu about D and Dhat} hold with
$1/r_{1}+1/r_{2}>1$. Then
\[
\|\hat{\theta}_{T}-\theta^{0}\|_{2}= O_{p}%
\del [3]{\sqrt{\frac{n\varpi \kappa(W)}{T}}},
\]
where $\|\cdot\|_{2}$ is the Euclidean norm, $\kappa(W)$ is the condition
number of $W$ for matrix inversion with respect to the spectral norm, i.e.,
$\kappa(W):=\|W^{-1}\|_{\ell_{2}}\|W\|_{\ell_{2}}$, and $\varpi$ is defined in
Assumption \ref{assu about D and Dhat}(ii).
\end{enumerate}
\end{thm}

\begin{proof}
See Appendix \ref{secArateofconvergence}.
\end{proof}

\bigskip

Theorem \ref{thm main rate of convergence}(i) provides a rate of convergence
of the spectral norm of $\hat{\Theta}_{T}- \Theta$, which is a stepping stone
for the rest of theoretical results. This rate is the same as that of
$\|\hat{\Sigma}_{T}-\Sigma\|_{\ell_{2}}$. The rate $\sqrt{n/T}$ is optimal in
the sense that it cannot be improved without a further structural assumption
on $\Theta$ or $\Sigma$.

Theorem \ref{thm main rate of convergence}(ii) is of independent interest as
it relates $\|\log\hat{\Theta}_{T}-\log\Theta\|_{\ell_{2}}$ to $\| \hat
{\Theta}_{T}- \Theta\|_{\ell_{2}}$. It is due to \cite{Gil2012}.

Theorem \ref{thm main rate of convergence}(iii) gives a rate of convergence of
the minimum distance estimator $\hat{\theta}_{T}$. Note that $\theta^{0}$ are
log parameters of the member in the Kronecker product model, which is closest
to $\Theta$ in the sense discussed earlier. For sample correlation matrix
$\hat{\Theta}_{T}$, the rate of convergence of $\Vert\ve(\hat{\Theta}%
_{T}-\Theta)\Vert_{2}$ is $\sqrt{n^{2}/T}$ (square root of a sum of $n^{2}$
terms each of which has a convergence rate $1/T$). Thus the minimum distance
estimator $\hat{\theta}_{T}$ of the Kronecker product model converges faster
provided $\varpi\kappa(W)$ is not too large, in line with the principle of
dimension reduction. However, given that the dimension of $\theta^{0}$ is
$s=O(\log n)$, one would conjecture that the optimal rate of convergence for
$\hat{\theta}_{T}$ should be $\sqrt{\log n/T}$. In this sense, Theorem
\ref{thm main rate of convergence}(iii) does not demonstrate the full
advantages of a Kronecker product model. Because of the severe non-linearity
introduced by the matrix logarithm it is a challenging problem to prove a
faster rate of convergence for $\Vert\hat{\theta}_{T}-\theta^{0}\Vert_{2}$.

\subsection{Asymptotic Normality}

We define $y_{t}$'s natural filtration $\mathcal{F}_{t}:=\sigma(y_{t},
y_{t-1},\ldots, y_{1})$ and $\mathcal{F}_{0}=\{\emptyset,\emptyset^c\}$.

\begin{assumption}
\label{assu mds}

\item
\begin{enumerate}
[(i)]

\item Suppose that $\{y_{t}-\mu,\mathcal{F}_{t}\}$ is a martingale difference
sequence; that is $\mathbb{E}[y_{t}-\mu|\mathcal{F}_{t-1}]=0$ for all
$t=1,\ldots, T$.

\item Suppose that $\{y_{t}y_{t}^{\intercal}-\mathbb{E}[y_{t}y_{t}^{\intercal
}],\mathcal{F}_{t}\}$ is a martingale difference sequence; that is
\[
\mathbb{E}%
\sbr[1]{y_{t,i}y_{t,j}-\mathbb{E}[y_{t,i}y_{t,j}]|\mathcal{F}_{t-1}}=0
\]
for all $i,j=1,\ldots, n$, $t=1,\ldots, T$.
\end{enumerate}
\end{assumption}

Assumption \ref{assu mds} allows us to establish inference results within a
martingale framework. Outside this martingale framework, one encounters the
issue of long-run variance whenever one tries to get some inference result.
This is particularly challenging in the large dimensional case and we hence
shall not consider it in this article.

To derive the asymptotic normality of the minimum distance estimator, we
consider two cases

\begin{enumerate}
[(i)]

\item $\mu$ is unknown but $D$ is known;

\item both $\mu$ and $D$ are unknown.
\end{enumerate}

We will derive the asymptotic normality of the minimum distance estimator for
both these cases. We first comment on the plausibility or relevance of case
(i). We present five situations/arguments to show that case (i) is relevant
and these are by no means exhaustive. First, one could use higher frequency
data to estimate the individual variances and thereby utilise a very large
sample size. But that is not an option for estimating correlations because of
the non-synchronicity problem, which is acute in the large dimensional case
(\cite{parkhonglinton2016}). Second, one could have unbalanced low frequency
data meaning that each firm has a long time series but they start and finish
at different times such that the overlap, which is relevant for estimation of
correlations, can be quite a bit smaller. In that situation one might
standardise marginally using all the individual time series data and then
estimate pairwise correlations using the smaller overlapping data. Third, we
could have a global parametric model for $D$ and $\mu$, but a local (in time)
Kronecker product model for correlations, i.e., $\Theta(u)$ varies with
rescaled time $u=t/T$. In this situation, the initial estimation of $D$ and
$\mu$ can be done at a faster rate than estimation of the time varying
correlation $\Theta(u)$, so effectively $D$ and $\mu$ could be treated as
known quantities. Fourth, case (i) reflects our two-step estimation procedure
where variances are estimated first without imposing any model structure. This
is a common approach in dynamic volatility model estimation such as the DCC
model of \cite{englesheppard2001} and the GO-GARCH model
(\cite{vanderweide2002}). Indeed, in many of the early articles in that
literature standard errors for dynamic parameters of the correlation process
were constructed without regard to the effect of the initial procedure.
Finally, we note that theoretically estimation of $\mu$ and $D$ is well
understood even in the high dimensional case, so in keeping with much practice
in the literature we do not emphasize estimation of $\mu$ and $D$ again.

Define the following $n^{2}\times n^{2}$ dimensional matrix $H$:
\[
H:=\int_{0}^{1}[t(\Theta-I)+I]^{-1}\otimes\lbrack t(\Theta-I)+I]^{-1}dt.
\]
%\label{eqn def H}%
Define also the $n\times n$ and $n^{2}\times n^{2}$ matrices, respectively:
\begin{align}
&  \tilde{\Sigma}_{T}:=\frac{1}{T}\sum_{t=1}^{T}(y_{t}-\mu)(y_{t}%
-\mu)^{\intercal}.\label{align tilde Sigma}\\
&  V:=\text{var}%
\del[2]{\ve\sbr[1]{(y_{t}-\mu)(y_{t}-\mu )^{\intercal}}}\nonumber\\
&  = \mathbb{E}%
\sbr[1]{(y_{t}-\mu)(y_{t}-\mu)^{\intercal}\otimes (y_{t}-\mu) (y_{t}-\mu)^{\intercal}}-\mathbb{E}%
\sbr[1]{(y_{t}-\mu)\otimes (y_{t}-\mu)}\mathbb{E}%
\sbr[1]{(y_{t}-\mu)^{\intercal}\otimes (y_{t}-\mu)^{\intercal}}.\nonumber
\end{align}
%\[\tilde{\Sigma}_{T}:=\frac{1}{T}\sum_{t=1}^{T}(y_{t}-\mu)(y_{t}-\mu)^{\intercal}\qquadV:=\text{var}\del[2]{\ve\sbr[1]{(y_{t}-\mu)(y_{t}-\mu)^{\intercal}}}.\]
Since $x\mapsto(\lceil\frac{x}{n}\rceil,x-\lfloor\frac{x}{n}\rfloor n)$ is a
bijection from $\{1,\ldots,n^{2}\}$ to $\{1,\ldots,n\}\times\{1,\ldots,n\}$,
it is easy to show that the $(a,b)$th entry of $V$ is
\[
V_{a,b}\equiv V_{i,j,k,\ell}=\mathbb{E}[(y_{t,i}-\mu_{i})(y_{t,j}-\mu
_{j})(y_{t,k}-\mu_{k})(y_{t,\ell}-\mu_{\ell})]-\mathbb{E}[(y_{t,i}-\mu
_{i})(y_{t,j}-\mu_{j})]\mathbb{E}[(y_{t,k}-\mu_{k})(y_{t,\ell}-\mu_{\ell})],
\]
where $\mu_{i}=\mathbb{E}y_{t,i}$ (similarly for $\mu_{j},\mu_{k},\mu_{\ell}%
$), $a,b\in\{1,\ldots,n^{2}\}$ and $i,j,k,\ell\in\{1,\ldots,n\}$. In the
special case of normality, $V=2D_{n}D_{n}^{+}(\Sigma\otimes\Sigma)$
(\cite{magnusneudecker1986} Lemma 9).

\begin{assumption}
\label{assu mineval of V}
%\item \begin{enumerate}[(i)]
Suppose that $V$ is positive definite for all $n$, with its minimum eigenvalue
bounded away from zero by an absolute constant and maximum eigenvalue bounded
from above by an absolute constant.
%\end{enumerate}

\end{assumption}

Assumption \ref{assu mineval of V} is also a standard regularity condition. It
is automatically satisfied under normality given Assumptions
\ref{assu n indexed by T}(i) and \ref{assu about D and Dhat}(i) (via Lemma
\ref{lemmaD}(vi) in Appendix \ref{secArateofconvergence}). Assumption
\ref{assu mineval of V} could be relaxed to the case where the minimum
(maximum) eigenvalue of $V$ is slowly drifting towards zero (infinity) at
certain rate. The proofs for Theorem \ref{thm asymptotic normality} and
Theorem \ref{thm asymptotic normality MD when D is unknown} remain unchanged,
but this rate will need to be incorporated in Assumption
\ref{assu n indexed by T}(ii).

\subsubsection{When $\mu$ Is Unknown But $D$ Is Known}

In this case, $\hat{\Theta}_{T}$ simplifies into $\hat{\Theta}_{T,D}%
:=D^{-1/2}\hat{\Sigma}_{T} D^{-1/2}$. Similarly, the minimum distance
estimator $\hat{\theta}_{T}$ simplifies into $\hat{\theta}_{T,D}%
:=(E^{\intercal}WE)^{-1}E^{\intercal}W\vech (\log\hat{\Theta}_{T,D})$. Let
$\hat{H}_{T,D}$ denote the $n^{2}\times n^{2}$ matrix
\begin{equation}
\label{eqn hat H D}\hat{H}_{T,D}:=\int_{0}^{1}[t(\hat{\Theta}_{T,D}%
-I)+I]^{-1}\otimes\lbrack t(\hat{\Theta}_{T,D}-I)+I]^{-1}dt.
\end{equation}
Define $V$'s sample analogue $\hat{V}_{T}$ whose $(a,b)$th entry is
\begin{align*}
\hat{V}_{T,a,b}\equiv\hat{V}_{T,i,j,k,\ell}  &  :=\frac{1}{T}\sum_{t=1}%
^{T}(y_{t,i}-\bar{y}_{i})(y_{t,j}-\bar{y}_{j})(y_{t,k}-\bar{y}_{k})(y_{t,\ell
}-\bar{y}_{\ell})\\
&  \qquad
-\del[3]{ \frac{1}{T}\sum_{t=1}^{T}(y_{t,i}-\bar{y}_i)(y_{t,j}-\bar{y}_j)}\del[3]{ \frac{1}{T}\sum_{t=1}^{T}(y_{t,k}-\bar{y}_k)(y_{t,\ell}-\bar{y}_{\ell})},
\end{align*}
where $\bar{y}_{i}:=\frac{1}{T}\sum_{t=1}^{T}y_{t,i}$ (similarly for $\bar
{y}_{j}, \bar{y}_{k}$ and $\bar{y}_{\ell}$), $a,b\in\{1,\ldots,n^{2}\}$ and
$i,j,k,\ell\in\{1,\ldots,n\}$.

For any $c\in\mathbb{R}^{s}$ define the scalar
\[
c^{\intercal}J_{D}c:=c^{\intercal}(E^{\intercal}WE)^{-1}E^{\intercal}%
WD_{n}^{+}H(D^{-1/2}\otimes D^{-1/2})V(D^{-1/2}\otimes D^{-1/2})HD_{n}%
^{+^{\intercal}}WE(E^{\intercal}WE)^{-1}c.
\]
In the special case of normality, $c^{\intercal}J_{D}c$ could be simplified
into (see Example \ref{ex GD special case} in SM \ref{secSM.cor.cramerwold}
for details): $2c^{\intercal}(E^{\intercal}WE)^{-1}E^{\intercal}WD_{n}%
^{+}H(\Theta\otimes\Theta)HD_{n}^{+^{\intercal}}WE(E^{\intercal}WE)^{-1}c.$ We
also define the estimate $c^{\intercal}\hat{J}_{T,D}c$:
\begin{align*}
&  c^{\intercal}\hat{J}_{T,D}c :=c^{\intercal}(E^{\intercal}WE)^{-1}%
E^{\intercal}WD_{n}^{+}\hat{H}_{T,D}(D^{-1/2}\otimes D^{-1/2})\hat{V}%
_{T}(D^{-1/2}\otimes D^{-1/2})\hat{H}_{T,D}D_{n}^{+^{\intercal}}%
WE(E^{\intercal}WE)^{-1}c.
\end{align*}

\bigskip

\begin{thm}
\label{thm asymptotic normality} Let Assumptions \ref{assu subgaussian vector}%
(i), \ref{assu mixing}, \ref{assu n indexed by T}(ii),
\ref{assu about D and Dhat}, \ref{assu mds} and \ref{assu mineval of V} be
satisfied with $1/r_{1}+1/r_{2}>1$. In particular we set $r_{1}=2$. Then
\[
\frac{\sqrt{T}c^{\intercal}(\hat{\theta}_{T,D}-\theta^{0})}{\sqrt
{c^{\intercal}\hat{J}_{T,D}c}}\xrightarrow{d}N(0,1),
\]
for any $s \times1$ non-zero vector $c$ with $\|c\|_{2}=1$.
\end{thm}

\begin{proof}
See Appendix \ref{sec A4}.
\end{proof}

\bigskip

Theorem \ref{thm asymptotic normality} is a version of the large-dimensional
CLT, whose proof is mathematically non-trivial. To simplify the technicality,
we assume subgaussianity ($r_{1}=2$). Because the dimension of $\theta^{0}$ is
growing with the sample size, for a CLT to make sense, we need to transform
$\hat{\theta}_{T,D}-\theta^{0}$ to a univariate quantity by pre-multiplying
$c^{\intercal}$. The magnitudes of the elements of $c$ are not important, so
we normalize it to have unit Euclidean norm. What is important is whether the
elements of $c$ are zero or not. The components of $\hat{\theta}_{T,D}%
-\theta^{0}$ whose positions correspond to the non-zero elements of $c$ are
effectively entering the CLT.

We contribute to the literature on the large-dimensional CLT (see
\cite{huber1973}, \cite{yohaimaronna1979}, \cite{portnoy1985},
\cite{mammen1989}, \cite{welsh1989}, \cite{baiwu1994},
\cite{saikkonenlutkepohl1996} and \cite{heshao2000}). In this strand of
literature, a distinct feature is that the dimension of parameter, say,
$\theta^{0}$, is growing with the sample size, and at the same time we do not
impose sparsity on $\theta^{0}$. As a result, the rate of growth of dimension
of parameter has to be restricted by an assumption like Assumption
\ref{assu n indexed by T}(ii); in particular, the dimension of parameter
cannot exceed the sample size. Assumption \ref{assu n indexed by T}(ii)
necessarily requires $n^{4}/T\to0$. In \cite{lewisreinsel1985},
\cite{saikkonenlutkepohl1996}, \cite{changchenchen2015}, they require
$n^{3}/T\to0$ for establishment of a CLT for an $n$-dimensional parameter.
Hence there is much room of improvement for Assumption
\ref{assu n indexed by T}(ii) because the dimension of $\theta^{0}$ is
$s=O(\log n)$. The difficulty for this relaxation is again, as we had
mentioned when we discussed the rate of convergence of $\hat{\theta}_{T}$
(Theorem \ref{thm main rate of convergence}), due to the severe non-linearity
introduced by matrix logarithm. In this sense Assumption
\ref{assu n indexed by T}(ii) is only a sufficient condition; the same
reasoning applies to Assumption \ref{assu n indexed by T}(iii).

Our approach is different from the recent literature on high-dimensional
statistics such as Lasso, where one imposes sparsity on parameter to allow its
dimension to exceed the sample size.

We also give a corollary which allows us to test multiple hypotheses like
$H_{0}: A^{\intercal}\theta^{0} = a$.

\begin{corollary}
\label{cor cramerwold} Let Assumptions \ref{assu subgaussian vector}(i),
\ref{assu mixing}, \ref{assu n indexed by T}(ii), \ref{assu about D and Dhat},
\ref{assu mds} and \ref{assu mineval of V} be satisfied with $1/r_{1}%
+1/r_{2}>1$. In particular we set $r_{1}=2$. Given a full-column-rank $s
\times k$ matrix $A$ where $k$ is \textit{finite} with $\|A\|_{\ell_{2}%
}=O(\sqrt{\log n\cdot n\kappa(W)})$, we have
\[
\sqrt{T}(A^{\intercal}\hat{J}_{T,D}A)^{-1/2}A^{\intercal}(\hat{\theta}%
_{T,D}-\theta^{0})\xrightarrow{d}N\del [1]{0,I_k}.
\]

\end{corollary}

\begin{proof}
See SM \ref{secSM.cor.cramerwold}.
\end{proof}

\bigskip

Note that the condition $\|A\|_{\ell_{2}}=O(\sqrt{\log n \cdot n\kappa(W)})$
is trivial because the dimension of $A$ is only of order $O(\log n)\times
O(1)$. Moreover we can always rescale $A$ when carrying out hypothesis testing.

If one chooses the weighting matrix $W$ optimally, albeit infeasibly,
\[
W_{D, op}%
=\sbr[1]{D_{n}^{+}H(D^{-1/2}\otimes D^{-1/2})V(D^{-1/2}\otimes D^{-1/2})HD_{n}^{+\intercal}}^{-1}%
,
\]
the scalar $c^{\intercal}J_{D}c$ reduces to
\[
c^{\intercal}%
\del [2]{E^{\intercal}\sbr[1]{D_{n}^{+}H(D^{-1/2}\otimes D^{-1/2})V(D^{-1/2}\otimes D^{-1/2})HD_{n}^{+\intercal}}^{-1}E}^{-1}%
c.
\]
Under a further assumption of normality (i.e., $V=2D_{n}D_{n}^{+}%
(\Sigma\otimes\Sigma)$), the preceding display further simplifies to
\[
c^{\intercal}%
\del [3]{\frac{1}{2}E^{\intercal}D_n^{\intercal}H^{-1}(\Theta^{-1}\otimes \Theta^{-1})H^{-1}D_nE}^{-1}%
c,
\]
by Lemmas 11 and 14 of \cite{magnusneudecker1986}. We shall compare the
preceding display with the variance of the asymptotic distribution of the
one-step estimator in Section \ref{sec QMLE}.

\subsubsection{When Both $\mu$ and $D$ Are Unknown}

The case where both $\mu$ and $D$ are unknown is considerably more difficult.
If one simply recycles the proof for the case where only $\mu$ is unknown and
replaces $D$ with its plug-in estimator $\hat{D}_{T}$, it will not work.

Let $\hat{H}_{T}$ denote the $n^{2}\times n^{2}$ matrix
\begin{equation}
\label{eqn hat H}\hat{H}_{T}:=\int_{0}^{1}[t(\hat{\Theta}_{T}-I)+I]^{-1}%
\otimes\lbrack t(\hat{\Theta}_{T}-I)+I]^{-1}dt.
\end{equation}
Define the $n^{2}\times n^{2}$ matrix $P$:
\[
P:=I_{n^{2}}-D_{n}D_{n}^{+}(I_{n}\otimes\Theta)M_{d},\qquad M_{d}:=\sum
_{i=1}^{n}(F_{ii}\otimes F_{ii}),
\]
where $F_{ii}$ is an $n\times n$ matrix with one in its $(i,i)$th position and
zeros elsewhere. Matrix $M_{d}$ is an $n^{2}\times n^{2}$ diagonal matrix with
diagonal elements equal to 0 or 1; the positions of 1 in the diagonal of
$M_{d}$ correspond to the positions of diagonal entries of an arbitrary
$n\times n$ matrix $A$ in $\ve A$. Matrix $P$ first appeared in (4.6) of
\cite{neudeckerwesselman1990}. Note that for any correlation matrix $\Theta$,
matrix $P$ is an idempotent matrix of rank $n^{2}-n$ and has $n$ rows of
zeros. \cite{neudeckerwesselman1990} proved that
\[
\frac{\partial\ve\Theta}{\partial\ve\Sigma}=P(D^{-1/2}\otimes D^{-1/2});
\]
that is, the derivative $\frac{\partial\ve\Theta}{\partial\ve\Sigma}$ is a
function of $\Sigma$.

For any $c\in\mathbb{R}^{s}$ define the scalar $c^{\intercal}Jc$ and its
estimate $c^{\intercal}\hat{J}_{T}c$:%
\[
c^{\intercal}Jc:=c^{\intercal}(E^{\intercal}WE)^{-1}E^{\intercal}WD_{n}%
^{+}HP(D^{-1/2}\otimes D^{-1/2})V(D^{-1/2}\otimes D^{-1/2})P^{\intercal}%
HD_{n}^{+^{\intercal}}WE(E^{\intercal}WE)^{-1}c.
\]
\[
c^{\intercal}\hat{J}_{T}c:=c^{\intercal}(E^{\intercal}WE)^{-1}E^{\intercal
}WD_{n}^{+}\hat{H}_{T}\hat{P}_{T}(\hat{D}_{T}^{-1/2}\otimes\hat{D}_{T}%
^{-1/2})\hat{V}_{T}(\hat{D}_{T}^{-1/2}\otimes\hat{D}_{T}^{-1/2})\hat{P}%
_{T}^{\intercal}\hat{H}_{T}D_{n}^{+^{\intercal}}WE(E^{\intercal}WE)^{-1}c,
\]
where $\hat{P}_{T}:=I_{n^{2}}-D_{n}D_{n}^{+}(I_{n}\otimes\hat{\Theta}%
_{T})M_{d}.$

\begin{assumption}
\label{assu when D is unknown}

\item
\begin{enumerate}
[(i)]

\item For every positive constant $C$
\[
\sup_{\Sigma^{*}:\|\Sigma^{*}-\Sigma\|_{F}\leq C\sqrt{\frac{n^{2}}{T}}%
}%
\enVert[4]{\left. \frac{\partial \ve \Theta}{\partial \ve \Sigma}\right| _{\Sigma=\Sigma^*}-P(D^{-1/2}\otimes D^{-1/2})}_{\ell
_{2}}=O\del[3]{\sqrt{\frac{n}{T}}},
\]
where $\cdot|_{\Sigma=\Sigma^{*}}$ means "evaluate the argument $\Sigma$ at
$\Sigma^{*}$".

\item The $s\times s$ matrix
\[
E^{\intercal}WD_{n}^{+}HP(D^{-1/2}\otimes D^{-1/2})V(D^{-1/2}\otimes
D^{-1/2})P^{\intercal}HD_{n}^{+^{\intercal}}WE
\]
has full rank $s$ (i.e, being positive definite). Moreover,
\begin{align*}
&  \text{mineval}%
\del [2]{E^{\intercal}WD_{n}^{+}HP(D^{-1/2}\otimes D^{-1/2})V(D^{-1/2}\otimes D^{-1/2})P^{\intercal}HD_{n}^{+^{\intercal}}WE}\geq
\frac{n}{\varpi}\text{mineval}^{2}(W).
\end{align*}

\end{enumerate}
\end{assumption}

Assumption \ref{assu when D is unknown}(i) characterises some sort of uniform
rate of convergence in terms of spectral norm of the Jacobian matrix
$\frac{\partial\ve \Theta}{\partial\ve \Sigma}$. This type of assumption is
usually made when one wants to stop Taylor expansion, say, of $\ve \hat
{\Theta}_{T}$, at first order. If one goes into the second-order expansion (a
tedious route), Assumption \ref{assu when D is unknown}(i) can be completely
dropped at some expense of further restricting the relative growth rate
between $n$ and $T$. The radius of the shrinking neighbourhood $\sqrt{n^{2}%
/T}$ is determined by the rate of convergence in terms of the Frobenius norm
of the sample covariance matrix $\hat{\Sigma}_{T}$. The rate on the right side
of Assumption \ref{assu when D is unknown}(i) is chosen to be $\sqrt{n/T}$
because it is the rate of convergence of
\[
\enVert[4]{\left. \frac{\partial \ve \Theta}{\partial \ve \Sigma}\right| _{\Sigma=\hat{\Sigma}_T}-P(D^{-1/2}\otimes D^{-1/2})}_{\ell
_{2}}
\]
which could be easily deduced from the proof of Theorem
\ref{thm asymptotic normality MD when D is unknown}. This rate $\sqrt{n/T}$
could even be relaxed to $\sqrt{n^{2}/T}$ as the part of the proof of Theorem
\ref{thm asymptotic normality MD when D is unknown} which requires Assumption
\ref{assu when D is unknown}(i) is not the "binding" part of the whole proof.

We now examine Assumption \ref{assu when D is unknown}(ii). The $s\times s$
matrix
\[
E^{\intercal}WD_{n}^{+}HP(D^{-1/2}\otimes D^{-1/2})V(D^{-1/2}\otimes
D^{-1/2})P^{\intercal}HD_{n}^{+^{\intercal}}WE
\]
is symmetric and positive semidefinite. By Observation 7.1.8 of
\cite{hornjohnson2013}, its rank is equal to $\text{rank}(E^{\intercal}%
WD_{n}^{+}HP)$, if $(D^{-1/2}\otimes D^{-1/2})V(D^{-1/2}\otimes D^{-1/2})$ is
positive definite. In other words, Assumption \ref{assu when D is unknown}(ii)
is assuming $\text{rank}(E^{\intercal}WD_{n}^{+}HP)=s$, provided
$(D^{-1/2}\otimes D^{-1/2})V(D^{-1/2}\otimes D^{-1/2})$ is positive definite.
Even though $P$ has only rank $n^{2}-n$, in general the rank condition does
hold except in a special case. The special case is $\Theta=I_{n}$ and
$W=I_{n(n+1)/2}$. In this special case
\[
\text{rank}(E^{\intercal}WD_{n}^{+}HP)=\text{rank}(E^{\intercal}D_{n}%
^{+}P)=\sum_{j=1}^{v}\frac{n_{j}(n_{j}-1)}{2}<s.
\]
The second part of Assumption \ref{assu when D is unknown}(ii) postulates a
lower bound for its minimum eigenvalue. The rate $\text{mineval}^{2}(W)
n/\varpi$ is specified as such because of Assumption
\ref{assu about D and Dhat}(ii). Other magnitudes of the rate are also
possible as long as the proof of Theorem
\ref{thm asymptotic normality MD when D is unknown} goes through.

\begin{thm}
\label{thm asymptotic normality MD when D is unknown} Let Assumptions
\ref{assu subgaussian vector}(i), \ref{assu mixing}, \ref{assu n indexed by T}%
(ii), \ref{assu about D and Dhat}, \ref{assu mds}, \ref{assu mineval of V} and
\ref{assu when D is unknown} be satisfied with $1/r_{1}+1/r_{2}>1$. In
particular we set $r_{1}=2$. Then
\[
\frac{\sqrt{T}c^{\intercal}(\hat{\theta}_{T}-\theta^{0})}{\sqrt{c^{\intercal
}\hat{J}_{T}c}}\xrightarrow{d}N(0,1),
\]
for any $s \times1$ non-zero vector $c$ with $\|c\|_{2}=1$.
\end{thm}

\begin{proof}
See SM \ref{sec A6}.
\end{proof}

\bigskip

Again Theorem \ref{thm asymptotic normality MD when D is unknown} is a version
of the large-dimensional CLT, whose proof is mathematically non-trivial. It
has the same structure as that of Theorem \ref{thm asymptotic normality}.
However $c^{\intercal}\hat{J}_{T}c$ differs from $c^{\intercal}\hat{J}_{T,D}c$
reflecting the difference between $c^{\intercal}Jc$ and $c^{\intercal}J_{D}c$.
That is, the asymptotic distribution of the minimum distance estimator depends
on whether $D$ is known or not.

We also give a corollary which allows us to test multiple hypotheses like
$H_{0}: A^{\intercal}\theta^{0} = a$.

\begin{corollary}
Let Assumptions \ref{assu subgaussian vector}(i), \ref{assu mixing},
\ref{assu n indexed by T}(ii), \ref{assu about D and Dhat}, \ref{assu mds},
\ref{assu mineval of V} and \ref{assu when D is unknown} be satisfied with
$1/r_{1}+1/r_{2}>1$. In particular we set $r_{1}=2$. Given a full-column-rank
$s \times k$ matrix $A$ where $k$ is \textit{finite} with $\|A\|_{\ell_{2}%
}=O(\sqrt{\log^{2}n \cdot n\kappa^{2}(W)\varpi})$, we have
\[
\sqrt{T}(A^{\intercal}\hat{J}_{T}A)^{-1/2}A^{\intercal}(\hat{\theta}%
_{T}-\theta^{0})\xrightarrow{d}N\del [1]{0,I_k}.
\]

\end{corollary}

\begin{proof}
Essentially the same as that of Corollary \ref{cor cramerwold}.
\end{proof}

\bigskip

The condition $\|A\|_{\ell_{2}}=O(\sqrt{\log^{2}n \cdot n\kappa^{2}(W)\varpi
})$ is trivial because the dimension of $A$ is only of order $O(\log n)\times
O(1)$. Moreover we can always rescale $A$ when carrying out hypothesis
testing. In the case of both $\mu$ and $D$ unknown, the infeasible optimal
weighting matrix will be
\[
W_{op}%
=\sbr[1]{D_{n}^{+}HP(D^{-1/2}\otimes D^{-1/2})V(D^{-1/2}\otimes D^{-1/2})P^{\intercal}HD_{n}^{+\intercal}}^{-1}%
.
\]

\subsection{Specification Test}

We give a specification test (also known as an over-identification test) based
on the minimum distance objective function in (\ref{eqn MD objective function}%
). Suppose we want to test whether the Kronecker product model $\{\Theta
^{*}\}$ is correctly specified given the factorization $n=n_{1}\times
\cdots\times n_{v}$. That is,
\[
H_{0}: \Theta\in\{\Theta^{*}\} \quad(i.e., \vech( \log\Theta)=E\theta), \qquad
H_{1}:\Theta\notin\{\Theta^{*}\}.
\]
We first \textit{fix} $n$ (and hence $v$ and $s$). Recall
(\ref{eqn MD objective function}):
\begin{align*}
\hat{\theta}_{T}=\hat{\theta}_{T}(W)  &  :=\arg\min_{b\in\mathbb{R}^{s}
}[\vech (\log\hat{\Theta}_{T})-Eb]^{\intercal}W[\vech (\log\hat{\Theta}%
_{T})-Eb] =:\arg\min_{b\in\mathbb{R}^{s}}g_{T}(b)^{\intercal}Wg_{T}(b).
\end{align*}

\begin{thm}
\label{thm overidentification test fixed dim} Fix $n$ (and hence $v$ and $s$).

\begin{enumerate}
[(i)]

\item Suppose $\mu$ is unknown but $D$ is known. Let Assumptions
\ref{assu subgaussian vector}(i), \ref{assu mixing},
\ref{assu about D and Dhat}, \ref{assu mds} and \ref{assu mineval of V} be
satisfied with $1/r_{1}+1/r_{2}>1$. In particular we set $r_{1}=2$. Thus,
under $H_{0}$,
\begin{equation}
Tg_{T,D}(\hat{\theta}_{T,D})^{\intercal}\hat{S}_{T,D}^{-1}g_{T,D}(\hat{\theta
}_{T,D})\xrightarrow{d}\chi_{n(n+1)/2-s}^{2}, \label{eqn chi fix overiden}%
\end{equation}
where
\begin{align*}
g_{T,D}(b)  &  :=\vech(\log\hat{\Theta}_{T,D})-Eb\\
\hat{S}_{T,D}  &  :=D_{n}^{+}\hat{H}_{T,D}(D^{-1/2}\otimes D^{-1/2})\hat
{V}_{T}(D^{-1/2}\otimes D^{-1/2})\hat{H}_{T,D}D_{n}^{+\intercal}.
\end{align*}

\item Suppose both $\mu$ and $D$ are unknown. Let Assumptions
\ref{assu subgaussian vector}(i), \ref{assu mixing},
\ref{assu about D and Dhat}, \ref{assu mds}, \ref{assu mineval of V}, and
\ref{assu when D is unknown} be satisfied with $1/r_{1}+1/r_{2}>1$. In
particular we set $r_{1}=2$. Thus, under $H_{0}$,
\[
Tg_{T}(\hat{\theta}_{T})^{\intercal}\hat{S}_{T}^{-1}g_{T}(\hat{\theta}%
_{T})\xrightarrow{d}\chi_{n(n+1)/2-s}^{2},
\]
where
\[
\hat{S}_{T}:=D_{n}^{+}\hat{H}_{T}\hat{P}_{T}(\hat{D}_{T}^{-1/2}\otimes\hat
{D}_{T}^{-1/2})\hat{V}_{T}(\hat{D}_{T}^{-1/2}\otimes\hat{D}_{T}^{-1/2})\hat
{P}_{T}^{\intercal}\hat{H}_{T}D_{n}^{+^{\intercal}}.
\]

\end{enumerate}
\end{thm}

\begin{proof}
See SM \ref{sec A9}.
\end{proof}

\bigskip

Note that $\hat{S}_{T,D}^{-1}$ and $\hat{S}_{T}^{-1}$ are the feasible
versions of optimal weighting matrices $W_{D,op}$ and $W_{op}$, respectively.
From Theorem \ref{thm overidentification test fixed dim}, we can easily get
the following result of the diagonal path asymptotics, which is more general
than the sequential asymptotics but less general than the joint asymptotics
(see \cite{phillipsmoon1999}).

\begin{corollary}
\label{coro diagonal asymptotics}

\item
\begin{enumerate}
[(i)]

\item Suppose $\mu$ is unknown but $D$ is known. Let Assumptions
\ref{assu subgaussian vector}(i), \ref{assu mixing},
\ref{assu about D and Dhat}, \ref{assu mds} and \ref{assu mineval of V} be
satisfied with $1/r_{1}+1/r_{2}>1$. In particular we set $r_{1}=2$. Under
$H_{0}$,
\[
\frac{Tg_{T,n,D}(\hat{\theta}_{T,n,D})^{\intercal}\hat{S}^{-1}_{T,n,D}%
g_{T,n,D}(\hat{\theta}_{T,n,D})-\sbr[1]{\frac{n(n+1)}{2}-s}}%
{\sbr[1]{n(n+1)-2s}^{1/2}}\xrightarrow{d} N(0,1),
\]
where $n=n_{T}$ as $T\to\infty$.

\item Suppose both $\mu$ and $D$ are unknown. Let Assumptions
\ref{assu subgaussian vector}(i), \ref{assu mixing},
\ref{assu about D and Dhat}, \ref{assu mds}, \ref{assu mineval of V}, and
\ref{assu when D is unknown} be satisfied with $1/r_{1}+1/r_{2}>1$. In
particular we set $r_{1}=2$. Under $H_{0}$,
\[
\frac{Tg_{T,n}(\hat{\theta}_{T,n})^{\intercal}\hat{S}^{-1}_{T,n}g_{T,n}%
(\hat{\theta}_{T,n})-\sbr[1]{\frac{n(n+1)}{2}-s}}{\sbr[1]{n(n+1)-2s}^{1/2}%
}\xrightarrow{d} N(0,1),
\]
where $n=n_{T}$ as $T\to\infty$.
\end{enumerate}
\end{corollary}

\begin{proof}
See SM \ref{sec A9}.
\end{proof}

\section{QMLE and One-Step Estimator}

\label{sec QMLE}

\subsection{QMLE}

\label{sec QMLE original parameter}

In the context of Gaussian quasi-maximum likelihood estimation (QMLE), given a
factorization $n=n_{1}\times\cdots\times n_{v}$, we shall additionally assume
that the Kronecker product model $\{\Theta^{\ast}\}$ is correctly specified
(i.e. $\vech(\log\Theta)=E\theta$). Let $\rho\in\lbrack-1,1]^{s_{\rho}}$ be
the original parameters of some member of the Kronecker product model; we have
mentioned that $s_{\rho}=\sum_{j=1}^{v}n_{j}(n_{j}-1)/2$. Given Assumption
\ref{assu mds}, the log likelihood function in terms of original parameters
$\rho$ for a sample $\{y_{1},y_{2},\ldots,y_{T}\}$ is given by
\begin{align}
&  \ell_{T}(\mu,D,\rho) =-\frac{Tn}{2}\log(2\pi)-\frac{T}{2}\log\left\vert
D^{1/2}\Theta(\rho)D^{1/2}\right\vert -\frac{1}{2}\sum_{t=1}^{T}(y_{t}%
-\mu)^{\intercal}D^{-1/2}\Theta(\rho)^{-1}D^{-1/2}(y_{t}-\mu).
\label{align likelihood rho general}%
\end{align}
Write $\Omega=\Omega(\theta):=\log\Theta.$ Given Assumption \ref{assu mds},
the log likelihood function in terms of log parameters $\theta$ for a sample
$\{y_{1},y_{2},\ldots,y_{T}\}$ is given by
\begin{align}
&  \ell_{T}(\mu,D,\theta)\nonumber\\
&  =-\frac{Tn}{2}\log(2\pi)-\frac{T}{2}\log\left\vert D^{1/2}\exp
(\Omega(\theta))D^{1/2}\right\vert -\frac{1}{2}\sum_{t=1}^{T}(y_{t}%
-\mu)^{\intercal}D^{-1/2}[\exp(\Omega(\theta))]^{-1}D^{-1/2}(y_{t}-\mu).
\label{align likelihood theta general}%
\end{align}

In practice, conditional on some estimates of $\mu$ and $D$, we use an
iterative algorithm based on the derivatives of $\ell_{T}$ with respect to
either $\rho$ or $\theta$ to compute the QMLE of either $\rho$ or $\theta$.
Theorem \ref{prop Haihan score functions and second derivatives} below
provides formulas for the derivatives of $\ell_{T}$ with respect to $\theta$.
The computations required are typically not too onerous, since for example the
Hessian matrix is of an order $\log n$ by $\log n$. See
\cite{singullahmad2012} and \cite{ohlsonahmadavonrosen2013} for a discussion
of estimation algorithms in the case where the data are multiway array and $v$
is of low dimension. Nevertheless since there is quite complicated
non-linearity involved in the definition of the QMLE, it is not so easy to
directly analyse QMLE.

Instead we shall consider a one-step estimator that uses the minimum distance
estimator in Section \ref{sec MD estimation} to provide a starting value and
then takes a Newton-Raphson step towards the QMLE of $\theta$. In the fixed
$n$ case it is known that the one-step estimator is equivalent to the QMLE in
the sense that it shares its asymptotic distribution (\cite{bickel1975}).

Below, for slightly abuse of notation, we shall use $\mu,D, \theta$ to denote
the true parameter (i.e., characterising the data generating process) as well
as the generic parameter of the likelihood function; we will be more specific
whenever any confusion is likely to arise.

\subsection{One-Step Estimator}

%It is a well-known result that for any choice of $\Sigma$ (i.e., $D$ and $\theta$), the QMLE for $\mu$ is $\bar{y}$. Substituting $\mu=\bar{y}$ into (\ref{l2}), we obtain the concentrated likelihood function
%\begin{align*}
%&\ell_{T}(D, \theta)=\\
%&-\frac{Tn}{2}\log(2\pi)-\frac{T}{2}\log\left\vert D^{1/2}\exp(\Omega
%(\theta))D^{1/2}\right\vert -\frac{1}{2}\sum_{t=1}^{T}(y_{t}-\bar{y})^{\intercal
%}D^{-1/2}[\exp(\Omega(\theta))]^{-1}D^{-1/2}(y_{t}-\bar{y}).
%\end{align*}
Here we only examine the one-step estimator when $\mu$ is unknown but $D$ is
known. When neither $\mu$ nor $D$ is known, one has to differentiate
(\ref{align likelihood theta general}) with respect to both $\theta$ and $D$.
The analysis becomes considerably more involved and we leave it for future
work. Suppose $D$ is known, the likelihood function
(\ref{align likelihood theta general}) reduces to
\begin{align}
&  \ell_{T,D}(\theta,\mu)=\nonumber\\
&  -\frac{Tn}{2}\log(2\pi)-\frac{T}{2}\log\left\vert D^{1/2}\exp(\Omega
(\theta))D^{1/2}\right\vert -\frac{1}{2}\sum_{t=1}^{T}(y_{t}-\mu)^{\intercal
}D^{-1/2}[\exp(\Omega(\theta))]^{-1}D^{-1/2}(y_{t}-\mu).
\label{align likelihood fun when D known}%
\end{align}
It is well-known that for any choice of $\Sigma$ (i.e., $D$ and $\theta$), the
QMLE for $\mu$ is $\bar{y}$. Hence we may define
\[
\hat{\theta}_{QMLE,D}=\arg\max_{\theta}\ell_{T,D}(\theta, \bar{y}).
\]

\bigskip

\begin{thm}
\label{prop Haihan score functions and second derivatives}
%Suppose Assumption \ref{assu mds} hold.

\item
\begin{enumerate}
[(i)]

\item The $s\times1$ score function of
(\ref{align likelihood fun when D known}) with respect to $\theta$ takes the
following form\footnote{The likelihood function
(\ref{align likelihood fun when D known}) implicitly assumes Assumption
\ref{assu mds} and positive definiteness of $\Theta$.}
\[
\frac{\partial\ell_{T,D}(\theta,\mu)}{\partial\theta^{\intercal}}=\frac{T}%
{2}E^{\intercal}D_{n}^{\intercal}%
\sbr[3]{\int_{0}^{1}e^{t\Omega}\otimes e^{(1-t)\Omega}dt}\ve\sbr[1]{e^{-\Omega}D^{-1/2}\tilde{\Sigma}_T D^{-1/2} e^{-\Omega}-e^{-\Omega}}
,
\]
where $\tilde{\Sigma}_{T}$ is defined in (\ref{align tilde Sigma}).

\item The $s\times s$ block of the Hessian matrix of
(\ref{align likelihood fun when D known}) corresponding to $\theta$ takes the
following form
\begin{align*}
&  \frac{\partial^{2} \ell_{T,D}(\theta,\mu)}{\partial\theta\partial
\theta^{\intercal}}=\\
&  -\frac{T}{4} E^{\intercal}D_{n}^{\intercal}\int_{0}^{1}\int_{0}%
^{1}%
\del[1]{ e^{-st\Omega}\otimes e^{-(1-s)t\Omega}A e^{-(1-t)\Omega}+e^{-(1-t)\Omega}Ae^{-(1-s)t\Omega}\otimes e^{-st\Omega}}ds\cdot
tdt D_{n} E\\
&  \quad-\frac{T}{4}E^{\intercal}D_{n}^{\intercal} \int_{0}^{1}\int_{0}^{1}
\del[1]{e^{-(1-s)t\Omega}Ae^{-(1-t)\Omega}\otimes e^{-st\Omega}+e^{-st\Omega}\otimes e^{-(1-t)\Omega}A e^{-(1-s)t\Omega} }ds\cdot
tdt D_{n} E
\end{align*}
where $A:= D^{-1/2}\tilde{\Sigma}_{T} D^{-1/2}$. Symmetry of $\frac
{\partial^{2} \ell_{T,D}(\theta,\mu)}{\partial\theta\partial\theta^{\intercal
}}$ is in an obvious way.

\item The negative normalized expected Hessian matrix evaluated at the true
parameter $\theta$ takes the following form
\begin{align}
\Upsilon_{D}  &  := \mathbb{E}%
\sbr[3]{ -\frac{1}{T}\frac{\partial^{2}\ell_{T,D}(\theta,\mu)}{\partial \theta\partial\theta^{\intercal}}}\nonumber\\
&  = \frac{1}{2} E^{\intercal}D_{n}^{\intercal}\int_{0}^{1}\int_{0}%
^{1}%
\del[1]{ e^{-st\Omega}\otimes e^{st\Omega}+ e^{st\Omega}\otimes e^{-st\Omega}}ds\cdot
tdt D_{n} E\label{align Hessian form 1}\\
&  = \frac{1}{2}E^{\intercal}D_{n}^{\intercal}\Psi
\del [1]{e^{-\Omega}\otimes e^{-\Omega}}\Psi D_{n}E,
\label{align Hessian form 2}%
\end{align}
where $\Psi:=\int_{0}^{1}e^{t\Omega}\otimes e^{(1-t)\Omega}dt$.

\item Under normality (i.e., $V=2D_{n}D_{n}^{+}(\Sigma\otimes\Sigma)$), we
have the well-known relation
\[
\Upsilon_{D}=\mathbb{E}%
\sbr[3]{\frac{1}{T}\frac{\partial\ell_{T,D}(\theta,\mu)}{\partial\theta^{\intercal}}\frac{\partial\ell_{T,D}(\theta,\mu)}{\partial\theta} }.
\]

\end{enumerate}
\end{thm}

\begin{proof}
See SM \ref{sec A7}.
\end{proof}

\bigskip

We hence propose the following one-step estimator in the spirit of
\cite{vandervaart1998} p72 or \cite{newey1994} p2150:
\begin{equation}
\tilde{\theta}_{T,D}:=\hat{\theta}_{T,D}-\frac{1}{T}\hat{\Upsilon}_{T,D}%
^{-1}\frac{\partial\ell_{T,D}(\hat{\theta}_{T,D},\bar{y})}{\partial
\theta^{\intercal}}, \label{1step}%
\end{equation}
where $\hat{\Upsilon}_{T,D}$ is a plug-in estimator of $\Upsilon_{D}$ and is
defined as $\frac{1}{2}E^{\intercal}D_{n}^{\intercal}%
\sbr[1]{\int_0^1\int_0^1 \hat{\Theta}_{T,D}^{t+s-1}\otimes \hat{\Theta}_{T,D}^{1-t-s}dtds}D_{n}%
E$ (We show in SM \ref{sec A8} that $\hat{\Upsilon}_{T,D}$ is invertible with
probability approaching 1.) We did not use the plain vanilla one-step
estimator because the Hessian matrix $\frac{\partial^{2}\ell_{T,D}(\theta
,\mu)}{\partial\theta\partial\theta^{\intercal}}$ is rather complicated to analyse.

\subsection{Large Sample Properties}

To provide the large sample theory for the one-step estimator $\tilde{\theta
}_{T,D}$, we make the following assumption.

\begin{assumption}
\label{assu uniform unverifiable condition} For every positive constant $M$
and uniformly in $b\in\mathbb{R}^{s}$ with $\|b\|_{2}=1$,
\[
\sup_{\theta^{*}: \|\theta^{*}-\theta\|_{2}\leq M\sqrt{\frac{n\varpi\kappa
(W)}{T}}}%
\envert[4]{\sqrt{T}b^{\intercal}\sbr[3]{\frac{1}{T}\frac{\partial \ell_{T,D}(\theta^*,\bar{y})}{\partial \theta^{\intercal}}-\frac{1}{T}\frac{\partial \ell_{T,D}(\theta,\bar{y})}{\partial \theta^{\intercal}}-\Upsilon_D(\theta^*-\theta)}}=o_{p}%
(1).
\]

\end{assumption}

\bigskip

Assumption \ref{assu uniform unverifiable condition} is one of the sufficient
conditions needed for the asymptotic normality of $\tilde{\theta}_{T,D}$
(Theorem \ref{thm one step estimator asymptotic normality}). This kind of
assumption is standard in the asymptotics of one-step estimators (see (5.44)
of \cite{vandervaart1998} p71) or of M-estimation (see (C3) of
\cite{heshao2000}). Assumption \ref{assu uniform unverifiable condition}
implies that $\frac{1}{T}\frac{\partial\ell_{T,D}(\theta,\bar{y})}%
{\partial\theta^{\intercal}}$ is differentiable at the true parameter $\theta
$, with derivative tending to $\Upsilon_{D}$ in probability. The radius of the
shrinking neighbourhood $\sqrt{n\varpi\kappa(W)/T}$ is determined by the rate
of convergence of any preliminary estimator, say, $\hat{\theta}_{T,D}$ in our
case.
%The uniform requirement of the shrinking neighbourhood could be relaxed using Le Cam's \textit{discretization} trick (see \cite{vandervaart1998} p72).
It is possible to relax the $o_{p}(1)$ on the right side of the display in
Assumption \ref{assu uniform unverifiable condition} to $o_{p}(\sqrt
{n/(\varpi^{2}\log n)})$ by examining the proof of Theorem
\ref{thm one step estimator asymptotic normality}.

\begin{thm}
\label{thm one step estimator asymptotic normality} Suppose that the Kronecker
product model $\{\Theta^{*}\}$ is correctly specified. Let Assumptions
\ref{assu subgaussian vector}(ii), \ref{assu mixing},
\ref{assu n indexed by T}(iii), \ref{assu about D and Dhat}, \ref{assu mds},
and \ref{assu uniform unverifiable condition} be satisfied with $1/r_{1}%
+1/r_{2}>1$ and $r_{1}=2$. Then
\[
\frac{\sqrt{T}c^{\intercal}(\tilde{\theta}_{T,D}-\theta)}{\sqrt{c^{\intercal
}\hat{\Upsilon}_{T,D}^{-1}c}}\xrightarrow{d}N(0,1)
\]
for any $s\times1$ vector $c$ with $\|c\|_{2}=1$.
\end{thm}

\begin{proof}
See SM \ref{sec A8}.
\end{proof}

\bigskip

Theorem \ref{thm one step estimator asymptotic normality} is a version of the
large-dimensional CLT, whose proof is mathematically non-trivial. It has the
same structure as that of Theorem \ref{thm asymptotic normality} or Theorem
\ref{thm asymptotic normality MD when D is unknown}. Note that under
Assumption \ref{assu subgaussian vector}(ii), the QMLE is actually the maximum
likelihood estimator (MLE). If we replace normality (Assumption
\ref{assu subgaussian vector}(ii)) with the subgaussian assumption (Assumption
\ref{assu subgaussian vector}(i) with $r_{1}=2$) - that is the Gaussian
likelihood is not correctly specified - although the norm consistency of
$\tilde{\theta}_{T,D}$ should still hold, the asymptotic variance in Theorem
\ref{thm one step estimator asymptotic normality} needs to be changed to have
a sandwich formula. Theorem \ref{thm one step estimator asymptotic normality}
says that $\sqrt{T}c^{\intercal}(\tilde{\theta}_{T,D}-\theta
)\xrightarrow{d}N\del[1]{0,c^{\intercal}\del[1]{\mathbb{E}\sbr[1]{-\frac{1}{T}\frac{\partial^2\ell_{T,D}(\theta, \mu)}{\partial \theta \partial \theta^{\intercal}}}}^{-1}c}$%
. In the fixed $n$ case, this estimator achieves the parametric efficiency
bound by recognising a well-known result $\frac{\partial^{2}\ell_{T,D}(\theta,
\mu)}{\partial\mu\partial\theta^{\intercal}}=0$. This shows that our one-step
estimator $\tilde{\theta}_{T,D}$ is efficient when $D$ (the variances) is known.

By recognising that $H^{-1}=\int_{0}^{1}e^{t\log\Theta}\otimes e^{(1-t)\log
\Theta}dt=\Psi$ (see Lemma \ref{prop H inverse} in SM
\ref{secSM.cor.cramerwold}), we see that, when $D$ is known, under normality
and correct specification of the Kronecker product model, $\tilde{\theta
}_{T,D}$ and the optimal minimum distance estimator $\hat{\theta}_{T,D}(W_{D,
op})$ have the same asymptotic variance, i.e.,
$\del [1]{\frac{1}{2}E^{\intercal}D_n^{\intercal}H^{-1}(\Theta^{-1}\otimes \Theta^{-1})H^{-1}D_nE}^{-1}%
$.

We also give the following corollary which allows us to test multiple
hypotheses like $H_{0}: A^{\intercal}\theta= a$.

\begin{corollary}
Suppose the Kronecker product model $\{\Theta^{*}\}$ is correctly specified.
Let Assumptions \ref{assu subgaussian vector}(ii), \ref{assu mixing},
\ref{assu n indexed by T}(iii), \ref{assu about D and Dhat}, \ref{assu mds},
and \ref{assu uniform unverifiable condition} be satisfied with $1/r_{1}%
+1/r_{2}>1$ and $r_{1}=2$. Given a full-column-rank $s \times k$ matrix $A$
where $k$ is \textit{finite} with $\|A\|_{\ell_{2}}=O(\sqrt{\log n\cdot n})$,
we have
\[
\sqrt{T}(A^{\intercal}\hat{\Upsilon}_{T,D}^{-1}A)^{-1/2}A^{\intercal}%
(\tilde{\theta}_{T,D}-\theta)\xrightarrow{d}N\del [1]{0,I_k}.
\]

\end{corollary}

\begin{proof}
Essentially the same as that of Corollary \ref{cor cramerwold}.
\end{proof}

\bigskip

The condition $\|A\|_{\ell_{2}}=O(\sqrt{\log n\cdot n})$ is trivial because
the dimension of $A$ is only of order $O(\log n)\times O(1)$. Moreover we can
always rescale $A$ when carrying out hypothesis testing.

\section{Model Selection}

\label{sec modelselection}

We discuss the issue of model selection here. One shall not worry about this
if the data are in the multi-index format with $v$ multiplicative factors.
This is because in this setting a Kronecker product model is pinned down by
the structure of multiway arrays - the Kronecker product model is correctly
specified. This issue will pop up when one uses Kronecker product models to
approximate a general covariance or correlation matrix - all Kronecker product
models are then misspecified. The rest of discussions in this section will be
based on this approximation framework.

First, if one permutes the data, the performance of a given Kronecker product
model is likely to change. However, based on our experience, the performance
of a Kronecker product model is not that sensitive to the ordering of the
data. We will illustrate this in the empirical study. Moreover, usually one
fixes the ordering of the data before considering the issue of covariance
matrix estimation. Thus, Kronecker product models have a second-mover
advantage: the choice of a Kronecker product model depends on the ordering of
the data.

%\footnote{It is interesting to note that for particular functions of the covariance matrix, the ordering of the data does not matter. For example, the minimum variance portfolio (MVP) weights only depend on the covariance matrix through the row sums of its inverse, i.e., $\Sigma^{-1}\iota_{n}$, where $\iota_{n}$ is a vector of ones. If a Kronecker structure is imposed on $\Sigma$, then its inverse has the same structure. If the Kronecker factors are $(2\times2)$ and all variances are identical, then the row sums of $\Sigma^{-1}$ are the same, leading to equal weights for the MVP: $w=(1/n)\iota_{n}$, and this is irrespective of the ordering of the data.}   %This issue may matter in some applications, and begs the question of how one chooses the correct permutation.

Second, if one fixes the ordering of the data as well as a factorization
$n=n_{1}\times\cdots\times n_{v}$, but permutes $\Theta_{j}^{*}$s, one obtains
a different $\Theta^{*}$ (i.e., a different Kronecker product model). Although
the eigenvalues of these two Kronecker product models are the same, the
eigenvectors of them are not.

Third, if one fixes the ordering of the data, but uses a different
factorization of $n$, one also obtains a different Kronecker product model.
Suppose that $n$ has the prime factorization $n=p_{1}\times p_{2}\times
\cdots\times p_{v}$ for some positive integer $v$ ($v\geq2$) and primes
$p_{j}$ for $j=1,\ldots,v$. Then there exist several different Kronecker
product models, each of which is indexed by the dimensions of the
sub-matrices. The baseline model has dimensions $(p_{1},p_{2},\ldots,p_{v})$, but
there are many possible aggregations of this, for example,
$\del [1]{(p_{1}\times p_{2}),\ldots,(p_{v-1}\times p_{v})}$.
%We may index the induced models by the dimensions $m_{1},\ldots,m_{v}$ (where some could be zero dimensions), which are subject to the constraint that $\sum_{j=1}^{v}m_{j}=n$ and $m_{j}=\sum_{i=1}^{v}\pi_{ji}p_{i}$ with $\pi_{ji}\in\{0,1\}.$ Let the total number of free parameters be $q=\sum_{j=1}^{v}(m_{j}+1)m_{j}/2$ (minus identification restrictions). This includes the base model and the unrestricted $n\times n$ model as special cases.

To address the second and third issues, we might choose among Kronecker
product models using some model selection criterion which penalizes models
with more parameters. For example, we may define the Bayesian Information
Criterion (BIC) in terms of the original parameters $\rho$:
\[
BIC(\rho)=-\frac{2}{T}\ell_{T}(\mu, D, \rho)+\frac{\log T}{T}s_{\rho},
\]
where $\ell_{T}$ is the log likelihood function defined in
(\ref{align likelihood rho general}), and $s_{\rho}$ is the dimension of
$\rho$. We seek the Kronecker product model with the minimum preceding
display. Typically there are not so many factorizations to consider, so this
is not too computationally burdensome.

\section{Monte Carlo Simulations and an Application}

\label{sec simu}

In this section, we first provide a set of Monte Carlo simulations that
evaluate the performance of the QMLE and MD estimator, and then give a small
application of our Kronecker product model to daily stock returns.

\subsection{Monte Carlo Simulations}

We simulate $T$ random vectors $y_{t}$ of dimension $n$ according to
\begin{align}
y_{t} = \Sigma^{1/2} z_{t},\qquad z_{t} \sim N(0,I_{n})\qquad\Sigma=\Sigma
_{1}\otimes\Sigma_{2}\otimes\cdots\otimes\Sigma_{v}, \label{align.simu.xt}%
\end{align}
where $n=2^{v}$ and $v\in\mathbb{N}$. That is, the sub-matrices $\Sigma_{i}$
are $2\times2$ for $i=1,\ldots,v$. These sub-matrices $\Sigma_{j}$ are
generated with unit variances and off-diagonal elements drawn
\textit{randomly} from a uniform distribution on $(-1,1)$. This ensures
positive definiteness of $\Sigma$. Note that we have two sources of randomness
in this data generating process: random innovations ($z_{t}$) and random
off-diagonal elements of the $\Sigma_{i}$ for $i=1,\ldots,v$. Due to the unit
variances, $\Sigma$ is also the correlation matrix $\Theta$ of $y_{t}$, but
the econometrician is unaware of this: He applies a Kronecker product model to
the correlation matrix $\Theta$. We consider the correctly specified case,
i.e., the Kronecker product model has a factorization $n=2^{v}$. The sample
size is set to $T=300$ while we vary $v$ (hence $n$). We set the Monte Carlo
simulations to 1000.

We shall consider the QMLE and MD estimator. For the QMLE, we estimate the
original parameters $\rho$ and obtain an estimator for $\Theta$ (and hence
$\Sigma$) directly. Recalling (\ref{align likelihood rho general}), we could
use $\ell_{T}(\bar{y}, \hat{D}_{T}, \rho)$ to optimise $\rho$. For the MD
estimator, we estimate the log parameters $\theta^{0}$ via formula
(\ref{eqn thetaTW closed form solution}), obtain an estimator for $\log\Theta
$, and finally obtain an estimator for $\Theta$ (and hence $\Sigma$) via
matrix exponential. In the MD case, we need to specify a choice of the
weighting matrix $W$. Given its sheer dimension ($n(n+1)/2\times n(n+1)/2$),
any non-sparse $W$ will be a huge computational burden in terms of memory for
the MD estimator. Hence we consider two diagonal weighting matrices
\[
W_{1}=I_{n(n+1)/2},\qquad W_{2} =
\sbr[2]{D_n^+\del[1]{\hat{D}_T\otimes \hat{D}_T}D_n^{+\intercal}}^{-1}.
\]
In the latter case, the MD estimator is inversely weighted by the sample
variances. Weighting matrix $W_{2}$ resembles, but is not the same as, a
feasible version of the optimal weighting matrix $W_{op}$. The choice of
$W_{2}$ is based on heuristics. In an unreported simulation, we also consider
the optimally weighted MD estimator. The optimally weighted MD estimator is
extremely computationally intensive and its finite sample performance is not
as good as those weighted by $W_{1}$ or $W_{2}$. This is probably because a
data-driven, large-dimensional weighting matrix introduces additional sizeable
estimation errors in small samples - such a phenomenon has been well
documented in the GMM framework by \cite{andersensorensen1996}.

We compare our estimators with \cite{ledoitwolf2017}'s direct nonlinear
shrinkage estimator (the LW2017 estimator hereafter).\footnote{The Matlab code
for the direct nonlinear shrinkage estimator is downloaded from the website of
Professor Michael Wolf from the Department of Economics at the University of
Zurich. We are grateful for this.}

Given a generic estimator $\tilde{\Sigma}$ of the covariance matrix $\Sigma$
and in each simulation, we can compute
\[
1-\frac{\|\tilde{\Sigma}-\Sigma\|_{F}^{2}}{\|\hat{\Sigma}_{T}-\Sigma\|_{F}%
^{2}}.
\]
The median of the preceding display is calculated among all the simulations
and denoted RI in terms of $\Sigma$. Criterion RI is closely related to the
\textit{percentage relative improvement in average loss} (PRIAL) criterion in
\cite{ledoitwolf2004}.\footnote{It is defined as
\[
\text{PRIAL}=1-\frac{\mathbb{E}\|\tilde{\Sigma}-\Sigma\|_{F}^{2}}%
{\mathbb{E}\|\hat{\Sigma}_{T}-\Sigma\|_{F}^{2}}.
\]
} As PRIAL, RI measures the performance of the estimator $\tilde{\Sigma}$ with
respect to the sample covariance estimator $\hat{\Sigma}_{T}$. Note that RI
$\in(-\infty,1]$: A negative value means $\tilde{\Sigma}$ performs worse than
$\hat{\Sigma}_{T}$ while a positive value means otherwise. RI is more robust
to outliers than PRIAL.

Often an estimator of the precision matrix $\Sigma^{-1}$ is of more interest
than that of $\Sigma$ itself, so we also compute RI for the inverse covariance
matrix; that is, we compute the median of
\[
1-\frac{\|\tilde{\Sigma}^{-1}-\Sigma^{-1}\|_{F}^{2}}{\|\hat{\Sigma}_{T}%
^{-1}-\Sigma^{-1}\|_{F}^{2}}.
\]
Note that this requires invertibility of the sample covariance matrix
$\hat{\Sigma}_{T}$ and therefore can only be calculated for $n<T$.

Our final criterion is the minimum variance portfolio (MVP) constructed from
an estimator of the covariance matrix. The weights of the minimum variance
portfolio are given by
\begin{equation}
\label{eqn minimum variance portfolio weights}w_{MVP}=\frac{\Sigma^{-1}%
\iota_{n}}{\iota_{n}^{^{\intercal}}\Sigma^{-1}\iota_{n}},
\end{equation}
where $\iota_{n}=(1,1,\ldots,1)^{^{\intercal}}$ is of dimension $n$ (see
\cite{ledoitwolf2003}, \cite{chankarceskilakonishok1999} etc).
%The inverse of a Kronecker product model is easily found by inverting the sub-matrices $\Sigma_{j}^{*},$ which can be done analytically, since
%\[(\Sigma^{*})^{-1}=(\Sigma_{1}^{*})^{-1}\otimes(\Sigma_{2}^{*})^{-1}\otimes\cdots\otimes(\Sigma_{v}^{*})^{-1}.\]
%In fact, because $\iota_{n}=\iota_{n_{1}}\otimes\iota_{n_{2}}\otimes \cdots\otimes\iota_{n_{v}},$ we can write%
%\begin{align*}
%w_{MV}  &  =\frac
%{\del[1]{(\Sigma_{1}^*)^{-1}\otimes(\Sigma_{2}^*)^{-1}\otimes\cdots\otimes (\Sigma_{v}^*)^{-1}}
%\iota_{n}}{\iota_{n}^{\intercal}%
%\del[1]{(\Sigma_{1}^*)^{-1}\otimes(\Sigma_{2}^*)^{-1}\otimes\cdots\otimes (\Sigma_{v}^*)^{-1}}
%\iota_{n}}\\
%& =\frac{(\Sigma_{1}^{*})^{-1}\iota_{n_{1}}}{\iota_{n_{1}}^{\intercal}%
%(\Sigma_{1}^{*})^{-1}\iota_{n_{1}}}\otimes\frac{(\Sigma_{2}^{*})^{-1}%
%\iota_{n_{2}}}{\iota_{n_{2}}^{\intercal}(\Sigma_{2}^{*})^{-1}\iota_{n_{2}}%
%}\otimes\cdots\otimes\frac{(\Sigma_{v}^{*})^{-1}\iota_{n_{v}}}{\iota_{n_{v}%
%}^{\intercal}(\Sigma_{v}^{*})^{-1}\iota_{n_{v}}},\\
%\mathrm{var}(w_{MV}^{^{\intercal}}y_{t})  &  =\frac{1}{\iota_{n_{1}%
%}^{\intercal}(\Sigma_{1}^{*})^{-1}\iota_{n_{1}}\times\cdots\times\iota_{n_{v}%
%}^{\intercal}(\Sigma_{v}^{*})^{-1}\iota_{n_{v}}}.
%\end{align*}
%In cases where $n$ is large, this structure is very convenient computationally. See \cite{fanliaoshi2015} for a discussion of risk estimation for large dimensional portfolio choice problems.
The first MVP weights are constructed using the sample covariance matrix
$\hat{\Sigma}_{T}$ while the second MVP weights are constructed using a
generic estimator of $\tilde{\Sigma}$. These two minimum variance portfolios
are then evaluated by calculating their standard deviations in the
out-of-sample data ($y_{t}$) generated using the same mechanism. The
out-of-sample size is set to $T^{\prime}=21$. The ratio of the standard
deviation of the minimum variance portfolio constructed from $\tilde{\Sigma}$
over that of the minimum variance portfolio constructed from $\hat{\Sigma}%
_{T}$ is calculated. We report its median (VR) over Monte Carlo simulations.
Note that VR $\in[0,+\infty)$: A value greater than one means $\tilde{\Sigma}$
performs worse than $\hat{\Sigma}_{T}$ while a value less than one means otherwise.

Table \ref{table true kronecker} reports RI-1 (RI in terms of $\Sigma$), RI-2
(RI in terms of $\Sigma^{-1}$) and VR for various $n$. We observe the
following patterns. First, we see that all our estimators QMLE, MD1, MD2
outperform the sample covariance matrix in all dimensional cases including
both the small-dimensional cases (e.g., $n=4$) and the large-dimensional cases
(e.g., $n=256$). Note that in the large dimensional case like $n=256, T=300$,
the ratio $n/T$ is close to 1 - a case not really covered by Assumption
\ref{assu n indexed by T}. This perhaps illustrates that Assumption
\ref{assu n indexed by T} is a sufficient but not necessary condition for
theoretical analysis of our proposed methodology. Second, such a phenomenon
holds in terms of RI-1, RI-2 and VR. The superiority of our estimators over
the sample covariance matrix increases when $n/T$ increases. Third, the QMLE
outperforms the MD estimators whenever $n/T$ is close to one, while the
opposite holds when $n/T$ is small. Fourth, the LW2017 estimator also beats
the sample covariance matrix but its RI-1 margin is thin. This is perhaps not
surprising as the LW2017 estimator does not utilise the Kronecker product
structure of the data generating process. Overall, the QMLE is the best
estimator in this baseline setting.

\begin{table}[ptb]
\centering
%centering the table%
\begin{tabular}
[c]{ccccccccc}%
\toprule & $n$ & 4 & 8 & 16 & 32 & 64 & 128 & 256\\
\midrule \multirow{4}{*}{ RI-1} & QMLE & 0.227 & 0.529 & 0.714 & 0.820 &
0.892 & 0.929 & 0.950\\
& MD1 & 0.345 & 0.632 & 0.789 & 0.862 & 0.897 & 0.909 & 0.618\\
& MD2 & 0.339 & 0.631 & 0.785 & 0.858 & 0.896 & 0.908 & 0.616\\
& LW2017 & 0.020 & 0.027 & 0.046 & 0.063 & 0.087 & 0.106 & 0.127\\
\midrule \multirow{4}{*}{RI-2} & QMLE & 0.323 & 0.615 & 0.805 & 0.914 &
0.973 & 0.995 & 1.000\\
& MD1 & 0.354 & 0.632 & 0.771 & 0.752 & 0.665 & 0.588 & 0.837\\
& MD2 & 0.344 & 0.643 & 0.790 & 0.796 & 0.714 & 0.628 & 0.846\\
& LW2017 & 0.136 & 0.181 & 0.235 & 0.351 & 0.521 & 0.756 & 0.991\\
\midrule \multirow{4}{*}{VR} & QMLE & 0.999 & 0.995 & 0.980 & 0.953 & 0.899 &
0.770 & 0.389\\
& MD1 & 0.999 & 0.993 & 0.979 & 0.953 & 0.900 & 0.774 & 0.401\\
& MD2 & 0.999 & 0.993 & 0.979 & 0.954 & 0.899 & 0.774 & 0.400\\
& LW2017 & 1.000 & 0.999 & 0.998 & 0.993 & 0.975 & 0.912 & 0.544\\
\bottomrule &  &  &  &  &  &  &  &
\end{tabular}
\caption{{\protect\small The baseline setting. QMLE, MD1, MD2 and LW2017 stand
for the quasi-maximum likelihood estimator of the Kronecker product model, the
minimum distance estimator (weighted by $W_{1}$) of the Kronecker product
model, the minimum distance estimator (weighted by $W_{2}$) of the Kronecker
product model, and the \cite{ledoitwolf2017}'s direct nonlinear shrinkage
estimator, respectively. RI-1 and RI-2 are RI criteria in terms of $\Sigma$
and $\Sigma^{-1}$, respectively. VR is the median of the ratio of the standard
deviation of the MVP using the estimator to that using the sample covariance
matrix out of sample. The sample size is fixed at $T=300$.}}%
\label{table true kronecker}%
\end{table}

As robustness checks, we consider two modifications of our baseline data
generating process:

\begin{enumerate}
[(i)]

\item Time series $y_{t}$ is still generated as in (\ref{align.simu.xt}) but
the actual data are $w_{t}$:
\begin{align*}
w_{1}  &  = y_{1}\\
w_{t}  &  = a_{w} w_{t-1}+\sqrt{1-a_{w}^{2}}y_{t},\qquad t=2,\ldots, T.
\end{align*}
The parameter $a_{w}$ is set to be 0.5 to capture the temporal dependence.

\item Same as modification (i), but $y_{t}$ is drawn from a multivariate $t$
distribution of 5 degrees of freedom with $\Sigma$ as its correlation matrix.
\end{enumerate}

In modification (i), $w_{t}$ is serially correlated given any non-zero
autoregressive scalar $a_{w}$ but its covariance matrix is still $\Sigma$. A
choice of $a_{w}=0.5$ is consistent with Assumption \ref{assu mixing}. Our
simulation results are reasonably robust to the choice of $a_{w}$. In
modification (ii), in addition to the serial dependence, we add heavy-tailed
features to the data which might be a better reflection of reality.
Heavy-tailed data are not covered by Assumption \ref{assu subgaussian vector},
so this modification serves as a robustness check for our theoretical findings.

The results of modification (i) are reported in Table
\ref{table modification i}. Those four observations we made from the baseline
setting (Table \ref{table true kronecker}) still hold when we relax the
independence assumption of the data. Modification (ii) are reported in Table
\ref{table modification ii}. When we switch on both temporal dependence and
heavy tails, all estimators - ours and the LW2017 estimator - are adversely
affected to a certain extent. In particular, in terms of RI-2, both the QMLE
and LW2017 estimators fare worse than the sample covariance matrix in small
dimensions. Overall, the identity weighted MD estimator is the best estimator
in modification (ii). That the MD estimator trumps the QMLE in heavy-tailed
data is intuitive because the MD estimator is derived not based on a
particular distributional assumption.

\begin{table}[ptb]
\centering
%centering the table%
\begin{tabular}
[c]{ccccccccc}%
\toprule & $n$ & 4 & 8 & 16 & 32 & 64 & 128 & 256\\
\midrule \multirow{4}{*}{ RI-1} & QMLE & 0.219 & 0.514 & 0.712 & 0.821 &
0.887 & 0.928 & 0.951\\
& MD1 & 0.321 & 0.611 & 0.775 & 0.849 & 0.880 & 0.889 & 0.798\\
& MD2 & 0.310 & 0.611 & 0.770 & 0.844 & 0.877 & 0.890 & 0.796\\
& LW2017 & 0.025 & 0.032 & 0.049 & 0.065 & 0.093 & 0.117 & 0.155\\
\midrule \multirow{4}{*}{RI-2} & QMLE & 0.320 & 0.654 & 0.824 & 0.932 &
0.980 & 0.997 & 1.000\\
& MD1 & 0.338 & 0.639 & 0.737 & 0.691 & 0.593 & 0.517 & 0.822\\
& MD2 & 0.347 & 0.652 & 0.775 & 0.753 & 0.657 & 0.571 & 0.839\\
& LW2017 & 0.220 & 0.292 & 0.429 & 0.634 & 0.818 & 0.939 & 0.997\\
\midrule \multirow{4}{*}{VR} & QMLE & 0.998 & 0.988 & 0.975 & 0.927 & 0.860 &
0.728 & 0.383\\
& MD1 & 0.997 & 0.987 & 0.973 & 0.925 & 0.862 & 0.733 & 0.406\\
& MD2 & 0.997 & 0.987 & 0.973 & 0.924 & 0.862 & 0.732 & 0.406\\
& LW2017 & 1.000 & 0.999 & 0.997 & 0.990 & 0.970 & 0.907 & 0.568\\
\bottomrule &  &  &  &  &  &  &  &
\end{tabular}
\caption{{\protect\small Modification (i). QMLE, MD1, MD2 and LW2017 stand for
the quasi-maximum likelihood estimator of the Kronecker product model, the
minimum distance estimator (weighted by $W_{1}$) of the Kronecker product
model, the minimum distance estimator (weighted by $W_{2}$) of the Kronecker
product model, and the \cite{ledoitwolf2017}'s direct nonlinear shrinkage
estimator, respectively. RI-1 and RI-2 are RI criteria in terms of $\Sigma$
and $\Sigma^{-1}$, respectively. VR is the median of the ratio of the standard
deviation of the MVP using the estimator to that using the sample covariance
matrix out of sample. The sample size is fixed at $T=300$.}}%
\label{table modification i}%
\end{table}

\begin{table}[ptb]
\centering
%centering the table%
\begin{tabular}
[c]{ccccccccc}%
\toprule & $n$ & 4 & 8 & 16 & 32 & 64 & 128 & 256\\
\midrule \multirow{4}{*}{ RI-1} & QMLE & 0.021 & 0.105 & 0.203 & 0.320 &
0.442 & 0.564 & 0.690\\
& MD1 & 0.071 & 0.211 & 0.348 & 0.492 & 0.621 & 0.719 & 0.605\\
& MD2 & 0.084 & 0.242 & 0.378 & 0.510 & 0.626 & 0.712 & 0.581\\
& LW2017 & $-0.023$ & $-0.001$ & 0.029 & 0.069 & 0.118 & 0.158 & 0.220\\
\midrule \multirow{4}{*}{RI-2} & QMLE & $-0.035$ & $-0.139$ & $-0.357$ &
$-0.831$ & $-0.202$ & 0.867 & 0.999\\
& MD1 & 0.006 & 0.035 & 0.111 & 0.385 & 0.896 & 0.636 & 0.829\\
& MD2 & $-0.006$ & $-0.009$ & 0.032 & 0.255 & 0.894 & 0.724 & 0.854\\
& LW2017 & $-0.103$ & $-0.206$ & $-0.428$ & $-0.847$ & $-0.279$ & 0.825 &
0.997\\
\midrule \multirow{4}{*}{VR} & QMLE & 0.996 & 0.982 & 0.956 & 0.923 & 0.842 &
0.708 & 0.379\\
& MD1 & 0.994 & 0.982 & 0.955 & 0.921 & 0.840 & 0.720 & 0.432\\
& MD2 & 0.994 & 0.982 & 0.955 & 0.920 & 0.840 & 0.719 & 0.429\\
& LW2017 & 1.000 & 0.999 & 0.995 & 0.989 & 0.968 & 0.906 & 0.577\\
\bottomrule &  &  &  &  &  &  &  &
\end{tabular}
\caption{{\protect\small Modification (ii). QMLE, MD1, MD2 and LW2017 stand
for the quasi-maximum likelihood estimator of the Kronecker product model, the
minimum distance estimator (weighted by $W_{1}$) of the Kronecker product
model, the minimum distance estimator (weighted by $W_{2}$) of the Kronecker
product model, and the \cite{ledoitwolf2017}'s direct nonlinear shrinkage
estimator, respectively. RI-1 and RI-2 are RI criteria in terms of $\Sigma$
and $\Sigma^{-1}$, respectively. VR is the median of the ratio of the standard
deviation of the MVP using the estimator to that using the sample covariance
matrix out of sample. The sample size is fixed at $T=300$.}}%
\label{table modification ii}%
\end{table}

\subsection{An Application}

We now consider estimation of the covariance matrix of $n^{\prime}=441$ stock
returns ($y_{t}$) in the S\&P 500 index. We have daily observations from
January 3, 2005 to November 6, 2015. The number of trading days is $T=2732$.
Since the underlying data might not have a multiplicative structure giving
rise to a Kronecker product - or if they do but we are unaware of it - a
Kronecker product model in this application is inherently misspecified. In
other words, we are exploiting Kronecker product models' approximating feature
to a general covariance matrix.

We have proved in Appendix \ref{sec A2} that in a given Kronecker product
model there exists a member which is closest to the true covariance matrix.
However, in order for this closest "distance" to be small, the chosen
Kronecker product model needs to be versatile enough to capture various data
patterns. In this sense, a parsimonious model, say, $441=3\times3 \times7
\times7$, is likely to be inferior to a less parsimonious model, say,
$441=21\times21$.

We add an $3\times1$ dimensional pseudo random vector $z_{t}$ which is $N(0,
I_{3})$ distributed and independent over $t$. The dimension of the final
system is $n=441+3=444$. Again we fit Kronecker product models to the
correlation matrix of the final system and recover an estimator for the
covariance matrix of the final system via left and right multiplication of the
estimated correlation matrix of the final system by $\hat{D}_{T}^{1/2}$. Last,
we extract the $441\times441$ upper-left block of the estimated covariance
matrix of the final system to form our Kronecker product estimator of the
covariance matrix of $y_{t}$. The dimension of the added pseudo random vector
should not be too large to avoid introducing additional noise, which could
adversely affect the performance of the Kronecker product models. We choose
the dimension of the final system to be 444 because its prime factorization is
$2\times2 \times3 \times37$, and we experiment with several Kronecker product
models. We did try other dimensions for the final system and the pattern
discussed below remains generally the same.

As we are considering less parsimonious models, the QMLE is computationally
intensive and found to perform worse than the MD estimator in preliminary
investigations, so we only use the MD estimator. The MD estimator is extremely
fast because its formula is just (\ref{eqn thetaTW closed form solution}). We
choose the weighting matrix to be the identity matrix.

%Kronecker product models are fitted to the correlation matrix $\Theta=D^{-1/2}\Sigma D^{-1/2}$, where $D$ is the diagonal matrix containing the variances of $y_{t}$. The first Kronecker model (M1) uses the factorization $2^{9}=512$ and assumes that
%\[\Theta^{*}=\Theta_{1}^{*}\otimes\Theta_{2}^{*}\otimes\cdots\otimes\Theta_{9}^{*},\]
%where $\Theta_{j}^{*}$ are $2\times2$ correlation matrices. We add a vector of 71 independent pseudo variables $u_{t}\sim N\left(  0, I_{71}\right)  $ such that $n+71=2^{9}$, and then extract the upper left $(n\times n)$ block of $\Theta^{*}$ to obtain the correlation matrix of $y_{t}$.

%Again we adopt the first approach of estimation to estimate original parameters directly. The estimation is done in two steps: First, $D$ is estimated using the sample variances, and then the original parameters of $\Theta^{*}$ are estimated by quasi-maximum likelihood estimation using the standardized returns $\hat{D}^{-1/2}y_{t}$ and pseudo variables $u_{t}$. Re-ordering the data $y_{t}$ according to variance in a descending way prior to adding the pseudo variables $u_{t}$ did not improve the final outcomes, so we keep the original order of the data. We experiment with more generous decompositions by looking at all factorizations of numbers from 441 to 512, and selecting some yielding not more than 30 parameters. Table \ref{table application} gives a summary of these models. %The Schwarz information criterion favours the specification of model M6 with 27 parameters.

\begin{table}[ptb]
\centering
\begin{tabular}
[c]{lcccccc}%
\toprule & MD & MD & MD & MD & \multirow{2}{*}{LW2004} &
\multirow{2}{*}{LW2017}\\
& $(2\times2 \times3 \times37)$ & $(4\times111)$ & $(3\times148)$ &
$(2\times222)$ &  & \\
\midrule & \multicolumn{6}{c}{\textit{original ordering of the data}}\\
\textit{Impr} & 0.265 & 0.379 & 0.394 & 0.440 & 0.459 & 0.518\\
\textit{Prop} & 0.811 & 0.896 & 0.915 & 0.953 & 0.991 & 0.981\\
\cmidrule(lr){2-7} & \multicolumn{6}{c}{\textit{a random permutation of the
data}}\\
\textit{Impr} & 0.259 & 0.364 & 0.404 & 0.431 & 0.459 & 0.518\\
\textit{Prop} & 0.811 & 0.887 & 0.915 & 0.943 & 0.991 & 0.981\\
\cmidrule(lr){2-7} & \multicolumn{6}{c}{\textit{a random permutation of the
data}}\\
\textit{Impr} & 0.263 & 0.351 & 0.366 & 0.436 & 0.459 & 0.518\\
\textit{Prop} & 0.811 & 0.887 & 0.906 & 0.943 & 0.991 & 0.981\\
\bottomrule &  &  &  &  &  &
\end{tabular}
\caption{{\protect\small MD, LW2004 and LW2017 stand for the (identity matrix
weighted) minimum distance estimators of the Kronecker product models
(factorisations given in parentheses), the \cite{ledoitwolf2004}'s linear
shrinkage estimator, and the \cite{ledoitwolf2017}'s direct nonlinear
shrinkage estimator, respectively. \textit{Impr} is the median of the 106
quantities calculated based on (\ref{eqn application evaluation formula}) and
\textit{Prop} is the proportion of the times (out of 106) that a competitor
MVP outperforms the sample covariance MVP (i.e., the proportion of the times
when (\ref{eqn application evaluation formula}) is positive). A random
permutation of the data means that the order of the 441 stocks is randomly
reshuffled.}}%
\label{table application}%
\end{table}

We follow the approach of \cite{fanliaomincheva2013} and estimate our model on
windows of size $504$ days (equal to two years' trading days) that are shifted
from the beginning to the end of the sample. The Kronecker product estimator
of the covariance matrix of $y_{t}$ is used to construct the minimum variance
portfolio (MVP) weights as in (\ref{eqn minimum variance portfolio weights}).
We also compute the MVP weights using the sample covariance matrix of $y_{t}$.
These two minimum variance portfolios are then evaluated using the next 21
days (equal to one month's trading days) out-of-sample. In particular, we
calculate
\begin{equation}
\label{eqn application evaluation formula}1-\frac{\text{sd(a competitor MVP)}%
}{\text{sd(sample covariance MVP)}},
\end{equation}
where sd$(\cdot)$ computes standard deviation. Then the estimation window of
$504$ days is shifted forward by 21 days. This procedure is repeated until we
reach the end of the sample; the total number of out-of-sample evaluations is
106. We consider two evaluation criteria of performance: \textit{Impr} and
\textit{Prop}. \textit{Impr} is the median of the 106 quantities calculated
based on (\ref{eqn application evaluation formula}). Note that \textit{Impr}
$\in(-\infty, 1]$: A negative value means a competitor MVP performs worse than
the sample covariance MVP while a positive value means otherwise.
\textit{Prop} is the proportion of the times (out of 106) that a competitor
MVP outperforms the sample covariance MVP (i.e., the proportion of the times
when (\ref{eqn application evaluation formula}) is positive).

For comparison, we consider \cite{ledoitwolf2004}'s linear shrinkage estimator
and \cite{ledoitwolf2017}'s direct nonlinear shrinkage estimator. The results
are reported in Table \ref{table application}. We first use the original
ordering of the data, i.e. alphabetical, and have the following observations.
First, all the Kronecker product MVPs outperform the sample covariance MVP.
Second, as we move from the most parsimonious factorisation ($444=2\times2
\times3 \times37$) to the least parsimonious factorisation ($444=2\times222$),
the performance of Kronecker product MVPs monotonically improves. This is
intuitive: Since we are using Kronecker product models to approximate a
general covariance matrix, a more flexible Kronecker product model could fit
the data better. There is no over-fitting at least in this application as we
consider out-of-sample evaluation. Third, the performance of the
$(2\times222)$ Kronecker product MVP is very close to that of a sophisticated
estimator like \cite{ledoitwolf2004}'s linear shrinkage estimator. This is
commendable because here a Kronecker product model is a misspecified
parametric model for a general covariance matrix while the linear shrinkage
estimator is in essence a data-driven, nonparametric estimator.

We next randomly reshuffle the order of the 441 stocks twice and use the same
Kronecker product models. In these two cases, the rows and columns of the true
covariance matrix also get reshuffled. We see that the performances of those
Kronecker product models are marginally affected by the reshuffle.
\cite{ledoitwolf2004}'s and \cite{ledoitwolf2017}'s shrinkage estimators are,
as expected, not affected by the ordering of the data.

%The last four columns of Table \ref{table application} summarize the relative performance of the Kronecker MVP with respect to those of the sample covariance matrix and the linear shrinkage estimator of \cite{ledoitwolf2004}. All models outperform the sample covariance matrix, while models with smaller dimensional sub-matrices (i.e., M1, M3 and M5) tend to outperform the shrinkage estimator. The reason could be that it is more difficult to ensure positive definiteness of a bigger sub-matrix in the constrained maximum likelihood optimisation.

%All models outperform the sample covariance matrix, while only the more generous factorizations also outperform the SFM. Comparing the results with Table 6 of \cite{fanliaomincheva2013} for similar data, it appears that the performance of the favored model M6 is quite close to their POET estimator. So our estimator may provide an alternative to high dimensional covariance modelling.

\section{Conclusions}

\label{sec conclusion}

We have established the large sample properties of estimators of Kronecker
product models in the large dimensional case. In particular, we obtained norm
consistency and the large dimensional CLTs for the MD and one-step estimators.
Kronecker product models outperform the sample covariance matrix
theoretically, in Monte Carlo simulations, and in an application to portfolio
choice. When a Kronecker product model is correctly specified, Monte Carlo
simulations show that estimators of it can beat \cite{ledoitwolf2017}'s direct
non-linear shrinkage estimator. In the application, when one uses Kronecker
product models as an approximating device to a general covariance matrix, a
less parsimonious one can perform almost as good as \cite{ledoitwolf2004}'s
linear shrinkage estimator. It is possible to extend the framework in various
directions to improve performance.
%For example, one may study cases (c) and (d) in Table \ref{table v and nj}.

A final motivation for the Kronecker product structure is that it can be used
as a component of a super model consisting of several components. For
instance, the idea of the decomposition in (\ref{eqn model}) could be applied
to components of \emph{dynamic} models such as multivariate GARCH, an area in
which Luc Bauwens has contributed significantly over the recent years, see
also his highly cited review paper \cite{bauwensetal2006}. For example, the
dynamic conditional correlation (DCC) model of \cite{engle2002}, or the BEKK
model of \cite{englekroner1995} both have intercept matrices that are required
to be positive definite and suffer from the curse of dimensionality, for which
model (\ref{eqn model}) would be helpful. Also, parameter matrices associated
with the dynamic terms in the model could be equipped with a Kronecker
product, similar to a suggestion by \cite{hoff2015} for vector autoregressions.

\appendix
%To indicate that following sections are to be numbered as appendices

\section{Appendix}

This appendix is organised as follows: Appendix \ref{sec A.1} further
discusses this matrix $E$ of the minimum distance estimator in Section
\ref{sec MD estimation}. Appendix \ref{sec A2} shows that a Kronecker product
model has a best approximation to a general covariance or correlation matrix.
Appendix \ref{secArateofconvergence} and \ref{sec A4} contain proofs of
Theorem \ref{thm main rate of convergence} and of Theorem
\ref{thm asymptotic normality}, respectively. Appendix \ref{sec oldappendixB}
contains auxiliary lemmas used in various places of this appendix.

\subsection{Matrix $E$}

\label{sec A.1}

The proof of the following theorem gives a concrete formula for the matrix $E$
of the minimum distance estimator.

\begin{thm}
\label{prop kronecker formula} Suppose that
\[
\Theta^{*}=\Theta^{*}_{1}\otimes\Theta^{*}_{2}\otimes\cdots\otimes\Theta
^{*}_{v},
\]
where $\Theta^{*}_{j}$ is $n_{j}\times n_{j}$ dimensional such that
$n=n_{1}\times n_{2}\times\cdots\times n_{v}$. Taking the logarithm on both
sides gives
\begin{align*}
\log\Theta^{*}  &  = \log\Theta^{*}_{1}\otimes I_{n_{2}}\otimes\cdots\otimes
I_{n_{v}}+ I_{n_{1}} \otimes\log\Theta^{*}_{2}\otimes I_{n_{3}}\otimes
\cdots\otimes I_{n_{v}}+\cdots+I_{n_{1}}\otimes I_{n_{2}}\otimes\cdots
\otimes\log\Theta^{*}_{v}.
\end{align*}
For identification we set the first diagonal entry of $\log\Theta^{*}_{j}$ to
be 0 for $j=1,\ldots,v-1$. In total there are
\[
s:=\sum_{j=1}^{v}\frac{n_{j}(n_{j}+1)}{2}-(v-1)
\]
unrestricted log parameters; let $\theta^{*}\in\mathbb{R}^{s}$ denote these.
Then there exists a $n(n+1)/2\times s$ full column rank matrix $E$ such that
\[
\vech (\log\Theta^{*})=E\theta^{*}.
\]

\end{thm}

\begin{proof}
Note that
\begin{align*}
\ve (\log \Theta^*) &= \ve(\log \Theta^*_1\otimes I_{n_2}\otimes \cdots \otimes I_{n_v})+\ve( I_{n_1} \otimes \log \Theta^*_2\otimes I_{n_3}\otimes \cdots \otimes I_{n_v})+\cdots\\
&\qquad +\ve (I_{n_1}\otimes I_{n_2}\otimes \cdots \otimes \log \Theta^*_v).
\end{align*}
If
\[\ve( I_{n_1} \otimes \log \Theta^*_i\otimes I_{n_3}\otimes \cdots \otimes I_{n_v})=E_i \vech (\log \Theta_i^*)\]
for some $n^2\times n_i(n_i+1)/2$ matrix $E_i$ for $i=1,\ldots,v$, then we have
\begin{align*}
\vech (\log \Theta^*)=D_n^+\ve (\log \Theta^*)  &  =D_n^+\left[
\begin{array}
[c]{cccc}%
E_{1} & E_{2} & \cdots & E_{v}%
\end{array}
\right]  \left[
\begin{array}
[c]{c}%
\vech ( \log \Theta^*_1)\\
\vech (\log \Theta^*_2)\\
\vdots\\
\vech (\log \Theta^*_v)
\end{array}
\right].
\end{align*}
For identification we set the first diagonal entry of $\log\Theta^{*}_{j}$ to
be 0 for $j=1,\ldots,v-1$. In total there are
\[
s:=\sum_{j=1}^{v}\frac{n_{j}(n_{j}+1)}{2}-(v-1)
\]
(identifiable) log parameters;
let $\theta^{*}\in \mathbb{R}^s$ denote these. Then there exists a $n(n+1)/2\times s$ full
column rank matrix $E$ such that
\[\vech (\log\Theta^{*})=E\theta^{*},\]
where
\[E:= D_n^+\left[ \begin{array}{ccccc}
E_{1,(-1)} & E_{2,(-1)} & \cdots & E_{v-1,(-1)} & E_v
\end{array}\right] \]
and $E_{i,(-1)}$ stands for matrix $E_i$ with its first column removed.
We now determine the formula for $E_i$. We first consider $\ve (\log \Theta^*_1\otimes I_{n_2}\otimes \cdots \otimes I_{n_v})$.
\begin{align*}
&\ve (\log \Theta^*_1\otimes I_{n_2}\otimes \cdots \otimes I_{n_v})=\ve (\log \Theta^*_1\otimes I_{n/n_1})=\del[1]{I_{n_1}\otimes K_{n/n_1,n_1}\otimes I_{n/n_1}}\del[1]{\ve (\log \Theta^*_1)\otimes \ve I_{n/n_1}} \\
&= \del[1]{I_{n_1}\otimes K_{n/n_1,n_1}\otimes I_{n/n_1}}\del[1]{I_{n_1^2}\ve (\log \Theta^*_1)\otimes \ve I_{n/n_1}\cdot 1}\\
&=\del[1]{I_{n_1}\otimes K_{n/n_1,n_1}\otimes I_{n/n_1}}\del[1]{I_{n_1^2}\otimes \ve I_{n/n_1}}\ve (\log \Theta^*_1)\\
&=\del[1]{I_{n_1}\otimes K_{n/n_1,n_1}\otimes I_{n/n_1}}\del[1]{I_{n_1^2}\otimes \ve I_{n/n_1}}D_{n_1}\vech (\log \Theta^*_1),
\end{align*}
where the second equality is due to \cite{magnusneudecker2007} Theorem 3.10 p55. Thus,
\[E_1:=\del[1]{I_{n_1}\otimes K_{n/n_1,n_1}\otimes I_{n/n_1}}\del[1]{I_{n_1^2}\otimes \ve I_{n/n_1}}D_{n_1}.\]
We now consider $\ve (I_{n_1}\otimes \cdots \otimes \log \Theta^*_i\otimes\cdots\otimes I_{n_v})$.
\begin{align*}
&\ve (I_{n_1}\otimes \cdots \otimes \log \Theta^*_i\otimes\cdots\otimes I_{n_v})=\ve \sbr[2]{K_{n_1\cdots n_{i-1},n/(n_1\cdots n_{i-1})}\del[1]{\log \Theta_i^*\otimes I_{n/n_i}}K_{n/(n_1\cdots n_{i-1}),n_1\cdots n_{i-1}}}\\
&= \sbr[1]{K_{n/(n_1\cdots n_{i-1}),n_1\cdots n_{i-1}}^{\intercal}\otimes K_{n_1\cdots n_{i-1},n/(n_1\cdots n_{i-1})}}\ve \del[1]{\log \Theta_i^*\otimes I_{n/n_i}}\\
&=\sbr[1]{K_{n_1\cdots n_{i-1},n/(n_1\cdots n_{i-1})}\otimes K_{n_1\cdots n_{i-1},n/(n_1\cdots n_{i-1})}} \del[1]{I_{n_i}\otimes K_{n/n_i,n_i}\otimes I_{n/n_i}}\del[1]{I_{n_i^2}\otimes \ve I_{n/n_i}}D_{n_i}\vech (\log \Theta^*_i),
\end{align*}
where the first equality is due to the identity $B\otimes A=K_{p,m}(A\otimes B) K_{m,p}$ for $A$ ($m\times m$) and $B$ ($p\times p$). Thus
\[E_i:= \sbr[1]{K_{n_1\cdots n_{i-1},n/(n_1\cdots n_{i-1})}\otimes K_{n_1\cdots n_{i-1},n/(n_1\cdots n_{i-1})}} \del[1]{I_{n_i}\otimes K_{n/n_i,n_i}\otimes I_{n/n_i}}\del[1]{I_{n_i^2}\otimes \ve I_{n/n_i}}D_{n_i},\]
for $i=2,\ldots,v$.
\end{proof}

\bigskip

\begin{lemma}
\label{prop decode E} Given that $n=n_{1}\times n_{2}\times\cdots\times n_{v}%
$, the $s\times s$ matrix $E^{\intercal}E$ takes the following form:

\begin{enumerate}
[(i)]

\item For $i=1,\ldots,s$, the $i$th diagonal entry of $E^{\intercal}E$ records
how many times the $i$th parameter in $\theta^{*}$ has appeared in
$\vech (\log\Theta^{*})$. The value depends on to which $\log\Theta^{*}_{j}$
the $i$th parameter in $\theta^{*}$, $\theta^{*}_{i}$, belongs to. For
instance, suppose $\theta^{*}_{i}$ is a parameter belonging to $\log\Theta
^{*}_{3}$, then
\[
(E^{\intercal}E)_{i,i}=n/n_{3}.
\]

\item For $i,k=1,\ldots,s$ ($i\neq k$), the $(i,k)$ entry of $E^{\intercal}E$
(or the $(k,i)$ entry of $E^{\intercal}E$ by symmetry) records how many times
the $i$th parameter in $\theta^{*}$, $\theta^{*}_{i}$, and $k$th parameter in
$\theta^{*}$, $\theta^{*}_{k}$, have appeared together (as summands) in an
entry of $\vech (\log\Theta^{*})$. The value depends on to which $\log
\Theta^{*}_{j}$ the $i$th parameter in $\theta^{*}$, $\theta^{*}_{i}$, and
$k$th parameter in $\theta^{*}$, $\theta^{*}_{k}$, belong to. For instance,
suppose $\theta^{*}_{i}$ is a parameter belonging to $\log\Theta^{*}_{3}$ and
$\theta^{*}_{k}$ is a parameter belonging to $\log\Theta^{*}_{5}$, then
\[
(E^{\intercal}E)_{i,k}=(E^{\intercal}E)_{k,i}=n/(n_{3}\cdot n_{5}).
\]
However, the formula in the preceding display is overridden for the following
two cases. If both $\theta^{*}_{i}$ and $\theta^{*}_{k}$ belong to the same
$\log\Theta^{*}_{j}$, then $(E^{\intercal}E)_{i,k}=(E^{\intercal}E)_{k,i}=0$.
Also note that when $\theta^{*}_{i}$ is an off-diagonal entry of some
$\log\Theta^{*}_{j}$, then
\[
(E^{\intercal}E)_{i,k}=(E^{\intercal}E)_{k,i}=0
\]
for any $k=1,\ldots,s$ ($i\neq k$).
\end{enumerate}
\end{lemma}

\begin{proof}
Proof by spotting the pattern.
\end{proof}

\bigskip

We here give a concrete example to illustrate Lemma \ref{prop decode E}.

\begin{example}
Suppose that $n_{1}=3, n_{2}=2, n_{3}=2$. We have
\begin{align*}
\log\Theta^{*}_{1}=\left(
\begin{array}
[c]{ccc}%
0 & a_{1,2} & a_{1,3}\\
a_{1,2} & a_{2,2} & a_{2,3}\\
a_{1,3} & a_{2,3} & a_{3,3}%
\end{array}
\right)  \qquad\log\Theta^{*}_{2}=\left(
\begin{array}
[c]{cc}%
0 & b_{1,2}\\
b_{1,2} & b_{2,2}%
\end{array}
\right)  \qquad\log\Theta^{*}_{3}=\left(
\begin{array}
[c]{cc}%
c_{1,1} & c_{1,2}\\
c_{1,2} & c_{2,2}%
\end{array}
\right)
\end{align*}
The leading diagonals of $\log\Theta^{*}_{1}$ and $\log\Theta^{*}_{2}$ are set
to zero for identification as explained before. Thus
\[
\theta^{*}=(a_{1,2},a_{1,3},a_{2,2},a_{2,3},a_{3,3},b_{1,2},b_{2,2}%
,c_{1,1},c_{1,2},c_{2,2})^{\intercal}.
\]
Then we can invoke Lemma \ref{prop decode E} to write down $E^{\intercal}E$
without even using Matlab to compute $E$; that is,
\begin{align*}
E^{\intercal}E=\left(
\begin{array}
[c]{cccccccccc}%
4 & 0 & 0 & 0 & 0 & 0 & 0 & 0 & 0 & 0\\
0 & 4 & 0 & 0 & 0 & 0 & 0 & 0 & 0 & 0\\
0 & 0 & 4 & 0 & 0 & 0 & 2 & 2 & 0 & 2\\
0 & 0 & 0 & 4 & 0 & 0 & 0 & 0 & 0 & 0\\
0 & 0 & 0 & 0 & 4 & 0 & 2 & 2 & 0 & 2\\
0 & 0 & 0 & 0 & 0 & 6 & 0 & 0 & 0 & 0\\
0 & 0 & 2 & 0 & 2 & 0 & 6 & 3 & 0 & 3\\
0 & 0 & 2 & 0 & 2 & 0 & 3 & 6 & 0 & 0\\
0 & 0 & 0 & 0 & 0 & 0 & 0 & 0 & 6 & 0\\
0 & 0 & 2 & 0 & 2 & 0 & 3 & 0 & 0 & 6
\end{array}
\right)
\end{align*}

\end{example}

\subsection{Best Approximation}

\label{sec A2}

In this section of the appendix, we show that for any given $n\times n$ real
symmetric, positive definite covariance matrix (or correlation matrix), there
is a uniquely defined member in the Kronecker product model that is closest to
the covariance matrix (or correlation matrix) in some sense in terms of the
\textit{log parameter} space, once a factorization $n=n_{1}\times\cdots\times
n_{v}$ is specified.

Let $\mathcal{M}_{n}$ denote the set of all $n\times n$ real symmetric
matrices. For any $n(n+1)/2\times n(n+1)/2$ known, deterministic, positive
definite matrix $W$, define a map
\[
\langle A,B\rangle_{W}:=(\vech A)^{\intercal}W\vech B\qquad A,B\in
\mathcal{M}_{n}.
\]
It is easy to show that $\langle\cdot,\cdot\rangle_{W}$ is an inner product.
Space $\mathcal{M}_{n}$ with inner product $\langle\cdot,\cdot\rangle_{W}$ can
be identified by $\mathbb{R}^{n(n+1)/2}$ with the usual Euclidean inner
product. Moreover, since, for finite $n$, $\mathbb{R}^{n(n+1)/2}$ with the
usual Euclidean inner product is a Hilbert space, so is $\mathcal{M}_{n}$. The
inner product $\langle\cdot,\cdot\rangle_{W}$ induces the following norm
\[
\|A\|_{W}:=\sqrt{\langle A, A\rangle_{W}}=\sqrt{(\vech A)^{\intercal}%
W\vech A}.
\]

Let $\mathcal{D}_{n}$ denote the set of matrices of the form
\begin{align*}
\Omega_{1}\otimes I_{n_{1}}\otimes\cdots\otimes I_{n_{v}} +I_{n_{1}}%
\otimes\Omega_{2}\otimes\cdots\otimes I_{n_{v}}+\cdots+I_{n_{1}}\otimes
\cdots\otimes\Omega_{v},
\end{align*}
where $\Omega_{j}$ are $n_{j}\times n_{j}$ real symmetric matrices for
$j=1,\ldots,v$.
%The subset $\mathcal{C}_{n}\subset\mathcal{M}_{n}$ is not a subspace of
%$\mathcal{M}_{n}$. First, $\otimes$ and $+$ do not distribute in general. That
%is, there might not exist positive definite $\Sigma_{1,3}^{*}$ and
%$\Sigma_{2,3}^{*}$ such that
%\[
%\Sigma_{1,1}^{*}\otimes\Sigma_{2,1}^{*}+\Sigma_{1,2}^{*}\otimes\Sigma
%_{2,2}^{*}=\Sigma_{1,3}^{*}\otimes\Sigma_{2,3}^{*}.
%\]
%Second, $\mathcal{C}_{n}$ is a positive cone, hence not necessarily a
%subspace. In fact, the smallest subspace of $\mathcal{M}_{n}$ that contains
%$\mathcal{C}_{n}$ is $\mathcal{M}_{n}$ itself. On the other hand,
Note that $\mathcal{D}_{n}$ is a (linear) subspace of $\mathcal{M}_{n}$ as,
for $\alpha,\beta\in\mathbb{R}$,
\begin{align*}
&  \alpha
\del [1]{\Omega_1\otimes I_{n_1}\otimes \cdots \otimes I_{n_v} +I_{n_1}\otimes \Omega_2\otimes \cdots \otimes I_{n_v}+\cdots+I_{n_1}\otimes \cdots\otimes \Omega_v}+\\
&  \beta
\del [1]{\Xi_1\otimes I_{n_1}\otimes \cdots \otimes I_{n_v} +I_{n_1}\otimes \Xi_2\otimes \cdots \otimes I_{n_v}+\cdots+I_{n_1}\otimes \cdots\otimes \Xi_v}\\
&  =(\alpha\Omega_{1}+\beta\Xi_{1})\otimes I_{n_{1}}\otimes\cdots\otimes
I_{n_{v}} +I_{n_{1}}\otimes(\alpha\Omega_{2}+\beta\Xi_{2})\otimes\cdots\otimes
I_{n_{v}}+\cdots+I_{n_{1}}\otimes\cdots\otimes(\alpha\Omega_{v}+\beta\Xi
_{v})\\
&  \in\mathcal{D}_{n}.
\end{align*}
For finite $n$, $\mathcal{D}_{n}$ is also closed.

Consider a real symmetric, positive definite covariance matrix $\Sigma$. We
have $\log\Sigma\in\mathcal{M}_{n}$. By the projection theorem of the Hilbert
space, there exists a unique matrix $L^{0}\in\mathcal{D}_{n}$ such that
\[
\Vert\log\Sigma-L^{0}\Vert_{W}=\min_{L\in\mathcal{D}_{n}}\Vert\log
\Sigma-L\Vert_{W}.
\]
(Note also that $\log\Sigma^{-1}=-\log\Sigma,$ so that $-L^{0}$ simultaneously
approximates the precision matrix $\Sigma^{-1}$ in the same norm.)

This says that any real symmetric, positive definite covariance matrix
$\Sigma$ has a closest approximating matrix $\Sigma^{0}$ in a sense that
\[
\|\log\Sigma-\log\Sigma^{0}\|_{W}=\min_{L\in\mathcal{D}_{n}}\|\log
\Sigma-L\|_{W},
\]
where $\Sigma^{0}:=\exp L^{0}$. Since $L^{0}\in\mathcal{D}_{n}$, we can write
\begin{align*}
L^{0}=L^{0}_{1}\otimes I_{n_{1}}\otimes\cdots\otimes I_{n_{v}} +I_{n_{1}%
}\otimes L^{0}_{2}\otimes\cdots\otimes I_{n_{v}}+\cdots+I_{n_{1}}\otimes
\cdots\otimes L^{0}_{v},
\end{align*}
where $L^{0}_{j}$ are $n_{j}\times n_{j}$ real symmetric matrices for
$j=1,\ldots,v$. Then
\begin{align*}
&  \Sigma^{0}=\exp L^{0} =\exp
\sbr[1]{L^0_1\otimes I_{n_1}\otimes \cdots \otimes I_{n_v} +I_{n_1}\otimes L^0_2\otimes \cdots \otimes I_{n_v}+\cdots+I_{n_1}\otimes \cdots\otimes L^0_v}\\
&  =\exp\sbr[1]{L^0_1\otimes I_{n_1}\otimes \cdots \otimes I_{n_v}}\times
\exp\sbr[1]{I_{n_1}\otimes L^0_2\otimes \cdots \otimes I_{n_v}}\times
\cdots\times\exp\sbr[1]{I_{n_1}\otimes \cdots\otimes L^0_v}\\
&  = \sbr[1]{\exp L^0_1 \otimes I_{n_1}\otimes \cdots \otimes I_{n_v}}\times
\sbr[1]{I_{n_1}\otimes \exp L^0_2\otimes \cdots \otimes I_{n_v}}\times
\cdots\times\sbr[1]{I_{n_1}\otimes \cdots\otimes \exp L^0_v}\\
&  = \exp L^{0}_{1} \otimes\exp L^{0}_{2}\otimes\cdots\otimes\exp L^{0}_{v} =:
\Sigma_{1}^{0}\otimes\cdots\otimes\Sigma_{v}^{0},
\end{align*}
where the third equality is due to Theorem 10.2 in \cite{higham2008} p235 and
the fact that $L^{0}_{1}\otimes I_{n_{1}}\otimes\cdots\otimes I_{n_{v}}$ and
$I_{n_{1}}\otimes L^{0}_{2}\otimes\cdots\otimes I_{n_{v}}$ commute, the fourth
equality is due to $f(A)\otimes I =f(A\otimes I)$ for any matrix function $f$
(e.g., Theorem 1.13 in \cite{higham2008} p10), the fifth equality is due to a
property of Kronecker products. Note that $\Sigma_{j} ^{0}$ is real symmetric,
positive definite $n_{j}\times n_{j}$ matrix for $j=1,\ldots,v$.

We thus see that $\Sigma^{0}$ is of the Kronecker product form, and that the
precision matrix $\Sigma^{-1}$ has a closest approximating matrix $(\Sigma
^{0})^{-1}$.
%This kind of best approximating property is found in linear regression (Best Linear Predictor) and
This reasoning provides a justification (i.e., interpretation) for using
$\Sigma^{0}$ even when the Kronecker product model is misspecified for the
covariance matrix. The same reasoning applies to any real symmetric, positive
definite correlation matrix $\Theta$.

\cite{vanloan2000} and \cite{pitsianis1997} also considered this nearest
approximation involving one Kronecker product only and in the original
parameter space (not in the log parameter space). In that simplified problem,
they showed that the optimisation problem could be solved by the singular
value decomposition.

\subsection{The Proof of Theorem \ref{thm main rate of convergence}}

\label{secArateofconvergence}

In this subsection, we give a proof for Theorem
\ref{thm main rate of convergence}. We will first give some preliminary lemmas
leading to the proof of this theorem.

The following lemma characterises the relationship between an exponential-type
moment assumption and an exponential tail probability.

\begin{lemma}
\label{lemmaexponentialtail} Suppose that a random variable $X$ satisfies the
exponential-type tail condition, i.e., there exist absolute constants
$K_{1}>1, K_{2}>0, r_{1}>0$ such that
\[
\mathbb{E}\sbr[2]{\exp\del [1]{K_2 |X|^{r_1}}}\leq K_{1}.
\]

\begin{enumerate}
[(i)]

\item Then for every $\epsilon\geq0$, there exists an absolute constant
$b_{1}>0$ such that
\[
\mathbb{P}(|X|\geq\epsilon)\leq\exp\sbr[1]{1-(\epsilon/b_1)^{r_1}}.
\]

\item We have $\mathbb{E}|X|<\infty$.
%a_1=O(1)$.

\item Part (i) implies that for every $\epsilon\geq0$, there exists an
absolute constant $c_{1}>0$ such that
\[
\mathbb{P}(|X-\mathbb{E}X|\geq\epsilon)\leq\exp
\sbr[1]{1-(\epsilon/c_1)^{r_1}}.
\]

\item Suppose that another random variable $Y$ satisfies $\mathbb{E}%
\sbr[2]{\exp\del [1]{K_2^* |Y|^{r_1^*}}}\leq K_{1}^{*}$ for some absolute
constants $K_{1}^{*}>1, K_{2}^{*}>0, r_{1}^{*}>0$. Then for every
$\epsilon\geq0$, there exists an absolute constant $b_{2}>0$ such that
\[
\mathbb{P}(|XY|\geq\epsilon)\leq\exp\sbr[1]{1-(\epsilon/b_2)^{r_2}},
\]
where $r_{2}\in\left(  0,\frac{r_{1}r_{1}^{*}}{r_{1}+r_{1}^{*}}\right]  $.
\end{enumerate}
\end{lemma}

\begin{proof}
For part (i), choose $C:=\log K_1\vee 1$ and $b_1:=(C/K_2)^{1/r_1}$. If $\epsilon>b_1$, we have
\begin{align*}
& \mathbb{P}(|X|\geq\epsilon)\leq \frac{\mathbb{E}\sbr[1]{\exp (K_2 |X|^{r_1})}}{\exp(K_2\epsilon^{r_1})}\leq K_1 e^{-K_2\epsilon^{r_1}}=e^{\log K_1-K_2\epsilon^{r_1}}=e^{\log K_1-C(\epsilon/b_1)^{r_1}}\\
& \leq e^{C[1-(\epsilon/b_1)^{r_1}]}\leq e^{1-(\epsilon/b_1)^{r_1}}
\end{align*}
where the first inequality is due to a variant of Markov's inequality. If $\epsilon\leq b_1$, we have
\begin{align*}
\mathbb{P}(|X|\geq\epsilon)\leq 1\leq e^{1-(\epsilon/b_1)^{r_1}}.
\end{align*}
For part (ii),
\begin{align*}
&\mathbb{E}|X|=\int_{0}^{\infty}\mathbb{P}(|X|\geq t)dt\leq \int_{0}^{\infty}e^{1-(t/b_1)^{r_1}}dt=e\int_{0}^{\infty}e^{-(t/b_1)^{r_1}}dt=\frac{eb_1}{r_1}\int_{0}^{\infty}y^{\frac{1}{r_1}-1}e^{-y}dy\\
&=\frac{eb_1}{r_1}\Gamma(r_1^{-1})< \infty,
\end{align*}
where the first inequality is due to part (i), the third equality is due to change of variable $y=(t/b_1)^{r_1}$, and the last equality is due to recognition of $\int_{0}^{\infty}[\Gamma(r_1^{-1})]^{-1}y^{\frac{1}{r_1}-1}e^{-y}dy=1$ using Gamma distribution.
For part (iii),
\begin{align*}
&\mathbb{P}(|X-\mathbb{E}X|\geq\epsilon)\leq \mathbb{P}(|X|\geq \epsilon -\mathbb{E}|X|)=\mathbb{P}(|X|\geq \epsilon -\mathbb{E}|X|\wedge \epsilon)\leq \exp\sbr[3]{1-\frac{(\epsilon-\mathbb{E}|X|\wedge \epsilon)^{r_1}}{b_1^{r_1}}}
\end{align*}
where the second inequality is due to part (i). First consider the case $0<r_1<1$.
\begin{align*}
&\exp\sbr[3]{1-\frac{(\epsilon-\mathbb{E}|X|\wedge \epsilon)^{r_1}}{b_1^{r_1}}}\leq \exp\sbr[3]{1-\frac{\epsilon^{r_1}-(\mathbb{E}|X|\wedge \epsilon)^{r_1}}{b_1^{r_1}}}=\exp\sbr[3]{1-\frac{\epsilon^{r_1}}{b_1^{r_1}}+\frac{(\mathbb{E}|X|\wedge \epsilon)^{r_1}}{b_1^{r_1}}}\\
&\leq \exp\sbr[3]{1-\frac{\epsilon^{r_1}}{b_1^{r_1}}+\frac{(\mathbb{E}|X|)^{r_1}}{b_1^{r_1}}}\leq \exp\sbr[3]{C-\frac{\epsilon^{r_1}}{b_1^{r_1}}}= \exp\sbr[3]{C\del[3]{1-\frac{\epsilon^{r_1}}{(C^{\frac{1}{r_1}}b_1)^{r_1}}}}=: \exp\sbr[3]{C\del[3]{1-\frac{\epsilon^{r_1}}{c_1^{r_1}}}}
\end{align*}
where the first inequality is due to subadditivity of the concave function: $(x+y)^{r_1}-x^{r_1}\leq y^{r_1}$ for $x,y\geq 0$. If $\epsilon>c_1$, we have, via recognising $C>1$,
\[\mathbb{P}(|X-\mathbb{E}X|\geq\epsilon)\leq \exp\sbr[3]{C\del[3]{1-\frac{\epsilon^{r_1}}{c_1^{r_1}}}}\leq \exp\sbr[3]{1-\frac{\epsilon^{r_1}}{c_1^{r_1}}}.\]
If $\epsilon\leq c_1$, we have
\[\mathbb{P}(|X-\mathbb{E}X|\geq\epsilon)\leq 1\leq \exp\sbr[3]{1-\frac{\epsilon^{r_1}}{c_1^{r_1}}}. \]
We now consider the case $r_1\geq 1$. The proof is almost the same: Instead of relying on subadditivity of the concave function, we rely on Loeve's $c_r$ inequality: $|x+y|^{r_1}\leq 2^{r_1-1}(|x|^{r_1}+|y|^{r_1})$ for $r_1\geq 1$ to get $2^{1-r_1}\epsilon^{r_1}-(\mathbb{E}|X|\wedge \epsilon)^{r_1}\leq (\epsilon-\mathbb{E}|X|\wedge \epsilon)^{r_1}$. $c_1$ is now defined as $C^{\frac{1}{r_1}}b_12^{\frac{r_1-1}{r_1}}$.
For part (iv), an original proof could be found in \cite{fanliaomincheva2011} p3338. Invoke part (i), $\mathbb{P}(|Y|\geq\epsilon)\leq \exp \sbr[1]{1-(\epsilon/b_1^*)^{r_1^*}}$. We have, for any $\epsilon\geq 0$, $M:= \del[1]{\frac{\epsilon (b_1^*)^{(r_1^*/r_1)}}{b_1}}^{\frac{r_1}{r_1+r_1^*}}$, $b:=b_1b_1^*$, $r:=\frac{r_1r_1^*}{r_1+r_1^*}$,
\begin{align*}
&\mathbb{P}(|XY|\geq\epsilon)\leq \mathbb{P}(|X|\geq\epsilon/M)+\mathbb{P}(|Y|\geq M)\leq \exp\sbr[3]{1-\del[3]{\frac{\epsilon/M}{b_1}}^{r_1}}+\exp\sbr[3]{1-\del[3]{\frac{M}{b_1^*}}^{r_1^*}}\\
&= 2\exp \sbr[1]{1-(\epsilon/b)^{r}}.
\end{align*}
Pick an $r_2\in (0,r]$ and $b_2> (1+\log 2)^{1/r}b$. We consider the case $\epsilon\leq b_2$ first.
\[\mathbb{P}(|XY|\geq\epsilon)\leq 1\leq \exp \sbr[1]{1-(\epsilon/b_2)^{r_2}}.\]
We now consider the case $\epsilon> b_2$. Define a function $F(\epsilon):=(\epsilon/b)^r-(\epsilon/b_2)^{r_2}$. Using the definition of $b_2$, we have $F(b_2)>\log 2$. It is also not difficult to show that $F'(\epsilon)>0$ when $\epsilon> b_2$. Thus we have $F(\epsilon)>F(b_2)>\log 2$ when $\epsilon> b_2$. Thus,
\begin{align*}
&\mathbb{P}(|XY|\geq\epsilon)\leq 2\exp \sbr[1]{1-(\epsilon/b)^{r}}=\exp \sbr[1]{\log 2+1-(\epsilon/b)^{r}}\leq \exp [(\epsilon/b)^r-(\epsilon/b_2)^{r_2}+1-(\epsilon/b)^{r}]\\
&=\exp[1-(\epsilon/b_2)^{r_2}].
\end{align*}
\end{proof}

\bigskip

%\subsubsection{title}

This following lemma gives a rate of convergence in terms of spectral norm for
the sample covariance matrix.

\begin{lemma}
\label{lemma spectral norm perturbation covariance matrix} Assume
$n,T\to\infty$ simultaneously and $n/T\leq1$. Suppose Assumptions
\ref{assu subgaussian vector}(i) and \ref{assu mixing} hold with
$1/r_{1}+1/r_{2}>1$. Then
\[
\| \hat{\Sigma}_{T}- \Sigma\|_{\ell_{2}}=O_{p}\del [3]{\sqrt{\frac{n}{T}}}.
\]

\end{lemma}

\begin{proof}
Write $\hat{\Sigma}_T=\frac{1}{T}\sum_{t=1}^{T}y_ty_t^{\intercal}-\bar{y}\bar{y}^{\intercal}$. We have
\begin{equation}
\label{eqn spectral norm Sigma perturbation}
\| \hat{\Sigma}_T- \Sigma\|_{\ell_2}\leq \enVert[3]{\frac{1}{T}\sum_{t=1}^{T}y_ty_t^{\intercal}-\mathbb{E}y_ty_t^{\intercal}}_{\ell_2}+\|\bar{y}\bar{y}^{\intercal}-\mu\mu^{\intercal}\|_{\ell_2}.
\end{equation}
We consider the first term on the right hand side of (\ref{eqn spectral norm Sigma perturbation}) first. Invoke Lemma \ref{lemma symmetric matrix spectral norm} in Appendix \ref{sec oldappendixB} with $\varepsilon=1/4$:
\begin{align*}
\enVert[3]{\frac{1}{T}\sum_{t=1}^{T}y_ty_t^{\intercal}-\mathbb{E}y_ty_t^{\intercal}}_{\ell_{2}}&\leq 2\max_{a\in \mathcal{N}_{1/4}}\envert[3]{a^{\intercal}\del[3]{\frac{1}{T}\sum_{t=1}^{T}y_ty_t^{\intercal}-\mathbb{E}y_ty_t^{\intercal}}a}=:2\max_{a\in \mathcal{N}_{1/4}}\envert[3]{ \frac{1}{T}\sum_{t=1}^{T}(z_{a,t}^2-\mathbb{E}z_{a,t}^2)},
\end{align*}
where $z_{a,t}:=y_t^{\intercal}a$. First, given Assumption \ref{assu subgaussian vector}(i), invoke Lemma \ref{lemmaexponentialtail}(i) and (iv): For every $\epsilon\geq 0$, there exists an absolute constant $b_2>0$ such that
\[\mathbb{P}(|z_{a,t}^2|\geq \epsilon)\leq \exp \sbr[1]{1-(\epsilon/b_2)^{r_1/2}}.\]
Next, invoke Lemma \ref{lemmaexponentialtail}(iii): For every $\epsilon\geq 0$, there exists an absolute constant $c_2>0$ such that
\[\mathbb{P}(|z_{a,t}^2-\mathbb{E}z_{a,t}^2|\geq \epsilon)\leq \exp \sbr[1]{1-(\epsilon/c_2)^{r_1/2}}.\]
Given Assumption \ref{assu mixing} and the fact that mixing properties are hereditary in the sense that for any measurable function $m(\cdot)$, the process $\{m(y_t)\}$ possesses the mixing property of $\{y_t\}$ (\cite{fanyao2003} p69), $z_{a,t}^2-\mathbb{E}z_{a,t}^2$ is strong mixing with the same coefficient: $\alpha(h)\leq \exp\del [1]{-K_3h^{r_2}}$. Define $r$ by $1/r :=2/r_1+1/r_2$. Using the fact that $2/r_1+1/r_2>1$, we can invoke a version of Bernstein's inequality for strong mixing time series (Theorem \ref{thmbernsteininequality} in Appendix \ref{sec oldappendixB}), followed by Lemma \ref{lemmabernsteinrate} in Appendix \ref{sec oldappendixB}:
\[2 \max_{a\in \mathcal{N}_{1/4}}\envert[3]{ \frac{1}{T}\sum_{t=1}^{T}(z_{a,t}^2-\mathbb{E}z_{a,t}^2)}=O_p\del[3]{ \sqrt{\frac{\log |\mathcal{N}_{1/4}|}{T}}}.\]
Invoking Lemma \ref{lemma vol ball} in Appendix \ref{sec oldappendixB}, we have $|\mathcal{N}_{1/4}|\leq 9^n$. Thus we have
\begin{align*}
\enVert[3]{\frac{1}{T}\sum_{t=1}^{T}y_ty_t^{\intercal}-\mathbb{E}y_ty_t^{\intercal}}_{\ell_{2}}&\leq2 \max_{a\in \mathcal{N}_{1/4}}\envert[3]{ \frac{1}{T}\sum_{t=1}^{T}(z_{a,t}^2-\mathbb{E}z_{a,t}^2)}=O_p\del [3]{ \sqrt{\frac{n}{T}}}.
\end{align*}
We now consider the second term on the right hand side of (\ref{eqn spectral norm Sigma perturbation}).
\begin{align*}
&\|\bar{y}\bar{y}^{\intercal}-\mu\mu^{\intercal}\|_{\ell_2}\leq
2 \max_{a\in \mathcal{N}_{1/4}}\envert[3]{a^{\intercal}\del [3]{\bar{y}\bar{y}^{\intercal}-\mu\bar{y}^{\intercal}+\mu\bar{y}^{\intercal}-\mu\mu^{\intercal}}a} =2 \max_{a\in \mathcal{N}_{1/4}}\envert[3]{a^{\intercal}\del [3]{(\bar{y}-\mu)\bar{y}^{\intercal}+\mu(\bar{y}-\mu)^{\intercal}}a}\\
&\leq 2\max_{a\in \mathcal{N}_{1/4}}\envert[1]{a^{\intercal}(\bar{y}-\mu)\bar{y}^{\intercal}a}+2\max_{a\in \mathcal{N}_{1/4}}\envert[1]{a^{\intercal}\mu (\bar{y}-\mu)^{\intercal}a}\\
&\leq 2\max_{a\in \mathcal{N}_{1/4}}\envert[1]{a^{\intercal}(\bar{y}-\mu)}\max_{a\in \mathcal{N}_{1/4}}\envert[1]{\bar{y}^{\intercal}a}+2\max_{a\in \mathcal{N}_{1/4}}\envert[1]{a^{\intercal}\mu}\max_{a\in \mathcal{N}_{1/4}}\envert[1]{ (\bar{y}-\mu)^{\intercal}a},
\end{align*}
where the first inequality is due to Lemma \ref{lemma symmetric matrix spectral norm} in Appendix \ref{sec oldappendixB} with $\varepsilon=1/4$. We consider $\max_{a\in \mathcal{N}_{1/4}}\envert[1]{ (\bar{y}-\mu)^{\intercal}a}$ first.
\begin{align*}
\max_{a\in \mathcal{N}_{1/4}}\envert[1]{(\bar{y}-\mu)^{\intercal}a}=\max_{a\in \mathcal{N}_{1/4}}\envert[3]{\frac{1}{T}\sum_{t=1}^{T}(y_t^{\intercal}a-\mathbb{E}[y_t^{\intercal}a])}=:\max_{a\in \mathcal{N}_{1/4}}\envert[3]{\frac{1}{T}\sum_{t=1}^{T}(z_{a,t}-\mathbb{E}z_{a,t})}.
\end{align*}
Recycling the proof for $\max_{a\in \mathcal{N}_{1/4}}\envert[1]{ \frac{1}{T}\sum_{t=1}^{T}(z_{a,t}^2-\mathbb{E}z_{a,t}^2)}=O_p\del [1]{ \sqrt{\frac{n}{T}}}$ but with $1/r :=1/r_1+1/r_2>1$ this time, we have
\begin{equation}
\max_{a\in \mathcal{N}_{1/4}}\envert[1]{(\bar{y}-\mu)^{\intercal}a}=\max_{a\in \mathcal{N}_{1/4}}\envert[3]{\frac{1}{T}\sum_{t=1}^{T}(z_{a,t}-\mathbb{E}z_{a,t})}=O_p\del[3]{ \sqrt{\frac{\log |\mathcal{N}_{1/4}|}{T}}}=O_p\del [3]{ \sqrt{\frac{n}{T}}}\label{ali amu.1}.
\end{equation}
%
%Invoking Lemma \ref{lemma vol ball} in Appendix \ref{sec oldappendixB}, we have $|\mathcal{N}_{1/4}|\leq 9^n$. Thus we have
%\begin{align}
%\max_{a\in \mathcal{N}_{1/4}}\envert[1]{(\bar{y}-\mu)^{\intercal}a}= O_p\del [3]{\frac{n}{T}\vee \sqrt{\frac{n}{T}}}=O_p\del [3]{ \sqrt{\frac{n}{T}}},\label{ali amu.1}
%\end{align}
%where the last equality is due to Assumption \ref{assu n indexed by T}(i).
Now let's consider $\max_{a\in \mathcal{N}_{1/4}}\envert[1]{a^{\intercal}\mu}$.
\begin{align}
& \max_{a\in \mathcal{N}_{1/4}}\envert[1]{a^{\intercal}\mu}:=\max_{a\in \mathcal{N}_{1/4}}\envert[1]{\mathbb{E}a^{\intercal}y_t}=\max_{a\in \mathcal{N}_{1/4}}\envert[1]{\mathbb{E}z_{a,t}}\leq \max_{a\in \mathcal{N}_{1/4}}\mathbb{E}|z_{a,t}|=O(1) ,\label{ali amu.2}%\max_{a\in \mathcal{N}_{1/4}}\|z_{a,t}\|_{L_1}\\
%&\leq \max_{a\in \mathcal{N}_{1/4}}\|z_{a,t}\|_{\psi_1}\leq\max_{a\in \mathcal{N}_{1/4}}\|z_{a,t}\|_{\psi_2}(\log 2)^{-1/2}\leq \frac{(1+K)^{1/2}}{C^{1/2}}(\log 2)^{-1/2},
\end{align}
where the last equality is due to Lemma \ref{lemmaexponentialtail}(ii).
%where $\|\cdot\|_{L_1}$ is the $L_1$ norm, the second and third inequalities are from \cite{vandervaartWellner1996} p95. Thus we have
%\begin{align}
%\max_{a\in \mathcal{N}_{1/4}}|a^{\intercal}\mu|=O(1).
%\end{align}
%
Next we consider $\max_{a\in \mathcal{N}_{1/4}}\envert[1]{a^{\intercal}\bar{y}}$.
\begin{align}
\max_{a\in \mathcal{N}_{1/4}}\envert[1]{a^{\intercal}\bar{y}}&=\max_{a\in \mathcal{N}_{1/4}}\envert[1]{a^{\intercal}(\bar{y}-\mu+\mu)}\leq \max_{a\in \mathcal{N}_{1/4}}\envert[1]{a^{\intercal}(\bar{y}-\mu)}+\max_{a\in \mathcal{N}_{1/4}}\envert[1]{a^{\intercal}\mu}=O_p\del [3]{ \sqrt{\frac{n}{T}}}+O(1)\notag\\
&=O_p(1),\label{ali amu.3}
\end{align}
where the last equality is due to $n\leq T$. Combining (\ref{ali amu.1}), (\ref{ali amu.2}) and (\ref{ali amu.3}), we have
\[\|\bar{y}\bar{y}^{\intercal}-\mu\mu^{\intercal}\|_{\ell_2}=O_p\del [3]{ \sqrt{\frac{n}{T}}}.\]
\end{proof}

\bigskip
%\subsubsection{title}

The following lemma gives a rate of convergence in terms of spectral norm for
various quantities involving variances of $y_{t}$. The rate $\sqrt{n/T}$ is
suboptimal, but there is no need improving it further as these quantities will
not be the dominant terms in the proof of Theorem
\ref{thm main rate of convergence}.

\begin{lemma}
\label{lemmaD} Suppose Assumptions \ref{assu subgaussian vector}(i),
\ref{assu mixing}, \ref{assu n indexed by T}(i) and
\ref{assu about D and Dhat}(i) hold with $1/r_{1}+1/r_{2}>1$. Then

\begin{enumerate}
[(i)]

\item
\[
\|\hat{D}_{T}-D\|_{\ell_{2}}=O_{p}\del [3]{\sqrt{\frac{n}{T}}}.
\]

\item The minimum eigenvalue of $D$ is bounded away from zero by an absolute
positive constant (i.e., $\|D^{-1}\|_{\ell_{2}}=O(1)$), so is the minimum
eigenvalue of $D^{1/2}$ (i.e., $\|D^{-1/2}\|_{\ell_{2}}=O(1)$).

\item
\[
\|\hat{D}_{T}^{1/2}-D^{1/2}\|_{\ell_{2}}=O_{p}\del [3]{\sqrt{\frac{n}{T}}}.
\]

\item
\[
\|\hat{D}_{T}^{-1/2}-D^{-1/2}\|_{\ell_{2}}=O_{p}\del [3]{\sqrt{\frac{n}{T}}}.
\]

\item
\[
\|\hat{D}_{T}^{-1/2}\|_{\ell_{2}}=O_{p}(1).
\]

\item The maximum eigenvalue of $\Sigma$ is bounded from the above by an
absolute constant (i.e., $\|\Sigma\|_{\ell_{2}}=O(1)$). The maximum eigenvalue
of $D$ is bounded from the above by an absolute constant (i.e., $\|D\|_{\ell
_{2}}=O(1)$).

\item
\[
\|\hat{D}_{T}^{-1/2}\otimes\hat{D}_{T}^{-1/2} -D^{-1/2}\otimes D^{-1/2}%
\|_{\ell_{2}}=O_{p}\del [3]{\sqrt{\frac{n}{T}}}.
\]

\end{enumerate}
\end{lemma}

\begin{proof}
Define $\sigma_i^2:=\mathbb{E}(y_{t,i}-\sigma_i)^2$ and $\hat{\sigma}_i^2:=\frac{1}{T}\sum_{t=1}^{T}(y_{t,i}-\bar{y}_i)^2$, where the subscript $i$ denotes the $i$th component of the corresponding vector. For part (i),
\begin{align*}
\|\hat{D}_T-D\|_{\ell_2}&=\max_{1\leq i\leq n}|\hat{\sigma}_i^2-\sigma_i^2|=\max_{1\leq i\leq n}|e_i^{\intercal}(\hat{\Sigma}_T-\Sigma)e_i|\leq \max_{\|a\|_2=1}|a^{\intercal}(\hat{\Sigma}_T-\Sigma)a|=\| \hat{\Sigma}_T- \Sigma\|_{\ell_2},
\end{align*}
where $e_i$ denotes a unit vector whose $i$th component is 1. Now invoke Lemma \ref{lemma spectral norm perturbation covariance matrix} to get the result.
For part (ii),
\begin{align*}
\text{mineval}(D)&=\min_{1\leq i\leq n}\sigma_i^2=\min_{1\leq i\leq n} e_i^{\intercal}\Sigma e_i\geq \min_{\|a\|_2=1}a^{\intercal}\Sigma a=\text{mineval}(\Sigma)>0
\end{align*}
where the last inequality is due to Assumption \ref{assu about D and Dhat}(i). The statement about the minimum eigenvalue of $D^{1/2}$ is also true. For part (iii), invoking Lemma \ref{lemma square root difference spectral norm} in Appendix \ref{sec oldappendixB} gives
\begin{align*}
\|\hat{D}_T^{1/2}-D^{1/2}\|_{\ell_2}\leq\frac{\|\hat{D}_T-D\|_{\ell_{2}}}{\text{mineval}(\hat{D}_T^{1/2})+\text{mineval}(D^{1/2})}=O_p(1)\|\hat{D}_T-D\|_{\ell_{2}}=O_p\del [3]{\sqrt{\frac{n}{T}}},
\end{align*}
where the first and second equalities are due to parts (ii) and (i), respectively.
Part (iv) follows from Lemma \ref{lemma saikkonen lemma} in Appendix \ref{sec oldappendixB} via parts (ii) and (iii).
For part (v),
\begin{align*}
\|\hat{D}_T^{-1/2}\|_{\ell_2}&=\|\hat{D}_T^{-1/2}-D^{-1/2}+D^{-1/2}\|_{\ell_2}\leq \|\hat{D}_T^{-1/2}-D^{-1/2}\|_{\ell_2}+\|D^{-1/2}\|_{\ell_2}\\
&=O_p\del [3]{\sqrt{\frac{n}{T}}}+O(1)=O_p(1),
\end{align*}
where the second equality is due to parts (iv) and (ii). For part (vi), we have
\begin{align*}
\|\Sigma\|_{\ell_2}=\max_{\|a\|_2=1}\envert[1]{a^{\intercal}\del[1]{\mathbb{E}[y_ty_t^{\intercal}]-\mu\mu^{\intercal}}a}\leq \max_{\|a\|_2=1}\mathbb{E}z_{a,t}^2+\max_{\|a\|_2=1}(\mathbb{E}z_{a,t})^2\leq 2\max_{\|a\|_2=1}\mathbb{E}z_{a,t}^2.
\end{align*}
We have shown that in the proof of Lemma \ref{lemma spectral norm perturbation covariance matrix} that $z_{a,t}^2$ has an exponential tail for any $\|a\|_2=1$. This says that $\mathbb{E}z_{a,t}^2$ is bounded for any $\|a\|_2=1$ via Lemma \ref{lemmaexponentialtail}(ii), so the result follows.
Next we consider
\begin{align*}
\|D\|_{\ell_2}=\max_{1\leq i\leq n}\sigma^2_i=\max_{1\leq i\leq n}e_i^{\intercal}\Sigma e_i\leq \max_{\|a\|_2=1}a^{\intercal}\Sigma a=\text{maxeval}(\Sigma)<\infty.
\end{align*}
For part (vii),
\begin{align*}
&\|\hat{D}_T^{-1/2}\otimes \hat{D}_T^{-1/2} -D^{-1/2}\otimes D^{-1/2}\|_{\ell_2}\\
&= \|\hat{D}_T^{-1/2}\otimes \hat{D}_T^{-1/2} -\hat{D}_T^{-1/2}\otimes D^{-1/2}+ \hat{D}_T^{-1/2}\otimes D^{-1/2}-D^{-1/2}\otimes D^{-1/2}\|_{\ell_2}\\
&\leq  \|\hat{D}_T^{-1/2}\otimes (\hat{D}_T^{-1/2} - D^{-1/2})\|_{\ell_2}+ \|(\hat{D}_T^{-1/2}-D^{-1/2})\otimes D^{-1/2}\|_{\ell_2}\\
&= \del [1]{\|\hat{D}_T^{-1/2}\|_{\ell_2}+ \|D^{-1/2}\|_{\ell_2}}\|\hat{D}_T^{-1/2} - D^{-1/2}\|_{\ell_2}=O_p\del [3]{\sqrt{\frac{n}{T}}},
\end{align*}
where the second equality is due to Lemma \ref{lemma l2norm of kronecker product} in Appendix \ref{sec oldappendixB}, and the last equality is due to parts (ii), (v) and (iv).
\end{proof}

\bigskip
%\subsubsection{title}

To prove part (ii) of Theorem \ref{thm main rate of convergence}, we shall use
Lemma 4.1 of \cite{Gil2012}. That lemma will further simplify when we consider
real symmetric, positive definite matrices. For the ease of reference, we
state this simplified version of Lemma 4.1 of \cite{Gil2012} here.

\begin{lemma}
[Simplified from Lemma 4.1 of \cite{Gil2012}]\label{lemma Gil2012 4.1} For
$n\times n$ real symmetric, positive definite matrices $A,B$, if $\| A-
B\|_{\ell_{2}}< a$ for some absolute constant $a>1$, then
\[
\|\log A-\log B\|_{\ell_{2}}\leq c\| A- B\|_{\ell_{2}},
\]
for some positive absolute constant $c$.
\end{lemma}

\begin{proof}
First note that for any real symmetric, positive definite matrix $Q$, $p(Q,x)=x$ for any $x>0$ in Lemma 4.1 of \cite{Gil2012}. Since $A$ is real symmetric and positive definite, all its eigenvalues lie in the region $|\arg (z-a)|\leq \pi/2$. Then according to \cite{Gil2012} p11, we have for any $t\geq 0$ not coinciding with eigenvalues of $A$
\begin{align*}
\rho(A,-t)&\geq (a+t) \sin (\pi/2)=a+t\\
\rho(A,-t)-\delta& \geq a+t-\delta,
\end{align*}
where
\[\delta :=\left\lbrace \begin{array}{cc}
\| A- B\|_{\ell_2}^{1/n} & \text{if }\| A- B\|_{\ell_2}\leq 1\\
\| A- B\|_{\ell_2} & \text{if } \| A- B\|_{\ell_2}\geq 1
\end{array} \right. \]
and $\rho(A,-t)$ is defined in \cite{Gil2012} p3. Then the condition of Lemma \ref{lemma Gil2012 4.1} allows one to invoke Lemma 4.1 of \cite{Gil2012} as
\[\rho(A,-t)\geq a+t\geq a >\delta.\]
Lemma 4.1 of \cite{Gil2012} says
\begin{align*}
& \|\log A-\log B\|_{\ell_2}\leq \| A- B\|_{\ell_2}\int_{0}^{\infty}p\del [3]{A,\frac{1}{\rho(A,-t)}}p\del [3]{B,\frac{1}{\rho(A,-t)-\delta}}dt\\
&=\| A- B\|_{\ell_2}\int_{0}^{\infty}\frac{1}{\rho(A,-t)}\frac{1}{\rho(A,-t)-\delta}dt\leq \| A- B\|_{\ell_2}\int_{0}^{\infty}\frac{1}{(a+t)(a+t-\delta)}dt\\
&\leq \| A- B\|_{\ell_2}\int_{0}^{\infty}\frac{1}{(a+t-\delta)^2}dt=\| A- B\|_{\ell_2}\frac{1}{a-\delta}=:c\| A- B\|_{\ell_2}.
\end{align*}
\end{proof}

\bigskip

We are now ready to give a proof for Theorem
\ref{thm main rate of convergence}

\begin{proof}[Proof of Theorem \ref{thm main rate of convergence}]
For part (i), recall that
\[ \hat{\Theta}_T=\hat{D}_T^{-1/2}\hat{\Sigma}_T\hat{D}_T^{-1/2},\qquad \Theta=D^{-1/2}\Sigma D^{-1/2}. \]
Then we have
\begin{align}
&\| \hat{\Theta}_T- \Theta\|_{\ell_2}=\|\hat{D}_T^{-1/2}\hat{\Sigma}_T\hat{D}_T^{-1/2}-\hat{D}_T^{-1/2}\Sigma\hat{D}_T^{-1/2}+\hat{D}_T^{-1/2}\Sigma\hat{D}_T^{-1/2}-D^{-1/2}\Sigma D^{-1/2}\|_{\ell_2}\notag \\
&\leq \|\hat{D}_T^{-1/2}\|_{\ell_2}^2\|\hat{\Sigma}_T-\Sigma\|_{\ell_2}+\|\hat{D}_T^{-1/2}\Sigma\hat{D}_T^{-1/2}-D^{-1/2}\Sigma D^{-1/2}\|_{\ell_2}.\label{ali correlation spectral norm}
\end{align}
Invoking Lemmas \ref{lemma spectral norm perturbation covariance matrix} and \ref{lemmaD}(v), we conclude that the first term of (\ref{ali correlation spectral norm}) is $O_p(\sqrt{n/T})$. Let's consider the second term of (\ref{ali correlation spectral norm}). Write
\begin{align*}
& \|\hat{D}_T^{-1/2}\Sigma \hat{D}_T^{-1/2}-D^{-1/2}\Sigma \hat{D}_T^{-1/2}+D^{-1/2}\Sigma \hat{D}_T^{-1/2}-D^{-1/2}\Sigma D^{-1/2}\|_{\ell_2}\\
&\leq \|(\hat{D}_T^{-1/2}-D^{-1/2})\Sigma \hat{D}_T^{-1/2}\|_{\ell_2}+\|D^{-1/2}\Sigma (\hat{D}_T^{-1/2}-D^{-1/2})\|_{\ell_2}\\
&\leq \|\hat{D}_T^{-1/2}\|_{\ell_2}\|\Sigma\|_{\ell_2}\| \hat{D}_T^{-1/2}-D^{-1/2}\|_{\ell_2}+ \|D^{-1/2}\|_{\ell_2}\|\Sigma\|_{\ell_2}\| \hat{D}_T^{-1/2}-D^{-1/2}\|_{\ell_2}.
\end{align*}
Invoking Lemma \ref{lemmaD}(ii), (iv), (v) and (vi), we conclude that the second term of (\ref{ali correlation spectral norm}) is $O_p(\sqrt{n/T})$.
For part (ii), it follows trivially from Lemma \ref{lemma Gil2012 4.1}.
For part (iii), we have
\begin{align*}
&\|\hat{\theta}_T-\theta^0\|_2=\|(E^{\intercal}WE)^{-1}E^{\intercal}W\|_{\ell_2}\|D_n^+\|_{\ell_2} \|\log \hat{\Theta}_T-\log \Theta\|_F\leq \\
& \|(E^{\intercal}WE)^{-1}E^{\intercal}W\|_{\ell_2}\sqrt{n} \|\log \hat{\Theta}_T-\log \Theta\|_{\ell_2}=O(\sqrt{\varpi \kappa(W)/n})\sqrt{n}O_p(\sqrt{n/T})=O_p\del [3]{\sqrt{\frac{n \varpi \kappa(W)}{T}}},
\end{align*}
where the first inequality is due to (\ref{eqn spectral norm for Dnplus and Dn}), and the second equality is due to (\ref{align Eplus spectral norm}) and parts (i)-(ii) of this theorem.
\end{proof}

\subsection{The Proof of Theorem \ref{thm asymptotic normality}}

\label{sec A4}

In this subsection, we give a proof for Theorem \ref{thm asymptotic normality}%
. We will first give some preliminary lemmas leading to the proof of this theorem.

The following lemma linearizes the matrix logarithm.

\begin{lemma}
\label{prop matrix log expansion} Suppose both $n\times n$ matrices $A+B$ and
$A$ are real, symmetric, and positive definite for all $n$ with the minimum
eigenvalues bounded away from zero by absolute constants. Suppose the maximum
eigenvalue of $A$ is bounded from above by an absolute constant. Further
suppose
\begin{equation}
\enVert[1]{[t(A-I)+I]^{-1}tB}_{\ell_{2}}\leq C<1
\label{eqn l2norm less than C less than 1}%
\end{equation}
for all $t\in\lbrack0,1]$ and some constant $C$. Then
\[
\log(A+B)-\log A=\int_{0}^{1}[t(A-I)+I]^{-1}B[t(A-I)+I]^{-1}dt+O(\Vert
B\Vert_{\ell_{2}}^{2}\vee\Vert B\Vert_{\ell_{2}}^{3}).
\]

\end{lemma}

\bigskip

The conditions of the preceding lemma implies that for every $t\in[0,1]$,
$t(A-I)+I$ is positive definite for all $n$ with the minimum eigenvalue
bounded away from zero by an absolute constant (\cite{hornjohnson1985} Theorem
4.3.1 p181). Lemma \ref{prop matrix log expansion} has a flavour of Frechet
derivative because $\int_{0}^{1}[t(A-I)+I]^{-1}B[t(A-I)+I]^{-1}dt$ is the
Frechet derivative of matrix logarithm at $A$ in the direction $B$
(\cite{higham2008} (11.10) p272); however, this lemma is slightly stronger in
the sense of a sharper bound on the remainder.

\begin{proof}
% Proofs last checked on 20160510; highly trustworthy.
% Proofs last checked on 20171220; highly trustworthy.
Since both $A+B$ and $A$ are positive definite for all $n$, with minimum eigenvalues real and  bounded away from zero by absolute constants, by Theorem \ref{thm integral form for log} in Appendix \ref{sec oldappendixB}, we have
\[\log (A+B)=\int_{0}^{1}(A+B-I)[t(A+B-I)+I]^{-1}dt,\quad \log A=\int_{0}^{1}(A-I)[t(A-I)+I]^{-1}dt.\]
Use (\ref{eqn l2norm less than C less than 1}) to invoke Lemma \ref{prop matrix inverse perturbation} in Appendix \ref{sec oldappendixB} to expand $[t(A-I)+I+tB]^{-1}$ to get
\[[t(A-I)+I+tB]^{-1}=[t(A-I)+I]^{-1}-[t(A-I)+I]^{-1}tB[t(A-I)+I]^{-1}+O(\|B\|^2_{\ell_2})\]
and substitute into the expression of $\log (A+B)$
\begin{align*}
&\log (A+B)\\
&=\int_{0}^{1}(A+B-I)\left\lbrace [t(A-I)+I]^{-1}-[t(A-I)+I]^{-1}tB[t(A-I)+I]^{-1}+O(\|B\|^2_{\ell_2})\right\rbrace dt\\
&=\log A+\int_{0}^{1}B[t(A-I)+I]^{-1}dt-\int_{0}^{1}t(A+B-I)[t(A-I)+I]^{-1}B[t(A-I)+I]^{-1}dt\\
&\qquad +(A+B-I)O(\|B\|^2_{\ell_2})\\
&=\log A+\int_{0}^{1}[t(A-I)+I]^{-1}B[t(A-I)+I]^{-1}dt-\int_{0}^{1}tB[t(A-I)+I]^{-1}B[t(A-I)+I]^{-1}dt\\
&\qquad +(A+B-I)O(\|B\|^2_{\ell_2})\\
&=\log A+\int_{0}^{1}[t(A-I)+I]^{-1}B[t(A-I)+I]^{-1}dt+O(\|B\|^2_{\ell_2}\vee \|B\|^3_{\ell_2}),
\end{align*}
where the last equality follows from $\text{maxeval}(A)<C<\infty$ and $\text{mineval}[t(A-I)+I]>C'>0$.
\end{proof}

%\subsubsection{title}
\bigskip

\begin{lemma}
\label{prop mini eigenvalue} Suppose Assumptions \ref{assu subgaussian vector}%
(i), \ref{assu mixing}, \ref{assu n indexed by T}(i) and
\ref{assu about D and Dhat}(i) hold with $1/r_{1}+1/r_{2}>1$.

\begin{enumerate}
[(i)]

\item Then $\Theta$ has minimum eigenvalue bounded away from zero by an
absolute constant and maximum eigenvalue bounded from above by an absolute constant.

\item Then $\hat{\Theta}_{T}$ has minimum eigenvalue bounded away from zero by
an absolute constant and maximum eigenvalue bounded from above by an absolute
constant with probability approaching 1.
\end{enumerate}
\end{lemma}

\begin{proof}
For part (i), the maximum eigenvalue of $\Theta$ is its spectral norm, i.e., $\|\Theta\|_{\ell_2}$.
\begin{align*}
\|\Theta\|_{\ell_2}=\|D^{-1/2}\Sigma D^{-1/2}\|_{\ell_2}\leq \|D^{-1/2}\|_{\ell_2}^2\|\Sigma \|_{\ell_2}<C,
\end{align*}
where the last inequality is due to Lemma \ref{lemmaD}(ii) and (vi). Now let's consider the minimum eigenvalue of $\Theta$.
\begin{align*}
& \text{mineval}(\Theta)=\text{mineval}(D^{-1/2}\Sigma D^{-1/2})=\min_{\|a\|_2=1}a^{\intercal}D^{-1/2}\Sigma D^{-1/2}a \geq \min_{\|a\|_2=1}\text{mineval}(\Sigma)\|D^{-1/2}a\|_2^2\\
& =\text{mineval}(\Sigma)\min_{\|a\|_2=1}a^{\intercal}D^{-1}a=\text{mineval}(\Sigma)\text{mineval}(D^{-1})=\frac{\text{mineval}(\Sigma)}{\text{maxeval}(D)}>0,
\end{align*}
where the second equality is due to Rayleigh-Ritz theorem, and the last inequality is due to Assumption \ref{assu about D and Dhat}(i) and Lemma \ref{lemmaD}(vi).
For part (ii), the maximum eigenvalue of $\hat{\Theta}$ is its spectral norm, i.e., $\|\hat{\Theta}\|_{\ell_2}$.
\begin{align*}
\|\hat{\Theta}_T\|_{\ell_2}\leq \|\hat{\Theta}_T-\Theta\|_{\ell_2}+\|\Theta\|_{\ell_2}=O_p\del[3]{\sqrt{\frac{n}{T}}}+\|\Theta\|_{\ell_2}=O_p(1)
\end{align*}
where the first equality is due to Theorem \ref{thm main rate of convergence}(i) and the last equality is due to part (i). The minimum eigenvalue of $\hat{\Theta}_T$ is $1/ \text{maxeval}(\hat{\Theta}^{-1}_T)$. Since $\|\Theta^{-1}\|_{\ell_2}=\text{maxeval}(\Theta^{-1})=1/\text{mineval}(\Theta)=O(1)$ by part (i) and $\|\hat{\Theta}_T-\Theta\|_{\ell_2}=O_p(\sqrt{n/T})$ by Theorem \ref{thm main rate of convergence}(i), we can invoke Lemma \ref{lemma saikkonen lemma} in Appendix \ref{sec oldappendixB} to get
\[\|\hat{\Theta}^{-1}_T-\Theta^{-1}\|_{\ell_2}=O_p(\sqrt{n/T}),\]
whence we have
\[\|\hat{\Theta}^{-1}_T\|_{\ell_2}\leq \|\hat{\Theta}^{-1}_T-\Theta^{-1}\|_{\ell_2}+\|\Theta^{-1}\|_{\ell_2}=O_p(1).\]
Thus the minimum eigenvalue of $\hat{\Theta}_T$ is bounded away from zero by an absolute constant.
\end{proof}

\bigskip

Recall $\hat{H}_{T}$ defined in (\ref{eqn hat H}). The following lemma gives a
rate of convergence for $\hat{H}_{T}$. It is also true when one replaces
$\hat{H}_{T}$ with $\hat{H}_{T,D}$ defined in (\ref{eqn hat H D}).

\begin{lemma}
Let Assumptions \ref{assu subgaussian vector}(i), \ref{assu mixing},
\ref{assu n indexed by T}(i) and \ref{assu about D and Dhat}(i) be satisfied
with $1/r_{1}+1/r_{2}>1$. Then we have
\begin{equation}
\label{eqn spectral norm of H and Hhat}\|H\|_{\ell_{2}}=O(1),\qquad\|\hat
{H}_{T}\|_{\ell_{2}}=O_{p}(1),\qquad\|\hat{H}_{T}-H\|_{\ell_{2}}%
=O_{p}\del[3]{ \sqrt{\frac{n}{T}}}.
\end{equation}

\end{lemma}

\begin{proof}
The proofs for $\|H\|_{\ell_2}=O(1)$ and $\|\hat{H}_T\|_{\ell_2}=O_p(1)$ are exactly the same, so we only give the proof for the latter. Define $A_t:=[t(\hat
{\Theta}_{T}-I)+I]^{-1}$ and $B_t:=[t(\Theta-I)+I]^{-1}$.
\begin{align*}
&\|\hat{H}_T\|_{\ell_2}=\enVert[3]{ \int_{0}^{1}A_t\otimes A_tdt}_{\ell_2}\leq \int_{0}^{1}\enVert[1]{ A_t\otimes A_t}_{\ell_2}dt\leq \max_{t\in [0,1]}\enVert[1]{ A_t\otimes A_t}_{\ell_2}=\max_{t\in [0,1]}\enVert[0]{ A_t}_{\ell_2}^2\\
&=\max_{t\in [0,1]}\{\text{maxeval}([t(\hat
{\Theta}_{T}-I)+I]^{-1})\}^2=\max_{t\in [0,1]}\cbr[3]{\frac{1}{\text{mineval}(t(\hat
{\Theta}_{T}-I)+I)}}^2=O_p(1),
\end{align*}
where the second equality is due to Lemma \ref{lemma l2norm of kronecker product} in Appendix \ref{sec oldappendixB}, and the last equality is due to Lemma \ref{prop mini eigenvalue}(ii). Now,
\begin{align*}
&\|\hat{H}_T-H\|_{\ell_2}=\enVert[3]{ \int_{0}^{1}A_t\otimes A_t-B_t\otimes B_tdt}_{\ell_2} \leq \int_{0}^{1}\left\| A_t\otimes A_t-B_t\otimes B_t\right\|_{\ell_2}dt\\
&\leq \max_{t\in [0,1]}\left\| A_t\otimes A_t-B_t\otimes B_t\right\|_{\ell_2}=\max_{t\in [0,1]}\left\| A_t\otimes A_t-A_t\otimes B_t+A_t\otimes B_t-B_t\otimes B_t\right\|_{\ell_2}\\
&=\max_{t\in [0,1]}\left\| A_t\otimes (A_t- B_t)+(A_t-B_t)\otimes B_t\right\|_{\ell_2}\leq \max_{t\in [0,1]}\del[1]{ \left\| A_t\otimes (A_t- B_t)\right\|_{\ell_2} +\left\| (A_t-B_t)\otimes B_t\right\|_{\ell_2}} \\
&=\max_{t\in [0,1]}\del[1]{ \left\| A_t\right\|_{\ell_2} \left\|  A_t- B_t\right\|_{\ell_2} +\left\| A_t-B_t\right\|_{\ell_2} \left\|  B_t\right\|_{\ell_2}} =\max_{t\in [0,1]} \left\| A_t-B_t\right\|_{\ell_2}( \left\| A_t\right\|_{\ell_2}+\left\|  B_t\right\|_{\ell_2}) \\
&=O_p(1)\max_{t\in [0,1]}\left\| [t(\hat
{\Theta}_{T}-I)+I]^{-1}-[t(\Theta-I)+I]^{-1}\right\|_{\ell_2}
\end{align*}
where the first inequality is due to Jensen's inequality, the third equality is due to special properties of Kronecker product, the fourth equality is due to Lemma \ref{lemma l2norm of kronecker product} in Appendix \ref{sec oldappendixB}, and the last equality is because Lemma \ref{prop mini eigenvalue} implies
\[\|[t(\hat{\Theta}_{T}-I)+I]^{-1}\|_{\ell_2}=O_p(1) \qquad \|[t(\Theta-I)+I]^{-1}\|_{\ell_2}=O(1).\]
Now
\begin{align*}
\left\| [t(\hat
{\Theta}_{T}-I)+I]-[t(\Theta-I)+I]\right\|_{\ell_2}=t\|\hat
{\Theta}_{T}-\Theta\|_{\ell_2}=O_p(\sqrt{n/T}),
\end{align*}
where the last equality is due to Theorem \ref{thm main rate of convergence}(i). The lemma then follows after invoking Lemma \ref{lemma saikkonen lemma} in Appendix \ref{sec oldappendixB}.
\end{proof}

\bigskip

\begin{lemma}
Given the $n^{2}\times n(n+1)/2$ duplication matrix $D_{n}$ and its
Moore-Penrose generalised inverse $D_{n}^{+}=(D_{n}^{\intercal}D_{n}%
)^{-1}D_{n}^{\intercal}$ (i.e., $D_{n}$ is full-column rank), we have
\begin{equation}
\label{eqn spectral norm for Dnplus and Dn}\|D_{n}^{+}\|_{\ell_{2}}%
=\|D_{n}^{+\intercal}\|_{\ell_{2}}=1,\qquad\|D_{n}\|_{\ell_{2}}=\|D_{n}%
^{\intercal}\|_{\ell_{2}}=2.
\end{equation}

\end{lemma}

\begin{proof}
First note that $D_n^{\intercal}D_n$ is a diagonal matrix with diagonal entries either 1 or 2.  Using the fact that for any matrix $A$, $AA^{\intercal}$ and $A^{\intercal}A$ have the same non-zero eigenvalues, we have
\begin{align*}
\|D_n^{+\intercal}\|_{\ell_2}^2&=\text{maxeval}(D_n^+D_n^{+\intercal})=\text{maxeval}((D_n^{\intercal}D_n)^{-1})=1\\
\|D_n^{+}\|_{\ell_2}^2&=\text{maxeval}(D_n^{+\intercal}D_n^+)=\text{maxeval}(D_n^+D_n^{+\intercal})=\text{maxeval}((D_n^{\intercal}D_n)^{-1})=1\\
\|D_n\|_{\ell_2}^2&=\text{maxeval}(D_n^{\intercal}D_n)=2\\
\|D_n^{\intercal}\|_{\ell_2}^2&=\text{maxeval}(D_nD_n^{\intercal})=\text{maxeval}(D_n^{\intercal}D_n)=2
\end{align*}
\end{proof}

\bigskip

We are now ready to give a proof for Theorem \ref{thm asymptotic normality}.

\begin{proof}[Proof of Theorem \ref{thm asymptotic normality}]
We first show that (\ref{eqn l2norm less than C less than 1}) is satisfied with probability approaching 1 for $A=\Theta$ and $B=\hat{\Theta}_T-\Theta$. That is,
\[\Vert\lbrack t(\Theta-I)+I]^{-1}t(\hat{\Theta}_T-\Theta)\Vert_{\ell_{2}}\leq C<1\qquad\text{with probability approaching 1},\]
for some constant $C$.
\begin{align*}
&\|[t(\Theta-I)+I]^{-1}t(\hat{\Theta}_T-\Theta)\|_{\ell_2}\leq t \|[t(\Theta-I)+I]^{-1}\|_{\ell_2}\|\hat{\Theta}_T-\Theta\|_{\ell_2}\\
&=\|[t(\Theta-I)+I]^{-1}\|_{\ell_2}O_p(\sqrt{n/T})=O_p(\sqrt{n/T})/\text{mineval}(t(\Theta-I)+I)=o_p(1),
\end{align*}
where the first equality is due to Theorem \ref{thm main rate of convergence}(i), and the last equality is due to $\text{mineval}(t(\Theta-I)+I)>C>0$ for some absolute constant $C$ (implied by Lemma \ref{prop mini eigenvalue}(i)) and Assumption \ref{assu n indexed by T}(i).
Together with Lemma \ref{prop mini eigenvalue}(ii) and Lemma 2.12 in \cite{vandervaart1998}, we can invoke Lemma \ref{prop matrix log expansion} stochastically with $A=\Theta$ and $B=\hat{\Theta}_T-\Theta$:
\begin{equation}
\label{eqn linearise logMt minus logTheta}
\log \hat{\Theta}_T-\log \Theta=\int_{0}^{1}[t(\Theta-I)+I]^{-1}(\hat{\Theta}_T-\Theta)[t(\Theta-I)+I]^{-1}dt +O_p(\|\hat{\Theta}_T-\Theta\|_{\ell_2}^2).
\end{equation}
(We can invoke Lemma \ref{prop matrix log expansion} stochastically because the remainder of the log linearization is zero when the perturbation is zero. Moreover, we have $\|\hat{\Theta}_T-\Theta\|_{\ell_2}\xrightarrow{p}0$ under Assumption \ref{assu n indexed by T}(i).)
Note that (\ref{eqn linearise logMt minus logTheta}) also holds with $\hat{\Theta}_{T}$ replaced by $\hat{\Theta}_{T,D}$ by repeating the same argument. That is,
\begin{equation*}
%\label{eqn linearise logtilde minus logTheta}
\log \hat{\Theta}_{T,D}-\log \Theta=\int_{0}^{1}[t(\Theta-I)+I]^{-1}(\hat{\Theta}_{T,D}-\Theta)[t(\Theta-I)+I]^{-1}dt +O_p(\|\hat{\Theta}_{T,D}-\Theta\|_{\ell_2}^2).
\end{equation*}
Now we can write
\begin{align*}
\frac{\sqrt{T}c^{\intercal}(\hat{\theta}_{T,D}-\theta^0)}{\sqrt{c^{\intercal}\hat{J}_{T,D}c}}&=\frac{\sqrt{T}c^{\intercal}(E^{\intercal}WE)^{-1}E^{\intercal}WD_n^+H(D^{-1/2}\otimes D^{-1/2})\ve (\hat{\Sigma}_T-\Sigma)}{\sqrt{c^{\intercal}\hat{J}_{T,D}c}}\\
&\qquad+\frac{\sqrt{T}c^{\intercal}(E^{\intercal}WE)^{-1}E^{\intercal}WD_n^+\ve O_p(\|\hat{\Theta}_{T,D}-\Theta\|_{\ell_2}^2)}{\sqrt{c^{\intercal}\hat{J}_{T,D}c}}\\
& =:\hat{t}_{D,1}+\hat{t}_{D,2}.
\end{align*}
Define
\[t_{D,1}:=\frac{\sqrt{T}c^{\intercal}(E^{\intercal}WE)^{-1}E^{\intercal}WD_n^+H(D^{-1/2}\otimes D^{-1/2})\ve (\tilde{\Sigma}_T-\Sigma)}{\sqrt{c^{\intercal}J_{D}c}}.\]
To prove Theorem \ref{thm asymptotic normality}, it suffices to show $t_{D,1}\xrightarrow{d}N(0,1)$, $t_{D,1}-\hat{t}_{D,1}=o_p(1)$, and $\hat{t}_{D,2}=o_p(1)$.
\subsubsection{$t_{D,1}\xrightarrow{d}N(0,1)$}
\label{sec asymptotic normality of t1 prime}
We now prove that $t_{D,1}$ is asymptotically distributed as a standard normal.
\begin{align*}
& t_{D,1}=\frac{\sqrt{T}c^{\intercal}(E^{\intercal}WE)^{-1}E^{\intercal}WD_n^+H(D^{-1/2}\otimes D^{-1/2})\ve \del[2] {\frac{1}{T}\sum_{t=1}^{T}\sbr[1]{ (y_t-\mu)(y_t-\mu)^{\intercal}-\mathbb{E}(y_t-\mu)(y_t-\mu)^{\intercal}}}}{\sqrt{c^{\intercal}J_{D}c}}\\
&=\sum_{t=1}^{T}\frac{T^{-1/2}c^{\intercal}(E^{\intercal}WE)^{-1}E^{\intercal}WD_n^+H(D^{-1/2}\otimes D^{-1/2})\ve \sbr[1]{ (y_t-\mu)(y_t-\mu)^{\intercal}-\mathbb{E}(y_t-\mu)(y_t-\mu)^{\intercal}}}{\sqrt{c^{\intercal}J_{D}c}}\\
&=:\sum_{t=1}^{T}U_{D,T,n,t}.
\end{align*}
Define a triangular array of sigma algebras $\{\mathcal{F}_{T,n,t}, t=0,1,2,\ldots, T\}$ by $\mathcal{F}_{T,n,t}:=\mathcal{F}_{t}$ (the only non-standard thing is that this triangular array has one more subscript $n$). It is easy to see that $U_{D,T,n,t}$ is $\mathcal{F}_{T,n,t}$-measurable. We now show that  $\{U_{D,T,n,t}, \mathcal{F}_{T,n,t}\}$ is a martingale difference sequence (i.e., $\mathbb{E}[U_{D,T,n,t}|\mathcal{F}_{T,n,t-1}]=0$ almost surely for $t=1,\ldots, T$). It suffices to show for all $t$
\begin{equation}
\label{eqn U mds}
\mathbb{E}\sbr[1]{(y_t-\mu)(y_t-\mu)^{\intercal}-\mathbb{E}[(y_t-\mu)(y_t-\mu)^{\intercal}]|\mathcal{F}_{T,n,t-1}}=0\qquad a.s..
\end{equation}
This is straightforward via Assumption \ref{assu mds}.
We now check conditions (i)-(iii) of Theorem \ref{thm mcleish clt} in Appendix \ref{sec oldappendixB}.
We first investigate at what rate the denominator $\sqrt{c^{\intercal}J_{D}c}$ goes to zero:
\begin{align*}
& c^{\intercal}J_{D}c=c^{\intercal}(E^{\intercal}WE)^{-1}E^{\intercal}WD_{n}^{+}H(D^{-1/2}\otimes D^{-1/2})V(D^{-1/2}\otimes D^{-1/2})HD_{n}^{+^{\intercal}}WE(E^{\intercal}WE)^{-1}c\\
&\geq \text{mineval}(V)\text{mineval}(D^{-1}\otimes D^{-1})\text{mineval}(H^2)\text{mineval}(D_n^+D_n^{+\intercal})\text{mineval}(W)\text{mineval}((E^{\intercal}WE)^{-1})\\
&=\frac{\text{mineval}(V)\text{mineval}^2(H)}{\text{maxeval}(D\otimes D)\text{maxeval}(D_n^{\intercal}D_n)\text{maxeval}(W^{-1})\text{maxeval}(E^{\intercal}WE)}\\
&\geq \frac{\text{mineval}(V)\text{mineval}^2(H)}{\text{maxeval}(D\otimes D)\text{maxeval}(D_n^{\intercal}D_n)\text{maxeval}(W^{-1})\text{maxeval}(W)\text{maxeval}(E^{\intercal}E)}
\end{align*}
where the first and third inequalities are true by repeatedly invoking the Rayleigh-Ritz theorem. Note that
\begin{equation}
\label{eqn maxeval EE}
\text{maxeval}(E^{\intercal}E)\leq \text{tr}(E^{\intercal}E)\leq s\cdot n,
\end{equation}
where the last inequality is due to Lemma \ref{prop decode E}. For future reference
\begin{equation}
\label{eqn E l2 rate}
\|E\|_{\ell_2}=\|E^{\intercal}\|_{\ell_2}=\sqrt{\text{maxeval}(E^{\intercal}E)}\leq \sqrt{sn}.
\end{equation}
Since the minimum eigenvalue of $H$ is bounded away from zero by an absolute constant by Lemma \ref{prop mini eigenvalue}(i), the maximum eigenvalue of $D$ is bounded from above by an absolute constant (Lemma \ref{lemmaD}(vi)), and $\text{maxeval}[D_n^{\intercal}D_n]$ is bounded from above since  $D_n^{\intercal}D_n$ is a diagonal matrix with diagonal entries either 1 or 2, we have, via Assumption \ref{assu mineval of V}
\begin{equation}
\label{eqn inverse square root G rate}
\frac{1}{\sqrt{c^{\intercal}J_{D}c}}=O(\sqrt{s\cdot n\cdot \kappa(W)}).
\end{equation}
Also note that
\begin{align}
&\|(E^{\intercal}WE)^{-1}E^{\intercal}W^{1/2}\|_{\ell_2}=\sqrt{\text{maxeval}\del[1]{\sbr[1]{(E^{\intercal}WE)^{-1}E^{\intercal}W^{1/2}}^{\intercal}(E^{\intercal}WE)^{-1}E^{\intercal}W^{1/2}}}\notag \\
&=\sqrt{\text{maxeval}\del[1]{(E^{\intercal}WE)^{-1}E^{\intercal}W^{1/2}\sbr[1]{(E^{\intercal}WE)^{-1}E^{\intercal}W^{1/2}}^{\intercal}}}\notag\\
&=\sqrt{\text{maxeval}\del[1]{(E^{\intercal}WE)^{-1}E^{\intercal}W^{1/2}W^{1/2}E(E^{\intercal}WE)^{-1}}}\notag \\
&=\sqrt{\text{maxeval}\del[1]{(E^{\intercal}WE)^{-1}}}=\sqrt{\frac{1}{\text{mineval}(E^{\intercal}WE)}}\leq \sqrt{\frac{1}{\text{mineval}(E^{\intercal}E)\text{mineval}(W)}}\notag\\
&=O\del[1]{\sqrt{\varpi/n}}\sqrt{\|W^{-1}\|_{\ell_2}},\notag
\end{align}
where the second equality is due to the fact that for any matrix $A$, $AA^{\intercal}$ and $A^{\intercal}A$ have the same non-zero eigenvalues, the third equality is due to $(A^{\intercal})^{-1}=(A^{-1})^{\intercal}$, and the last equality is due to Assumption \ref{assu about D and Dhat}(ii). Thus
\begin{align}
\|(E^{\intercal}WE)^{-1}E^{\intercal}W\|_{\ell_2}=O(\sqrt{\varpi \kappa(W)/n}), \label{align Eplus spectral norm}
\end{align}
whence we have
\begin{align}
\enVert[1]{c^{\intercal}(E^{\intercal}WE)^{-1}E^{\intercal}WD_n^+H(D^{-1/2}\otimes D^{-1/2})}_2=O(\sqrt{\varpi \kappa(W)/n}),\label{align Eplus spectral norm +1}
\end{align}
via (\ref{eqn spectral norm of H and Hhat}) and Lemma \ref{lemmaD}(ii).
We now verify (i) and (ii) of Theorem \ref{thm mcleish clt} in Appendix \ref{sec oldappendixB}. We shall use Orlicz norms as defined in \cite{vandervaartWellner1996}: Let $\psi: \mathbb{R}^+\to \mathbb{R}^+$ be a non-decreasing, convex function with $\psi(0)=0$ and $\lim_{x\to \infty}\psi(x)=\infty$, where $\mathbb{R}^+$ denotes the set of nonnegative real numbers. Then, the Orlicz norm of a random variable $X$ is given by
\begin{align*}
\enVert{X}_\psi=\inf\left\{C>0:\mathbb{E}\psi\left(|X|/C\right)\leq 1\right\},
\end{align*}
where $\inf \emptyset =\infty$. We shall use Orlicz norms for  $\psi(x)=\psi_p(x)=e^{x^p}-1$ for $p=1,2$ in this article.
We consider $|U_{D,T,n,t}|$ first.
\begin{align*}
&|U_{D,T,n,t}|=\\
&\envert[3]{\frac{T^{-1/2}c^{\intercal}(E^{\intercal}WE)^{-1}E^{\intercal}WD_n^+H(D^{-1/2}\otimes D^{-1/2})\ve \sbr[1]{ (y_t-\mu)(y_t-\mu)^{\intercal}-\mathbb{E}(y_t-\mu)(y_t-\mu)^{\intercal}}}{\sqrt{c^{\intercal}J_{D}c}}}\\
&\leq \frac{T^{-1/2}\|c^{\intercal}(E^{\intercal}WE)^{-1}E^{\intercal}WD_n^+H(D^{-1/2}\otimes D^{-1/2})\|_2\|\ve \sbr[1]{ (y_t-\mu)(y_t-\mu)^{\intercal}-\mathbb{E}(y_t-\mu)(y_t-\mu)^{\intercal}}\|_2}{\sqrt{c^{\intercal}J_{D}c}}\\
&= O\del[3]{\sqrt{\frac{\varpi s \kappa^2(W)}{T}}}\enVert[1]{  (y_t-\mu)(y_t-\mu)^{\intercal}-\mathbb{E}(y_t-\mu)(y_t-\mu)^{\intercal}}_F\\
&\leq  O\del[3]{\sqrt{\frac{n^2 \varpi s \kappa^2(W)}{T}}}\enVert[1]{  (y_t-\mu)(y_t-\mu)^{\intercal}-\mathbb{E}(y_t-\mu)(y_t-\mu)^{\intercal}}_{\infty}
\end{align*}
where the second equality is due to (\ref{eqn inverse square root G rate}) and (\ref{align Eplus spectral norm +1}).
Consider
\begin{align*}
& \enVert[2]{\enVert[1]{  (y_t-\mu)(y_t-\mu)^{\intercal}-\mathbb{E}(y_t-\mu)(y_t-\mu)^{\intercal}}_{\infty}}_{\psi_1}=\enVert[2]{\max_{1\leq i,j\leq n}\envert[1]{ (y_{t,i}-\mu_i)(y_{t,j}-\mu_j)-\mathbb{E}(y_{t,i}-\mu_i)(y_{t,j}-\mu_j)}}_{\psi_1}\\
&\leq\log (1+n^2) \max_{1\leq i,j\leq n}\enVert[2]{ (y_{t,i}-\mu_i)(y_{t,j}-\mu_j)-\mathbb{E}(y_{t,i}-\mu_i)(y_{t,j}-\mu_j)}_{\psi_1}\\
&\leq2\log (1+n^2) \max_{1\leq i,j\leq n}\enVert[2]{ (y_{t,i}-\mu_i)(y_{t,j}-\mu_j)}_{\psi_1}
\end{align*}
where the first inequality is due to Lemma 2.2.2 in \cite{vandervaartWellner1996}.
Assumption \ref{assu subgaussian vector}(i) with $r_1=2$ gives $\mathbb{E}\sbr[1]{\exp(K_2|y_{t,i}|^2)}\leq K_1$ for all $i$. Then
\begin{align*}
& \mathbb{P}\del[1]{|(y_{t,i}-\mu_i)(y_{t,j}-\mu_j)|\geq \epsilon}\leq \mathbb{P}\del[1]{|y_{t,i}-\mu_i|\geq \sqrt{\epsilon}}+\mathbb{P}\del[1]{|y_{t,j}-\mu_j|\geq \sqrt{\epsilon}}\\
&\leq 2 \exp \sbr[1]{1-(\sqrt{\epsilon}/c_1)^2}=:Ke^{-C\epsilon}
\end{align*}
where the second inequality is due to Lemma \ref{lemmaexponentialtail}(iii). It follows from Lemma 2.2.1 in \cite{vandervaartWellner1996} that $\|(y_{t,i}-\mu_i)(y_{t,j}-\mu_j)\|_{\psi_1}\leq (1+K)/C$ for all $i,j,t$. Thus
\begin{align*}
&\enVert[2]{\max_{1\leq t\leq T}|U_{D,T,n,t}|}_{\psi_1}\leq \log(1+T) \max_{1\leq t\leq T}\enVert[2]{U_{D,T,n,t}}_{\psi_1}\\
&=O\del[3]{\log(1+T)\sqrt{\frac{n^2 \varpi s \kappa^2(W)}{T}}}\max_{1\leq t\leq T}\enVert[2]{\enVert[1]{  (y_t-\mu)(y_t-\mu)^{\intercal}-\mathbb{E}(y_t-\mu)(y_t-\mu)^{\intercal}}_{\infty}}_{\psi_1}\\
&=O\del[3]{\log(1+T)\log (1+n^2)\sqrt{\frac{n^2 \varpi s \kappa^2(W)}{T}}}\max_{1\leq t\leq T} \max_{1\leq i,j\leq n}\enVert[2]{ (y_{t,i}-\mu_i)(y_{t,j}-\mu_j)}_{\psi_1}\\
&=O\del[3]{\log(1+T)\log (1+n^2)\sqrt{\frac{n^2 \varpi s \kappa^2(W)}{T}}}=O\del[3]{\sqrt{\frac{n^2 \varpi s \kappa^2(W)\log^2(1+T)\log ^2(1+n^2)}{T}}}\\
&=o(1)
\end{align*}
where the last equality is due to Assumption \ref{assu n indexed by T}(ii). Since $\|U\|_{L_r}\leq r!\|U\|_{\psi_1}$ for any random variable $U$ (\cite{vandervaartWellner1996}, p95), we conclude that (i) and (ii) of Theorem \ref{thm mcleish clt} in Appendix \ref{sec oldappendixB} are satisfied.
We now verify condition (iii) of Theorem \ref{thm mcleish clt} in Appendix \ref{sec oldappendixB}. Since we have already shown in (\ref{eqn inverse square root G rate}) that  $s n\kappa(W)c^{\intercal}J_{D}c$ is bounded away from zero by an absolute constant, it suffices to show
\begin{align*}
& s n\kappa(W)\cdot  \envert[3]{\frac{1}{T}\sum_{t=1}^{T}\del[2]{c^{\intercal}(E^{\intercal}WE)^{-1}E^{\intercal}WD_n^+H(D^{-1/2}\otimes D^{-1/2})u_t}^2-c^{\intercal}J_{D}c}=o_p(1),
\end{align*}
where $u_t:=\ve \sbr[1]{ (y_t-\mu)(y_t-\mu)^{\intercal}-\mathbb{E}(y_t-\mu)(y_t-\mu)^{\intercal}}$. Note that
\begin{align*}
& s n\kappa(W)\cdot  \envert[3]{\frac{1}{T}\sum_{t=1}^{T}\del[2]{c^{\intercal}(E^{\intercal}WE)^{-1}E^{\intercal}WD_n^+H(D^{-1/2}\otimes D^{-1/2})u_t}^2-c^{\intercal}J_{D}c}\\
&\leq s n\kappa(W)\enVert[3]{\frac{1}{T}\sum_{t=1}^{T}u_tu_t^{\intercal}-V}_{\infty}\|c^{\intercal}(E^{\intercal}WE)^{-1}E^{\intercal}WD_n^+H(D^{-1/2}\otimes D^{-1/2})\|_1^2\\
&\leq s n^3\kappa(W)\enVert[3]{\frac{1}{T}\sum_{t=1}^{T}u_tu_t^{\intercal}-V}_{\infty}\|c^{\intercal}(E^{\intercal}WE)^{-1}E^{\intercal}WD_n^+H(D^{-1/2}\otimes D^{-1/2})\|_2^2\\
&\leq s n^3\kappa(W)\enVert[3]{\frac{1}{T}\sum_{t=1}^{T}u_tu_t^{\intercal}-V}_{\infty}\|(E^{\intercal}WE)^{-1}E^{\intercal}W\|_{\ell_2}^2\|D_n^+\|_{\ell_2}^2\|H\|_{\ell_2}^2\|D^{-1/2}\otimes D^{-1/2}\|_{\ell_2}^2\\
&=O_p(s n^3\kappa(W)) \sqrt{\frac{\log n}{T}} \cdot \frac{\varpi\kappa(W)}{n}=O_p\del[3]{\sqrt{\frac{s^2n^4\kappa^4(W) \log n\cdot\varpi^2}{T}}}=o_p(1)
\end{align*}
where the first equality is due to Lemma \ref{lemmaD}(ii), Lemma \ref{lemma l2norm of kronecker product} in Appendix \ref{sec oldappendixB}, (\ref{eqn spectral norm of H and Hhat}), (\ref{align Eplus spectral norm}), (\ref{eqn spectral norm for Dnplus and Dn}), and the fact that $\enVert[1]{T^{-1}\sum_{t=1}^{T}u_tu_t^{\intercal}-V}_{\infty}=O_p(\sqrt{\frac{\log n}{T}})$, which can be deduced from the proof of Lemma \ref{lemma rate for hatV-V infty} in SM \ref{sec A5},\footnote{To see this, write
\begin{align*}
& \frac{1}{T}u_tu_t^{\intercal}-V=\frac{1}{T}\sum_{t=1}^{T}\sbr[1]{(y_t-\mu)(y_t-\mu)^{\intercal}\otimes (y_t-\mu)(y_t-\mu)^{\intercal}}-\mathbb{E}\sbr[1]{(y_t-\mu)(y_t-\mu)^{\intercal}\otimes (y_t-\mu)(y_t-\mu)^{\intercal}}\\
&\quad  -\frac{1}{T}\sum_{t=1}^{T}\sbr[1]{(y_t-\mu)\otimes (y_t-\mu)}\cdot\mathbb{E}\sbr[1]{(y_t-\mu)^{\intercal}\otimes (y_t-\mu)^{\intercal}}+\mathbb{E}\sbr[1]{(y_t-\mu)\otimes (y_t-\mu)}\cdot\mathbb{E}\sbr[1]{(y_t-\mu)^{\intercal}\otimes (y_t-\mu)^{\intercal}}\\
&\quad  -\mathbb{E}\sbr[1]{(y_t-\mu)\otimes (y_t-\mu)}\cdot \frac{1}{T}\sum_{t=1}^{T}\sbr[1]{(y_t-\mu)^{\intercal}\otimes (y_t-\mu)^{\intercal}}+\mathbb{E}\sbr[1]{(y_t-\mu)\otimes (y_t-\mu)}\cdot\mathbb{E}\sbr[1]{(y_t-\mu)^{\intercal}\otimes (y_t-\mu)^{\intercal}}.
\end{align*}
Then many parts of the proof of Lemma \ref{lemma rate for hatV-V infty} in SM \ref{sec A5} could be recycled.} the last equality is due to Assumption \ref{assu n indexed by T}(ii). Thus condition (iii) of Theorem \ref{thm mcleish clt} in Appendix \ref{sec oldappendixB} is verified and $t_{D,1}\xrightarrow{d}N(0,1)$.
\subsubsection{$t_{D,1}-\hat{t}_{D,1}=o_p(1)$}
We now show that $t_{D,1}-\hat{t}_{D,1}=o_p(1)$. Let $A_D$ and $\hat{A}_D$ denote the numerators of $t_{D,1}$ and $\hat{t}_{D,1}$, respectively.
\begin{align*}
t_{D,1}-\hat{t}_{D,1} &=\frac{A_D}{\sqrt{c^{\intercal}J_{D}c}}-\frac{\hat{A}_D}{\sqrt{c^{\intercal}\hat{J}_{T,D}c}}=\frac{\sqrt{s n \kappa(W)}A_D}{\sqrt{s n \kappa(W)c^{\intercal}J_{D}c}}-\frac{\sqrt{s n \kappa(W)}\hat{A}_D}{\sqrt{s n \kappa(W)c^{\intercal}\hat{J}_{T,D}c}}.
\end{align*}
Since we have already shown in (\ref{eqn inverse square root G rate}) that  $s n\kappa(W)c^{\intercal}J_{D}c$ is bounded away from zero by an absolute constant, it suffices to show the denominators as well as numerators of $t_{D,1}$ and $\hat{t}_{D,1}$ are asymptotically equivalent.
\subsubsection{Denominators of $t_{D,1}$ and $\hat{t}_{D,1}$}
We first show that the denominators of $t_{D,1}$ and $\hat{t}_{D,1}$ are asymptotically equivalent, i.e.,
\[s n \kappa(W)|c^{\intercal}\hat{J}_{T,D}c-c^{\intercal}J_{D}c|=o_p(1).\]
Define
\[c^{\intercal}\tilde{J}_{T,D}c:=c^{\intercal}(E^{\intercal}WE)^{-1}E^{\intercal}WD_{n}^{+}\hat
{H}_{T,D}(D^{-1/2}\otimes D^{-1/2})V(D^{-1/2}\otimes D^{-1/2})\hat{H}_{T,D}D_{n}^{+^{\intercal}}WE(E^{\intercal
}WE)^{-1}c.\]
By the triangular inequality: $|sn\kappa(W)c^{\intercal}\hat{J}_{T,D}c-sn\kappa(W)c^{\intercal}J_{D}c|\leq |sn\kappa(W)c^{\intercal}\hat{J}_{T,D}c-sn\kappa(W)c^{\intercal}\tilde{J}_{T,D}c|+|sn\kappa(W)c^{\intercal}\tilde{J}_{T,D}c-sn\kappa(W)c^{\intercal}J_{D}c|$. First, we prove $|sn\kappa(W)c^{\intercal}\hat{J}_{T,D}c-sn\kappa(W)c^{\intercal}\tilde{J}_{T,D}c|=o_p(1)$.
\begin{align*}
& sn\kappa(W)|c^{\intercal}\hat{J}_{T,D}c-c^{\intercal}\tilde{J}_{T,D}c|\\
&=sn\kappa(W)|c^{\intercal}(E^{\intercal}WE)^{-1}E^{\intercal}WD_{n}^{+}\hat
{H}_{T,D}(D^{-1/2}\otimes D^{-1/2})\hat{V}_T(D^{-1/2}\otimes D^{-1/2})\hat{H}_{T,D}D_{n}^{+^{\intercal}}WE(E^{\intercal
}WE)^{-1}c\\
&\qquad -c^{\intercal}(E^{\intercal}WE)^{-1}E^{\intercal}WD_{n}^{+}\hat
{H}_{T,D}(D^{-1/2}\otimes D^{-1/2})V(D^{-1/2}\otimes D^{-1/2})\hat{H}_{T,D}D_{n}^{+^{\intercal}}WE(E^{\intercal
}WE)^{-1}c|\\
&=sn\kappa(W)\\
&\quad \cdot |c^{\intercal}(E^{\intercal}WE)^{-1}E^{\intercal}WD_{n}^{+}\hat
{H}_{T,D}(D^{-1/2}\otimes D^{-1/2})(\hat{V}_T-V)(D^{-1/2}\otimes D^{-1/2})\hat{H}_{T,D}D_{n}^{+^{\intercal}}WE(E^{\intercal
}WE)^{-1}c|\\
&\leq sn\kappa(W) \|\hat{V}_T-V\|_{\infty}\|(D^{-1/2}\otimes D^{-1/2})\hat{H}_{T,D}D_{n}^{+^{\intercal}}WE(E^{\intercal
}WE)^{-1}c\|_1^2\\
&\leq sn^3\kappa(W) \|\hat{V}_T-V\|_{\infty}\|(D^{-1/2}\otimes D^{-1/2})\hat{H}_{T,D}D_{n}^{+^{\intercal}}WE(E^{\intercal
}WE)^{-1}c\|_2^2\\
&\leq sn^3\kappa(W) \|\hat{V}_T-V\|_{\infty}\|(D^{-1/2}\otimes D^{-1/2})\|_{\ell_2}^2\|\hat{H}_{T,D}\|_{\ell_2}^2\|D_{n}^{+^{\intercal}}\|_{\ell_2}^2\|WE(E^{\intercal
}WE)^{-1}\|_{\ell_2}^2\\
&=O_p(sn^2\kappa^2(W)\varpi)\|\hat{V}_T-V\|_{\infty}=O_p\del [3]{\sqrt{\frac{n^4\kappa^4(W)s^2\varpi^2\log n}{T}}}=o_p(1),
\end{align*}
where $\|\cdot\|_{\infty}$ denotes the absolute elementwise maximum, the third equality is due to Lemma \ref{lemmaD}(ii), Lemma \ref{lemma l2norm of kronecker product} in Appendix \ref{sec oldappendixB}, (\ref{eqn spectral norm of H and Hhat}), (\ref{align Eplus spectral norm}), and (\ref{eqn spectral norm for Dnplus and Dn}), the second last equality is due to Lemma \ref{lemma rate for hatV-V infty} in SM \ref{sec A5}, and the last equality is due to Assumption \ref{assu n indexed by T}(ii).
We now prove $s n\kappa(W)|c^{\intercal}\tilde{J}_{T,D}c-c^{\intercal}J_{D}c|=o_p(1)$.
\begin{align}
& sn\kappa(W)|c^{\intercal}\tilde{J}_{T,D}c-c^{\intercal}J_{D}c|\notag\\
&=sn\kappa(W)|c^{\intercal}(E^{\intercal}WE)^{-1}E^{\intercal}WD_{n}^{+}\hat{H}_{T,D}(D^{-1/2} \otimes D^{-1/2} )V(D^{-1/2} \otimes D^{-1/2} )\hat{H}_{T,D}D_{n}^{+^{\intercal}}WE(E^{\intercal}WE)^{-1}c\notag \\
&\qquad -c^{\intercal}(E^{\intercal}WE)^{-1}E^{\intercal}WD_{n}^{+}H(D^{-1/2} \otimes D^{-1/2} )V(D^{-1/2} \otimes D^{-1/2} )HD_{n}^{+^{\intercal}}WE(E^{\intercal}WE)^{-1}c| \notag\\
&\leq s n\kappa(W) \envert[1]{\text{maxeval}\sbr[1]{(D^{-1/2} \otimes D^{-1/2} )V(D^{-1/2} \otimes D^{-1/2} )}}\|(\hat{H}_{T,D}-H)D_{n}^{+^{\intercal}}WE(E^{\intercal}WE)^{-1}c\|_2^2\notag \\
&\quad+2 s n \kappa(W)\|(D^{-1/2} \otimes D^{-1/2} )V(D^{-1/2} \otimes D^{-1/2} )HD_{n}^{+^{\intercal}}WE(E^{\intercal}WE)^{-1}c\|_2\notag\\
&\qquad \cdot\|(\hat{H}_{T,D}-H)D_{n}^{+^{\intercal}}WE(E^{\intercal}WE)^{-1}c\|_2\label{align tildeG-G}
\end{align}
where the inequality is due to Lemma \ref{lemma vandegeer} in Appendix \ref{sec oldappendixB}. We consider the first term of (\ref{align tildeG-G}) first.
\begin{align*}
& s n\kappa(W) \envert[1]{\text{maxeval}\sbr[1]{(D^{-1/2} \otimes D^{-1/2} )V(D^{-1/2} \otimes D^{-1/2} )}}\|(\hat{H}_{T,D}-H)D_{n}^{+^{\intercal}}WE(E^{\intercal}WE)^{-1}c\|_2^2\\
& =O(sn\kappa(W))\|\hat{H}_{T,D}-H\|_{\ell_2}^2\|D_{n}^{+^{\intercal}}\|_{\ell_2}^2\|WE(E^{\intercal}WE)^{-1}\|^2_{\ell_2}\\
&= O_p(sn\kappa^2(W)\varpi/T)=o_p(1),
\end{align*}
where the second last equality is due to (\ref{eqn spectral norm of H and Hhat}), (\ref{eqn spectral norm for Dnplus and Dn}), and (\ref{align Eplus spectral norm}), and the last equality is due to Assumption \ref{assu n indexed by T}(ii). We now consider the second term of (\ref{align tildeG-G}).
\begin{align*}
& 2 s n \kappa(W)\|(D^{-1/2} \otimes D^{-1/2} )V(D^{-1/2} \otimes D^{-1/2} )HD_{n}^{+^{\intercal}}WE(E^{\intercal}WE)^{-1}c\|_2\\
&\qquad \cdot \|(\hat{H}_{T,D}-H)D_{n}^{+^{\intercal}}WE(E^{\intercal}WE)^{-1}c\|_2\\
&\leq O(sn\kappa(W))\|H\|_{\ell_2}\|\hat{H}_{T,D}-H\|_{\ell_2}\|D^{+\intercal}_n\|_{\ell_2}^2\|WE(E^{\intercal}WE)^{-1}c\|_2^2=O(\sqrt{n\kappa^4(W)s^2\varpi^2/T})=o_p(1),
\end{align*}
where the first equality is due to (\ref{eqn spectral norm of H and Hhat}), (\ref{eqn spectral norm for Dnplus and Dn}), and (\ref{align Eplus spectral norm}), and the last equality is due to Assumption \ref{assu n indexed by T}(ii). We have proved $|sn\kappa(W)c^{\intercal}\tilde{J}_{T,D}c-sn\kappa(W)c^{\intercal}J_{D}c|=o_p(1)$ and hence $|sn\kappa(W)c^{\intercal}\hat{J}_{T,D}c-sn\kappa(W)c^{\intercal}J_{D}c|=o_p(1)$.
\subsubsection{Numerators of $t_{D,1}$ and $\hat{t}_{D,1}$}
We now show that numerators of $t_{D,1}$ and $\hat{t}_{D,1}$ are asymptotically equivalent, i.e.,
\[\sqrt{sn\kappa(W)}|A_D-\hat{A}_D|=o_p(1).\]
This is relatively straight forward.
\begin{align*}
&\sqrt{Tsn\kappa(W)}\envert[1]{c^{\intercal}(E^{\intercal}WE)^{-1}E^{\intercal}WD_n^+H(D^{-1/2}\otimes D^{-1/2})\ve (\hat{\Sigma}_T-\Sigma-\tilde{\Sigma}_T+\Sigma)}\\
&=\sqrt{Tsn\kappa(W)}\envert[1]{c^{\intercal}(E^{\intercal}WE)^{-1}E^{\intercal}WD_n^+H(D^{-1/2}\otimes D^{-1/2})\ve (\hat{\Sigma}_T-\tilde{\Sigma}_T)}\\
&=\sqrt{Tsn\kappa(W)}\envert[1]{c^{\intercal}(E^{\intercal}WE)^{-1}E^{\intercal}WD_n^+H(D^{-1/2}\otimes D^{-1/2})\ve \sbr[1]{(\bar{y}-\mu)(\bar{y}-\mu)^{\intercal}}}\\
&\leq \sqrt{Tsn\kappa(W)}\|(E^{\intercal}WE)^{-1}E^{\intercal}W\|_{\ell_2}\|D_n^+\|_{\ell_2}\|H\|_{\ell_2}\|D^{-1/2}\otimes D^{-1/2}\|_{\ell_2}\|\ve \sbr[1]{(\bar{y}-\mu)(\bar{y}-\mu)^{\intercal}}\|_2\\
&=O(\sqrt{Tsn\kappa(W)})\sqrt{\varpi \kappa(W)/n}\|(\bar{y}-\mu)(\bar{y}-\mu)^{\intercal}\|_F\\
&\leq O(\sqrt{Tsn\kappa(W)})\sqrt{\varpi \kappa(W)/n}n\|(\bar{y}-\mu)(\bar{y}-\mu)^{\intercal}\|_{\infty}\\
&=O(\sqrt{Tsn^2\kappa^2(W)\varpi})\max_{1\leq i,j\leq n}\envert[1]{(\bar{y}-\mu)_i(\bar{y}-\mu)_j}=O_p(\sqrt{Tsn^2\kappa^2(W)\varpi})\log n/T\\
&=O_p\del[3]{\sqrt{\frac{\log^3n \cdot n^2\kappa^2(W)\varpi}{T}}}=o_p(1),
\end{align*}
where the third equality is due to (\ref{eqn spectral norm of H and Hhat}), (\ref{eqn spectral norm for Dnplus and Dn}), and (\ref{align Eplus spectral norm}), the third last equality is due to (\ref{eqn rate for hatV-V infinity part D}) in SM \ref{sec A5}, and the last equality is due to Assumption \ref{assu n indexed by T}(ii).
\subsubsection{$\hat{t}_{D,2}=o_p(1)$}
Write
\[\hat{t}_{D,2}=\frac{\sqrt{T}\sqrt{s n\kappa(W)}c^{\intercal}(E^{\intercal}WE)^{-1}E^{\intercal}WD_n^+\ve O_p(\|\hat{\Theta}_{T,D}-\Theta\|_{\ell_2}^2)}{\sqrt{s n\kappa(W)c^{\intercal}\hat{J}_{T,D}c}}.\]
Since the denominator of the preceding equation is bounded away from zero by an absolute constant with probability approaching one by (\ref{eqn inverse square root G rate}) and that $|sn\kappa(W)c^{\intercal}\hat{J}_{T,D}c-sn\kappa(W)c^{\intercal}J_{D}c|=o_p(1)$, it suffices to show
\[\sqrt{T}\sqrt{sn\kappa(W)}c^{\intercal}(E^{\intercal}WE)^{-1}E^{\intercal}WD_n^+\ve O_p(\|\hat{\Theta}_{T,D}-\Theta\|_{\ell_2}^2)=o_p(1).\]
This is straightforward:
\begin{align*}
&|\sqrt{Tsn\kappa(W)}c^{\intercal}(E^{\intercal}WE)^{-1}E^{\intercal}WD^+_n\ve O_p(\|\hat{\Theta}_{T,D}-\Theta\|_{\ell_2}^2)|\\
&\leq \sqrt{Tsn\kappa(W)}\|c^{\intercal}(E^{\intercal}WE)^{-1}E^{\intercal}WD^+_n\|_2\|\ve O_p(\|\hat{\Theta}_{T,D}-\Theta\|_{\ell_2}^2)\|_2\\
&= O(\sqrt{Ts\varpi}\kappa(W))\|O_p(\|\hat{\Theta}_{T,D}-\Theta\|_{\ell_2}^2)\|_F= O(\sqrt{Ts\varpi n}\kappa(W))\|O_p(\|\hat{\Theta}_{T,D}-\Theta\|_{\ell_2}^2)\|_{\ell_2}\\
&=O(\sqrt{Ts\varpi n}\kappa(W))O_p(\|\hat{\Theta}_{T,D}-\Theta\|_{\ell_2}^2)=O_p\del[3]{ \frac{\kappa(W)\sqrt{Ts\varpi n}n}{T}} =O_p\del[3]{ \sqrt{\frac{s\varpi n^3\kappa^2(W)}{T}}}=o_p(1),
\end{align*}
where the last equality is due to Assumption \ref{assu n indexed by T}(ii).
\end{proof}

\subsection{Auxiliary Lemmas}

\label{sec oldappendixB}

This subsection of Appendix contains auxiliary lemmas which have been used in
other subsections of Appendix. We first review definitions of nets and
covering numbers.

\begin{definition}
[Nets and covering numbers]
%Last checked on 20160510; highly trustworthy
Let $(T,d)$ be a metric space and fix $\varepsilon>0$.

\begin{enumerate}
[(i)]

\item A subset $\mathcal{N}_{\varepsilon}$ of $T$ is called an $\varepsilon
$-\textit{net} of $T$ if every point $x \in T$ satisfies $d(x,y)\leq
\varepsilon$ for some $y\in\mathcal{N}_{\varepsilon}$.

\item The minimal cardinality of an $\varepsilon$-net of $T$ is denoted
$|\mathcal{N}_{\varepsilon}|$ and is called the \textit{covering number} of $T
$ (at scale $\varepsilon$). Equivalently, $|\mathcal{N}_{\varepsilon}|$ is the
minimal number of balls of radius $\varepsilon$ and with centers in $T$ needed
to cover $T$.
\end{enumerate}
\end{definition}

%\bigskip

\begin{lemma}
\label{lemma vol ball}
%Last checked on 20160510; highly trustworthy
The unit Euclidean sphere $\{x\in\mathbb{R}^{n}:\|x\|_{2}= 1\}$ equipped with
the Euclidean metric satisfies for every $\varepsilon>0$ that
\[
|\mathcal{N}_{\varepsilon}|\leq\left(  1+\frac{2}{\varepsilon} \right)  ^{n}.
\]

\end{lemma}

\begin{proof}
See \cite{vershynin2011} Lemma 5.2 p8.
\end{proof}

%\bigskip

Recall that for a symmetric $n\times n$ matrix $A$, its $\ell_{2}$ spectral
norm can be written as: $\Vert A\Vert_{\ell_{2}}=\max_{\Vert x\Vert_{2}%
=1}|x^{\intercal}Ax|$.

\begin{lemma}
\label{lemma symmetric matrix spectral norm}
%Last checked on 20160510; highly trustworthy
Let $A$ be a symmetric $n\times n$ matrix, and let $\mathcal{N}_{\varepsilon}$
be an $\varepsilon$-net of the unit sphere $\{x\in\mathbb{R}^{n}:\Vert
x\Vert_{2}=1\}$ for some $\varepsilon\in\lbrack0,1)$. Then
\[
\Vert A\Vert_{\ell_{2}}\leq\frac{1}{1-2\varepsilon}\max_{x\in\mathcal{N}%
_{\varepsilon}}|x^{\intercal}Ax|.
\]

\end{lemma}

\begin{proof}
See \cite{vershynin2011} Lemma 5.4 p8.
\end{proof}

\bigskip

The following theorem is a version of Bernstein's inequality which
accommodates strong mixing time series.

\begin{thm}
[Theorem 1 of \cite{merlevedepeligradrio2011}]\label{thmbernsteininequality}
Let $\{X_{t}\}_{t\in\mathbb{Z}}$ be a sequence of centered real-valued random
variables. Suppose that for every $\epsilon\geq0$, there exist absolute
constants $\gamma_{2}\in(0,+\infty]$ and $b\in(0,+\infty)$ such that
\[
\sup_{t\geq1} \mathbb{P}(|X_{t}|\geq\epsilon)\leq\exp
\sbr[1]{1-(\epsilon/b)^{\gamma_2}}.
\]
Moreover, assume its alpha mixing coefficient $\alpha(h)$ satisfies
\[
\alpha(h)\leq\exp(-ch^{\gamma_{1}}),\qquad h\in\mathbb{N}%
\]
for absolute constants $c>0$ and $\gamma_{1}>0$. Define $\gamma$ by
$1/\gamma:=1/\gamma_{1}+1/\gamma_{2}$; constants $\gamma_{1}$ and $\gamma_{2}$
need to be restricted to make sure $\gamma<1$. Then, for any $T\geq4$, there
exist positive constants $C_{1}, C_{2}, C_{3}, C_{4}, C_{5}$ depending only on
$b,c,\gamma_{1},\gamma_{2}$ such that, for every $\epsilon\geq0$,
\begin{align*}
\mathbb{P}\del [3]{\envert[3]{\frac{1}{T}\sum_{t=1}^{T}X_t}\geq \epsilon}  &
\leq T\exp\del[3]{-\frac{(T\epsilon)^{\gamma}}{C_1}}+\exp
\del[3]{-\frac{(T\epsilon)^2}{C_2(1+C_3T)}}+\exp
\sbr[3]{-\frac{(T\epsilon)^2}{C_4T}\exp \del [3]{\frac{(T\epsilon)^{\gamma(1-\gamma)}}{C_5 (\log (T\epsilon))^{\gamma}}}}.
\end{align*}

\end{thm}

%\begin{thm}[Bernstein's inequality]
%Last checked on 20160511; highly trustworthy
%We let $Z_{1},\ldots, Z_{T}$ be independent random variables, satisfying for positive constants $A$ and $\sigma_{0}^{2}$
%\[\mathbb{E}Z_{t}=0 \quad\forall t, \quad\frac{1}{T}\sum_{t=1}^{T}\mathbb{E}|Z_{t}|^{m}\leq\frac{m!}{2}A^{m-2}\sigma_{0}^{2}, \quad m=2,3,\ldots.\]
%Let $\epsilon>0$ be arbitrary. Then
%\[\mathbb{P}%
%\del[3]{\envert[3]{ \frac{1}{T}\sum_{t=1}^{T}Z_t } \geq \sigma_0^2\left[ A\epsilon+\sqrt{2\epsilon}\right]}\leq
%2e^{-T\sigma_{0}^{2}\epsilon}.
%\]
%\end{thm}

%\begin{proof}
%Slightly adapted from \cite{buhlmannvandegeer2011} p487.
%\end{proof}

\bigskip

We can use the preceding theorem to establish a rate for the maximum.

\begin{lemma}
\label{lemmabernsteinrate} Suppose that we have for $1\leq i\leq n$, for every
$\epsilon\geq0$,
\begin{align*}
\mathbb{P}\del [3]{\envert[3]{\frac{1}{T}\sum_{t=1}^{T}X_{t,i}}\geq \epsilon}
&  \leq T\exp\del[3]{-\frac{(T\epsilon)^{\gamma}}{C_1}}+\exp
\del[3]{-\frac{(T\epsilon)^2}{C_2(1+C_3T)}}+\exp
\sbr[3]{-\frac{(T\epsilon)^2}{C_4T}\exp \del [3]{\frac{(T\epsilon)^{\gamma(1-\gamma)}}{C_5 (\log (T\epsilon))^{\gamma}}}}.
\end{align*}
Suppose $\log n=o(T^{\frac{\gamma}{2-\gamma}})$ if $n>T$. Then
\[
\max_{1\leq i\leq n}\envert[3]{ \frac{1}{T}\sum_{t=1}^{T}X_{t,i} }=O_{p}%
\del[3]{\sqrt{\frac{\log n}{T}}}.
\]

\end{lemma}

\begin{proof}
\begin{align*}
&\mathbb{P}\del [3]{\max_{1\leq i\leq n}\envert[3]{\frac{1}{T}\sum_{t=1}^{T}X_{t,i}}\geq \epsilon}\leq \sum_{i=1}^{n}\mathbb{P}\del [3]{\envert[3]{\frac{1}{T}\sum_{t=1}^{T}X_{t,i}}\geq \epsilon}\\
&\leq nT\exp \del[3]{-\frac{(T\epsilon)^{\gamma}}{C_1}}+n\exp \del[3]{-\frac{(T\epsilon)^2}{C_2(1+C_3T)}} +n\exp \sbr[3]{-\frac{(T\epsilon)^2}{C_4T}\exp \del [3]{\frac{(T\epsilon)^{\gamma(1-\gamma)}}{C_5 (\log (T\epsilon))^{\gamma}}}}
\end{align*}
We shall choose $\epsilon=C\sqrt{\log n/T}$ for some $C>0$ and consider the three terms on the right side of inequality separately. We consider the first term for the case $n\leq T$
\begin{align*}
& nT\exp \del[3]{-\frac{(T\epsilon)^{\gamma}}{C_1}}=\exp \del[3]{\log (nT)-\frac{C^{\gamma}}{C_1}(T\log n)^{\gamma/2}}=\exp \sbr[3]{(T\log n)^{\gamma/2}\del [3]{\frac{\log (nT)}{(T\log n)^{\gamma/2}}-\frac{C^{\gamma}}{C_1}}}\\
&\leq \exp \sbr[3]{(T\log n)^{\gamma/2}\del [3]{\frac{2\log T}{(T\log n)^{\gamma/2}}-\frac{C^{\gamma}}{C_1}}}=\exp \sbr[3]{(T\log n)^{\gamma/2}\del [3]{o(1)-\frac{C^{\gamma}}{C_1}}}=o(1),
\end{align*}
for large enough $C$. We next consider the first term for the case $n> T$
\begin{align*}
& nT\exp \del[3]{-\frac{(T\epsilon)^{\gamma}}{C_1}}=\exp \del[3]{\log (nT)-\frac{C^{\gamma}}{C_1}(T\log n)^{\gamma/2}}=\exp \sbr[3]{(T\log n)^{\gamma/2}\del [3]{\frac{\log (nT)}{(T\log n)^{\gamma/2}}-\frac{C^{\gamma}}{C_1}}}\\
&\leq \exp \sbr[3]{(T\log n)^{\gamma/2}\del [3]{\frac{2\log n}{(T\log n)^{\gamma/2}}-\frac{C^{\gamma}}{C_1}}}=\exp \sbr[3]{(T\log n)^{\gamma/2}\del [3]{o(1)-\frac{C^{\gamma}}{C_1}}}=o(1),
\end{align*}
for large enough $C$ given the assumption $\log n=o(T^{\frac{\gamma}{2-\gamma}})$.
We consider the second term.
\begin{align*}
& n\exp \del[3]{-\frac{(T\epsilon)^2}{C_2(1+C_3T)}}=\exp \del[3]{\log n-\frac{C^2\log n}{C_2/T+C_2C_3}}=\exp \sbr[3]{\log n\del[3]{1-\frac{C^2}{C_2/T+C_2C_3}}}\\
&=o(1)
\end{align*}
for large enough $C$.
We consider the third term.
\begin{align*}
& n\exp \sbr[3]{-\frac{(T\epsilon)^2}{C_4T}\exp \del [3]{\frac{(T\epsilon)^{\gamma(1-\gamma)}}{C_5 (\log (T\epsilon))^{\gamma}}}}\leq n\exp \sbr[3]{-\frac{(T\epsilon)^2}{C_4T}\exp \del [3]{\frac{(T\epsilon)^{\gamma(1-\gamma)}}{C_5 (T\epsilon)^{\gamma}}}}\\
&=n\exp \sbr[3]{-\frac{(T\epsilon)^2}{C_4T}\exp \del [3]{\frac{1}{C_5 (T\epsilon)^{\gamma^2}}}}=n\exp \sbr[3]{-\frac{(T\epsilon)^2}{C_4T}(1+o(1))}\\
&=\exp \sbr[3]{\log n-\frac{C^2\log n}{C_4}(1+o(1))}=o(1),
\end{align*}
for large enough $C$. This yields the result.
\end{proof}

%\begin{proposition}
%\label{prop Bernstein follow up rate trick} Suppose via Bernstein's inequality that we have for $1\leq i\leq n$
%\[\mathbb{P}\del[4]{\envert[3]{ \frac{1}{T}\sum_{t=1}^{T}Z_{t,i} } \geq \sigma^2_0\sbr[1]{ K\epsilon+\sqrt{2\epsilon}}}\leq 2e^{-T\sigma^{2}_{0}\epsilon}.\]
%Then
%\[\max_{1\leq i\leq n}\envert[3]{ \frac{1}{T}\sum_{t=1}^{T}Z_{t,i} }=O_{p}\del[3]{\frac{\log n}{T}\vee \sqrt{\frac{\log n}{T}}}.\]
%\end{proposition}

%\begin{proof}
%We need to use joint asymptotics $n,T\to \infty$. We shall use the preceding inequality with $\epsilon=(2\log n)/(T\sigma_0^2)$. Fix $\varepsilon>0$. These exist $N_{\varepsilon}:=2/\varepsilon$, $T_{\varepsilon}$ and $M_{\varepsilon}:=\max (4K, 4\sigma_0)$ such that for all $n>N_{\varepsilon}$ and $T>T_{\varepsilon}$ we have
%\begin{align*}
%&\mathbb{P}\del[4]{\max_{1\leq i\leq n}\envert[3]{ \frac{1}{T}\sum_{t=1}^{T}Z_{t,i} } \geq M_{\varepsilon}\del[3]{\frac{\log n}{T}\vee \sqrt{\frac{\log n}{T}}}}\\
%&\leq \sum_{i=1}^{n}\mathbb{P}\del[4]{\envert[3]{ \frac{1}{T}\sum_{t=1}^{T}Z_{t,i} } \geq \sigma^2_0\sbr[1]{ K\epsilon+\sqrt{2\epsilon}}}\leq 2e^{\log n-2\log n}=\frac{2}{n}< \varepsilon.
%\end{align*}
%\end{proof}

\bigskip

\begin{lemma}
\label{lemma square root difference spectral norm} Let $A, B$ be $n\times n$
positive semidefinite matrices and not both singular. Then
\[
\|A-B\|_{\ell_{2}}\leq\frac{\|A^{2}-B^{2}\|_{\ell_{2}}}{\text{mineval}%
(A)+\text{mineval}(B)}.
\]

\end{lemma}

\begin{proof}
See \cite{hornjohnson1985} Problem 17 p410.
\end{proof}

\bigskip

\begin{lemma}
\label{lemma saikkonen lemma} Let $\hat{\Omega}_{n}$ and $\Omega_{n}$ be
invertible (both possibly stochastic) $n\times n$ square matrices whose
dimensions could be growing. Let $T$ be the sample size. For any matrix norm,
suppose that $\|\Omega_{n}^{-1}\|=O_{p}(1)$ and $\|\hat{\Omega}_{n}-\Omega
_{n}\|=O_{p}(a_{n,T})$ for some sequence $a_{n,T}$ with $a_{n,T}\to0$ as
$n\to\infty$, $T\to\infty$ simultaneously (joint asymptotics). Then
$\|\hat{\Omega}_{n}^{-1}-\Omega_{n}^{-1}\|=O_{p}(a_{n,T})$.
\end{lemma}

\begin{proof}
The original proof could be found in \cite{saikkonenlutkepohl1996} Lemma A.2.
\begin{align*}
\|\hat{\Omega}^{-1}_n-\Omega^{-1}_n\|\leq \|\hat{\Omega}_n^{-1}\|\|\Omega_n-\hat{\Omega}_n\|\|\Omega^{-1}_n\|\leq \del[1]{\|\Omega_n^{-1}\|+\|\hat{\Omega}_n^{-1}-\Omega_n^{-1}\|}\|\Omega_n-\hat{\Omega}_n\|\|\Omega^{-1}_n\|.
\end{align*}
Let $v_{n,T}$, $z_{n,T}$ and $x_{n,T}$ denote $\|\Omega_n^{-1}\|$, $\|\hat{\Omega}_n^{-1}-\Omega_n^{-1}\|$ and $\|\Omega_n-\hat{\Omega}_n\|$, respectively. From the preceding equation, we have
\[w_{n,T}:=\frac{z_{n,T}}{(v_{n,T}+z_{n,T})v_{n,T}}\leq x_{n,T}=O_p(a_{n,T})=o_p(1).\]
We now solve for $z_{n,T}$:
\[z_{n,T}=\frac{v_{n,T}^2w_{n,T}}{1-v_{n,T}w_{n,T}}=O_p(a_{n,T}).\]
\end{proof}

\bigskip

\begin{thm}
[\cite{higham2008} (11.1) p269; \cite{diecimorinipapini1996}]%
\label{thm integral form for log}
%Last checked on 20160510; highly trustworthy
%Last checked on 20171215; highly trustworthy
For $A\in\mathbb{C}^{n\times n}$ with no eigenvalues lying on the closed
negative real axis $(-\infty,0]$,
\[
\log A=\int_{0}^{1}(A-I)[t(A-I)+I]^{-1}dt.
\]

\end{thm}

\bigskip

\begin{lemma}
\label{prop matrix inverse perturbation} Let $A, B$ be $n\times n$ real
matrices. Suppose that $A$ is symmetric, positive definite for all $n$ and its
minimum eigenvalue is bounded away from zero by an absolute constant. Assume
$\|A^{-1}B\|_{\ell_{2}}\leq C<1$ for some constant $C$. Then $A+B$ is
invertible for every $n$ and
\[
(A+B)^{-1}=A^{-1}-A^{-1}BA^{-1}+O(\|B\|_{\ell_{2}}^{2}).
\]

\end{lemma}

\begin{proof}
% Proofs last checked on 20160510; highly trustworthy
% Proofs last checked on 20171215; highly trustworthy
We write $A+B=A[I-(-A^{-1}B)]$. Since $\|-A^{-1}B\|_{\ell_2}\leq C<1$, $I-(-A^{-1}B)$ and hence $A+B$ are invertible (\cite{hornjohnson1985} p301). We then can expand
\begin{align*}
(A+B)^{-1}&=\sum_{k=0}^{\infty}(-A^{-1}B)^kA^{-1}= A^{-1}-A^{-1}BA^{-1}+\sum_{k=2}^{\infty}(-A^{-1}B)^kA^{-1}.
\end{align*}
Then
\begin{align*}
&\enVert[4]{\sum_{k=2}^{\infty}(-A^{-1}B)^kA^{-1} }_{\ell_2} \leq \enVert[4]{\sum_{k=2}^{\infty}(-A^{-1}B)^k }_{\ell_2}\|A^{-1}\|_{\ell_2}\leq \sum_{k=2}^{\infty}\left\|(-A^{-1}B)^k \right\|_{\ell_2}\|A^{-1}\|_{\ell_2}\\
&\leq \sum_{k=2}^{\infty}\left\|-A^{-1}B\right\|_{\ell_2}^k\|A^{-1}\|_{\ell_2}=\frac{\left\|A^{-1}B \right\|_{\ell_2}^2\|A^{-1}\|_{\ell_2}}{1-\left\|A^{-1}B \right\|_{\ell_2}}\leq \frac{\|A^{-1}\|_{\ell_2}^3\|B\|_{\ell_2}^2}{1-C},
\end{align*}
where the first and third inequalities are due to the submultiplicative property of a matrix norm, the second inequality is due to the triangular inequality. Since $A$ is real, symmetric, and positive definite with the minimum eigenvalue bounded away from zero by an absolute constant, $\|A^{-1}\|_{\ell_2}=\text{maxeval}(A^{-1})=1/\text{mineval}(A)<D<\infty$ for some absolute constant $D$. Hence the result follows.
\end{proof}

%\begin{proposition}
%\label{prop relation between spectral norm and max row sum}
%For any real $s\times m$ matrix $A$, we have
%\[\|A\|_{\ell_2}\leq \sqrt{s}\|A\|_{\ell_{\infty}}.\]
%\end{proposition}

%\begin{proof}
%\begin{align*}
%&\|A\|_{\ell_2}^2=\text{maxeval}(A^{\intercal}A)\leq \text{tr}(A^{\intercal}A)=\|A\|_F^2=\sum_{i=1}^{s}\sum_{j=1}^{m}|A_{i,j}|^2\leq \sum_{i=1}^{s}\del [3]{\sum_{j=1}^{m}|A_{i,j}|}^2\\
%&\leq s\max_{1\leq i\leq s}\del [3]{\sum_{j=1}^{m}|A_{i,j}|}^2=s\del [3]{\max_{1\leq i\leq s}\sum_{j=1}^{m}|A_{i,j}|}^2=s\|A\|_{\ell_{\infty}}^2.
%\end{align*}
%\end{proof}

\bigskip

\begin{lemma}
\label{lemma l2norm of kronecker product}
%Last checked on 20160511; highly trustworthy
Consider real matrices $A$ ($m\times n$) and $B$ ($p\times q$). Then
\[
\|A\otimes B\|_{\ell_{2}}=\|A\|_{\ell_{2}}\| B\|_{\ell_{2}}.
\]

\end{lemma}

\begin{proof}
\begin{align*}
&\|A\otimes B\|_{\ell_2}=\sqrt{\text{maxeval}[(A\otimes B)^{\intercal}(A\otimes B)]}=\sqrt{\text{maxeval}[(A^{\intercal}\otimes B^{\intercal})(A\otimes B)]}\\
&=\sqrt{\text{maxeval}[A^{\intercal}A\otimes B^{\intercal}B]}=\sqrt{\text{maxeval}[A^{\intercal}A]\text{maxeval}[B^{\intercal}B]}=\|A\|_{\ell_2}\| B\|_{\ell_2},
\end{align*}
where the fourth equality is due to the fact that both $A^{\intercal}A$ and $B^{\intercal}B$ are symmetric, positive semidefinite.
\end{proof}

\bigskip

\begin{lemma}
\label{lemma vandegeer} Let $A$ be a $p\times p$ symmetric matrix and $\hat
{v},v\in\mathbb{R}^{p}$. Then
\[
|\hat{v}^{\intercal}A\hat{v}-v^{\intercal}Av|\leq|\text{maxeval}(A)|\Vert
\hat{v}-v\Vert_{2}^{2}+2\Vert Av\Vert_{2}\Vert\hat{v}-v\Vert_{2}.
\]

\end{lemma}

\begin{proof}
See Lemma 3.1 in the supplementary material of \cite{vandegeerbuhlmannritovdezeure2014}.
\end{proof}

\bigskip

\begin{thm}
[\cite{mcleish1974}]\label{thm mcleish clt} Let $\{X_{n,i}, i=1,...,k_{n}\}$
be a martingale difference array with respect to the triangular array of
$\sigma$-algebras $\{\mathcal{F}_{n,i}, i=0,...,k_{n}\}$ (i.e., $X_{n,i}$ is
$\mathcal{F}_{n,i}$-measurable and $\mathbb{E}[X_{n,i}|\mathcal{F}_{n,i-1}]=0$
almost surely for all $n$ and $i$) satisfying $\mathcal{F}_{n,i-1}%
\subset\mathcal{F}_{n,i}$ for all $n\geq1$. Assume,

\begin{enumerate}
[(i)]

\item $\max_{i\leq k_{n}}|X_{n,i}|$ is uniformly (in $n$) bounded in $L_{2}$ norm,

\item $\max_{i\leq k_{n}}|X_{n,i}|\xrightarrow{p}0$, and

\item $\sum_{i=1}^{k_{n}}X_{n,i}^{2}\xrightarrow{p} 1$.
\end{enumerate}

Then, $S_{n}=\sum_{i=1}^{k_{n}}X_{n,i}\xrightarrow{d} N(0,1)$ as $n\to\infty$.
\end{thm}

\bigskip

\noindent\textbf{Acknowledgements}  We thank the editors, two anonymous referees, Heather Battey, Michael I. Gil', Liudas
Giraitis, Bowei Guo, Xumin He, Chen Huang, Anders Bredahl Kock, Alexei
Onatskiy, Qi-Man Shao, Richard Smith, Chen Wang, Tengyao Wang, Tom Wansbeek,
Jianbin Wu, Qiwei Yao for useful comments. The article has been presented at
University of K\"oln, University of Messina, University of Bonn, University of
Santander, SoFie 2017 in New York, CFE'16 in Sevilla, and Shanghai University
of Finance and Economics, and we are grateful for discussions and comments of
seminar participants. Linton gratefully acknowledges the financial support from the ERC. Tang gratefully acknowledges the financial support from Shanghai 2018 CHEN GUANG Project (project code 18CG06).

\bibliographystyle{ecta}
\bibliography{KronHLT_Biblio}

%Johnstone, I. M. and Onatski, A. Testing in High-dimensional Spiked Models, (2018) Annals of Statistics, forthcoming

%Park, S., Hong, S.Y., and O. Linton (2016). Estimating the quadratic covariation matrix for an asynchronously observed continuous time signal masked by additive noise \textit{Journal of Econometrics }191, 325-347

\end{document}